\documentclass[letterpaper,12pt,oneside,reqno]{amsart}

\usepackage{setspace}

\usepackage[DIV12]{typearea}
\usepackage{adjustbox}
\usepackage[width=.9\textwidth]{caption}
\allowdisplaybreaks%
\numberwithin{equation}{section}%

\usepackage{amssymb}   

\usepackage{amsmath}

\usepackage{mathrsfs}

\usepackage{stmaryrd}

\usepackage{enumitem}

\usepackage[all]{xy}

\usepackage{textcomp}

\usepackage{amsbsy}

\usepackage{datetime}
\renewcommand{\dateseparator}{-}
\renewcommand{\today}{\the\year \dateseparator \twodigit\month
\dateseparator \twodigit\day}

\usepackage
[pdfauthor={Francesco Cavazzani},
 pdftitle={Complete Homogeneous Varieties via Representation Theory},
 bookmarks=false]
{hyperref}

\title{Complete Homogeneous Varieties via Representation Theory} 
\author{Francesco Cavazzani}

\newtheorem{theorem}{Theorem}[section]
\newtheorem{lemma}[theorem]{Lemma}
\newtheorem{quest}[theorem]{Question}
\newtheorem{problem}[theorem]{Problem}
\newtheorem{proposition}[theorem]{Proposition}
\newtheorem{corollary}[theorem]{Corollary}
\newtheorem{conjecture}[theorem]{Conjecture}
\theoremstyle{definition}
\newtheorem{definition}[theorem]{Definition}
\newtheorem{example}[theorem]{Example}
\theoremstyle{remark}
\newtheorem{remark}[theorem]{Remark}

\numberwithin{equation}{section}

\usepackage[titletoc]{appendix} 
\usepackage{cite} 

\usepackage{amsmath}
\usepackage{amssymb}
\usepackage{amsthm} 
\usepackage{units} 
\usepackage{mathrsfs} 
\usepackage[vcentermath]{youngtab}
\usepackage{mathtools} 
\usepackage{graphics}
\usepackage{setspace}

\usepackage{listings} 
\usepackage{hyperref} 

\usepackage[english]{babel} 
\usepackage[utf8]{inputenc}
\usepackage{verbatim} 
\usepackage{soul} 
\usepackage{color}
\usepackage{enumitem} 

\usepackage{tikz}
\usetikzlibrary{matrix,arrows} 

\newcommand{\A}{\mathbb{A}}
\newcommand{\C}{\mathbb{C}}
\newcommand{\G}{\mathbb{G}(1,3)}
\newcommand{\GG}{\mathbb{G}}
\newcommand{\R}{\mathbb{R}}
\newcommand{\Q}{\mathbb{Q}}

\newcommand{\PP}{\mathbb{P}}
\newcommand{\Z}{\mathbb{Z}}
\newcommand{\N}{\mathbb{N}}

\newcommand{\wt}{\widetilde}

\newcommand{\OO}{\mathcal{O}}
\newcommand{\zzz}[3] {z_1^{#1}z_2^{#2}z_3^{#3}}
\newcommand{\dimas}{dim^{\!\!\!\!\!^{as}}(V_\mu)}
\newcommand{\dimasi}{dim^{\!\!\!\!\!^{as}}(V_\mu^H)}
\newcommand{\dimm}{dim^{\!\!\!\!\!^{as}}}
\newcommand{\dimasx}{dim^{\!\!\!\!\!^{as}}(V_x^H)}

\begin{document}

\pagenumbering{roman}

\doublespacing


\thispagestyle{empty}

\vspace*{\fill}

\begin{center}
Complete Homogeneous Varieties via Representation Theory\\
\vspace{0.2in}
A dissertation presented\\
\vspace{0.2in}
by\\
\vspace{0.2in}
Francesco Cavazzani\\
\vspace{0.2in}
to\\
\vspace{0.2in}
The Department of Mathematics\\
\vspace{0.2in}
in partial fulfillment of the requirements\\
for the degree of\\
Doctor of Philosophy\\
in the subject of\\
Mathematics\\
\vspace{0.2in}
Harvard University\\
Cambridge, Massachusetts\\
\vspace{0.2in}
April 2016\\
\end{center}

\vspace*{\fill}

\pagebreak






\thispagestyle{empty} 

\vspace*{\fill}

\begin{center}
\copyright \, 2016 -- Francesco Cavazzani \\
All rights reserved.
\end{center}

\vspace*{\fill}

\pagebreak




\doublespacing


\noindent Dissertation Advisor: Professor Harris \hfill
Francesco Cavazzani

\vspace{0.5in}

\centerline{Complete Homogeneous Varieties via Representation Theory}

\vspace{0.8in}

\centerline{Abstract}

\vspace{0.3in}

Given an algebraic variety $X\subset\PP^N$ with stabilizer $H$, the quotient $PGL_{N+1}/H$ can be interpreted a parameter space for all $PGL_{N+1}$-translates of $X$. We define $X$ to be a \textit{homogeneous variety} if $H$ acts on it transitively, and satisfies a few other properties, such as $H$ being semisimple. Some examples of homogeneous varieties are quadric hypersurfaces, rational normal curves, and Veronese and Segre embeddings. In this case, we construct new compactifications of the parameter spaces $PGL_{N+1}/H$, obtained compactifying $PGL_{N+1}$ to the classically known space of \textit{complete collineations}, and taking the G.I.T. quotient by $H$, and we will call the result space of \textit{complete homogeneous varieties}; this extends the same construction for quadric hypersurfaces in \cite{kannan}. We establish a few properties of these spaces: in particular, we find a formula for the volume of divisors that depends only on the dimension of $H$-invariants in irreducible representations of $SL_{N+1}$. We then develop some tools in invariant theory, combinatorics and spline approximation to calculate such invariants, and carry out the entire calculations for the case of $SL_2$-invariants in irreducible representations of $SL_4$, that gives us explicit values for the volume function in the case of $X$ being a twisted cubic. Afterwards, we focus our attention on the case of twisted cubics, giving a more explicit description of these compactifications, including the relation with the previously known moduli spaces. In the end, we make some conjectures about how the volume function might be used in solving some enumerative problems.

 \pagebreak





\tableofcontents

\pagebreak





\section*{Acknowledgements}

There is plenty of people I wish to thank, at the end of this journey.

The first, and more important thanks goes to my advisor Joe Harris, for his unlimited support, both in math and in life; he is probably the kindest person I know, and the best anyone could hope as a supervisor and guide.

Then, I owe many parts of this thesis to the many conversations that I had about it all around the world. It is very likely that all of this would not have even started, if I had not met Michel Brion in Kloster Heiligkreuztal in 2013. Determinant pieces of the puzzle were also added by Alessio Corti, Noam Elkies, Maksym Fedorchuk and Rahul Pandharipande. Without these people, this thesis would be worth much less. I am also very grateful to Corrado De Concini, a real mentor.

I also want thank the many people that listened to what I had to say, a returned me a lot of useful pieces of advice; in particular, Valeri Alexeev, Carolina Araujo, Alessandro Chiodo, Igor Dolgachev, Bill Fulton, Giovanni Gaiffi, Dick Gross, Joseph Landsberg, Barry Mazur and Ragni Piene, Frank Schreyer, Jerzy Weyman, and the many math buddies Andrea Appel, Nasko Atanasov, Anand Deopurkar, Gabriele Di Cerbo, Simone Di Marino, Andrea Fanelli, Jacopo Gandini, Gijs Heuts, Luca Moci, Anand Patel, Salvatore Stella, Roberto Svaldi and Giulio Tiozzo.

A special thanks goes to Susan Gilbert (happy retirement!), Darryl Zeigler, and the entire Math Department staff. Their everyday work makes the life of a graduate student much easier - especially a first year grad student from Italy with an embarrassingly low English proficiency.

I wish of course to thank also my family and my closest friends; during my entire life, I've never really felt alone, and that's because of them.

In the end, the biggest thanks goes to \textit{you}, standing by my side every single day in the past seven years. I look forward to marry you, and to wake up next to you every single day of my life. \textit{Buon proseguimento}

\pagebreak




\pagenumbering{arabic}

\section{Introduction}
\subsection{History and motivation}
Hermann Schubert's book \textit{Kalk\"ul der abz\"ahlenden Geometrie}, in 1879 (cf. \cite{Kalkul}), was a breakthrough in enumerative geometry. There, he indicated the answers to thousands of questions such as

\begin{quest}\label{Q1}
How many conics in $\PP^2$ are tangent to 5 general conics?
\end{quest}
\begin{quest}\label{Q2}
How many twisted cubics in $\PP^3$ are tangent to 12 general quadric surfaces?
\end{quest}

Question~\ref{Q1} has an interesting storyline - an incorrect answer (7776) was given at first by Steiner in 1848; he was later corrected by Chasles, that gave the correct answer (3264) in 1864. Steiner's miscalculation can be easily understood; in the space $\PP^5$ of conics, the locus of conics that are tangent to a given conic is a hypersurface of degree 6, and intersecting five of them we should get, by Bezout's theorem, $6^5$ that is exactly 7776. However, any five such hypersurfaces fail to satisfy Bezout's theorem hypothesis, that is to intersect transversely. The right answer given by Chasles relies on considering a different space parametrizing conics, the space of \textit{complete conics}, where we endow a conic the information of its dual conic as well. While for conics everything is quite simple, for twisted cubics the situation gets much more complicated; most likely, Schubert did not have in mind a precise moduli space for twisted cubic when finding all the enumerative answers (also called \textit{characteristic numbers}); his calculations work in a quite mysterious way - but he was correct, in every single one of them. It is not by chance that Hilbert dedicated his 15th problem to putting a rigorous foundation to Schubert's work. From a modern prospective, Schubert was taking product of divisors in a parameter space for twisted cubics with 11 boundary divisors, that he called \textit{aspects}; as in the case of complete conics, such divisors can be obtained as degenerations of a ``complete'' twisted cubic, when we endow it with the information of the curve of tangent lines and the dual twisted cubic of osculating planes. Rigorously defined moduli spaces for twisted cubics were eventually found in the 20th century, such as the Hilbert scheme $Hilb_{3m+1}(\PP^3)$ (cf. \cite{pieneschl}, \cite{vainsblowup}, and a variant in \cite{netquadrics} and \cite{netquadrics2}) and the Kontsevich space of stable maps $\overline{\mathcal{M}}_{0,0}(\PP^3,3)$ (cf. \cite{pandha}); using the latter, in particular, with the use of quantum cohomology, people were able to prove sistematically many of Schubert's numbers (cf. \cite{pandha}). In fact, some of these numbers have been proved earlier by ad-hoc arguments in \cite{sketchnumber} and \cite{vainsenchersegre}, including the answer to Question~\ref{Q2}, that is 5819539783680. None of these spaces, though, has the richness of the space Schubert had in mind, or the right amount of symmetry that the space of complete conic has. Tentatives of creating a space more similar to Schubert's one have been done in \cite{alguneid} (where we see the complete twisted cubic as a triple of Chow cycles, later extended in \cite{alguneid2} and \cite{alguneid3}) and in \cite{piene11} and \cite{pienerandom} (where we see the complete twisted cubic as a point in a triple Hilbert scheme). None of these approaches, though, gave any indication about how to recover Schubert calculations rules on intersection products, or to prove them. In addition, none of the current enumerative techniques is easily generalizable to the higher dimensional case, to solve questions like the following.

\begin{quest}\label{Q3}
How many Segre threefolds in $\PP^5$ meet 24 lines?
\end{quest}
\begin{quest}\label{Q4}
How many Veronese surfaces in $\PP^5$ contain 9 general points?
\end{quest}
\begin{quest}\label{Q5}
How many Veronese surfaces in $\PP^5$ are tangent to 27 general hyperplanes?
\end{quest}

Here, Question~\ref{Q3} has been solved using an ad-hoc argument in \cite{vainsenchersegre}, and the answer is 7265560058820, while Question~\ref{Q4} has answer 4, and it follows from the fact that there is a unique elliptic curve of degree 6 through 9 general points in $\PP^5$ (cf. \cite{patel}). We are not aware of anybody ever attempting to answer Question~\ref{Q5}.

This works has the aim to be a step in the direction of filling this void. We will construct moduli spaces of \textit{complete twisted cubics}, and more in general of \textit{complete homogeneous varieties}, in a way that resembles the construction of complete conics and has more ties with Schubert's work. We will then develop an algorithmic way of doing intersection theory on such spaces.

In order to see how we did that, we will now take a step back, and look again at Question~\ref{Q2}. The key to solve it was to use a different compactification of the space of smooth conics, that we can identify with the quotient $PGL_3/PGL_2$. From a different prospective, we can state this problem as the one of finding ($G$-equivariant) compactifications of a quotient $G/H$. In the seminal work \cite{CSV} in 1983, it is proved that in some fortunate cases (that would later be called the cases where $H$ is \textit{spherical}) such a quotient has a very nice and symmetrical compactification, called \textit{wonderful} because of the boundary that is a simple normal crossing divisor. One such example is, in fact, the space of complete conics $\Omega_c$ as compactification of $PGL_3/PGL_2$. One other example is a compactification $\Omega_2$ of $PGL_3$ itself (seen as the quotient of $PGL_3\times PGL_3$ by the diagonal), called classically \textit{complete collineations} (cf. \cite{LaksovCC2}). Many papers have been written about the geometry of these wonderful compactifications, including ways to find intersection theory on them (cf. \cite{CSV}, \cite{brion_enum}). In the relatively under the radar paper \cite{kannan}, in 1999, Kannan showed an even different way of obtaining the space of complete conics; instead of directly compactifying the quotient $PGL_3/PGL_2$, one can compactify $PGL_3$ first to the space $\Omega_2$ of complete collineations, and then take the quotient by $PGL_2$, in the form of a G.I.T. quotient. In other words, in the diagram below, we can get to the space of complete conics $\Omega_c$ going either way.
\begin{center}
\begin{tikzpicture}[description/.style={fill=white,inner sep=2pt}]
\matrix (m) [matrix of math nodes, row sep=3em,
column sep=2.5em, text height=1.5ex, text depth=0.25ex]
{ PGL_3 &  & \Omega_2 \\
PGL_3/PGL_2 &  & \Omega_c\cong\Omega_2/\!\!/PGL_2\ \ \ \  \\ };
\path[->,font=\scriptsize]
(m-1-1) edge node[auto] {$  $} (m-1-3)
(m-1-3) edge node[auto] {$  $} (m-2-3)
(m-1-1) edge node[auto] {$  $} (m-2-1)
(m-2-1) edge node[auto] {$  $} (m-2-3);
\end{tikzpicture}
\end{center}

The space of twisted cubics that we want to compactify, on the other hand, is isomorphic to $PGL_4/PGL_2$, and it is not spherical, so \cite{CSV} does not give us any compactification. Our idea was to use Kannan's method in \cite{kannan} to do that. We first compactified $PGL_4$ to the space of complete collineations $\Omega_3$, took the G.I.T. quotient by $PGL_2$, ando so obtained a compactification of $PGL_4/PGL_2$. Notice that this can be done much more in general, to any quotient $G/H$; in this thesis we restricted ourselves to the cases of $G=PGL_{N+1}$, and $H$ the stabilizer of a nice enough variety. Then, we adapted the intersection theory methods from the theory of wonderful varieties to this case. In this way, intersection theory calculations translate to questions in invariant theory of $H$-invariants in irreducible representations of $G$, that can be dealt with algorithmically. For twisted cubics, this turned out to be a quite different compactification from the more common $Hilb_{3m+1}(\PP^3)$ and $\overline{\mathcal{M}}_{0,0}(\PP^3,3)$, having much more ties with the 11 degeneration Schubert had in mind. Extending the calculation to the case of Veronese varieties, for the quotient $PGL_6/PGL_3$, this work has the potential to give the answer to Question~\ref{Q5}.

\subsection{Summary}
In Section~\ref{sectCC}, we will describe the space of complete collineations $\Omega_N$, obtained taking the closure of $PGL_{N+1}$ in a suitable product of projective spaces. We will then show a few properties of this variety, that has lots of ties with representation theory. In particular, we show that $Pic(\Omega_N)$ can be identified with the lattice $\Lambda_{SL_{N+1}}$ of dominant weights for $SL_{N+1}$.

In Section~\ref{GIT}, we first define in Definition~\ref{homogdef} the notion of homogeneous variety $X\subset\PP^N$, as a projective nondegenerate variety with a transitive action of its stabilizer $H$, and satisfying a few more technical properties; we then prove that quadric hypersurfaces (for $N$ even), rational curves, Veronese surfaces and Segre embeddings (without repeated factors) are homogeneous in the sense of Definition~\ref{homogdef}, and we prove a few properties of the quotient $PGL_{N+1}/H$. Then, in Definition~\ref{maindef}, we define the spaces of complete homogeneous varieties 
$$M_L=\Omega_N/\!\!/_{_{\!\! L}}H=Proj\left(\bigoplus_{k\in\N} H^0(\Omega_N,L^{\otimes k})^H\right)$$
where $L\in Pic(\Omega_N)$ is the linearization for the G.I.T. quotient (different $L$ will lead to different models). In Proposition~\ref{4points}, then, we derive a few geometric properties in for these spaces $M_L$ from the same for $\Omega_N$, such as $Pic(M_L)$ embedding into $\Lambda_{SL_{N+1}}$ as before, and we move to intersection theory. Here, we focus our attention to the \textit{volume} of a divisor $D$ on a space $M_L$, defined as
$$\text{vol}(D)=s!\cdot\limsup_{k\to\infty}\frac{h^0(M_L,\OO(kD))}{k^{s}}.$$
Splitting the spaces of sections in $SL_{N+1}$-isotypical components, we get the first main theorem of this work, Theorem~\ref{main}.

\noindent\textbf{Theorem~\ref{main}.} 
\textit{Let $D$ be a Cartier divisor on $M_L$, such that $\OO(D)$ corresponds to the weight $\lambda$ of $SL_{N+1}$. We have then
$$\text{vol}(D)=\frac{s!}{N+1}\cdot\int_{\mathscr{P}_\lambda}\dimas\dimasi d\mu.$$}

Here, the functions $\dimas$ and $\dimasi$ are asymptotic values of the dimension of (respectively) the vector spaces $V_\mu$ and $V_{\mu}^H$, where $V_{\mu}$ is the irreducible representation of $SL_{N+1}$ of highest weight $\mu$, and $\mathscr{P}_\lambda$ is a polytope in the Weyl chamber of $SL_{N+1}$ depending on $\lambda$. Of these three ingredients for this formula, $\dimas$ and $\mathscr{P}_\lambda$ are very easy to compute, while $\dimasi$ is much harder.

In Section~\ref{invariant_theory}, we develop some invariant theory tools to find a closed formula for the generating function
$$\Xi_H^{SL_{N+1}}(z)=\sum_{\lambda\in\Lambda^+}dim(V_\lambda^H) z^\lambda.$$
We find such a closed formula in the second main theorem of this work, Theorem~\ref{mess}, in the twisted cubic case, for $SL_2$ invariants in irreducible representations of $SL_4$. The proof of this fact is an application of \cite{brion_residue}.

In Section~\ref{asymptotics}, we use the theory of splines to approximate and find the asymptotic value for the coefficients of $\Xi_H^{SL_{N+1}}$, that ultimately give us explicit values for the function $\dimasi$ in the twisted cubic cases, in Corollary~\ref{finalformula}.

In Section~\ref{TC}, we focus our attention to the spaces $M_L$ of complete twisted cubics; we describe explicitly the boundary components and their counterparts in Schubert's 11 degenerations, the different models $M_L$ as $L$ varies, and we calculate the volume explicitly in a few cases. Then, we state some conjectures and speculations about some possible relations between the volume and some enumerative questions, and some other phenomena that we have observed. 

\section{Complete collineations}\label{sectCC}

\subsection{Definition and first properties}

Complete collineations will be the main object to start from, to construct our spaces of complete homogeneous varieties. For a rather complete report on the wide history of the subject, we suggest \cite{LaksovCC2}. We will now describe its definition as wonderful compactification, following mainly \cite{CSV}. 

Everything will happen over the field $\mathbb{C}$ of complex numbers; as main reference, we refer to \cite{alggeo} for algebraic geometry and to \cite{reprthy} for representation theory. From now on, we will indicate by $G$ the group $SL_{N+1}$; we will also use the notation $G_a$ for its adjoint form $PGL_{N+1}$. We will denote by $B$ a Borel subgroup of $G$, $U$ its unipotent radical, $B^{-}$ its opposite, $T$ its Cartan subgroup, by $\Lambda$ its weight lattice, and by $\Lambda_{\R}$ the vector space it lies in (that we can think of as the Lie algebra of $T$). Subsequent to the choice of $B$, we will also denote by $\mathcal{W}\subset\Lambda_{\R}$ the Weyl chamber, by $\Lambda^+\subset\Lambda$ the monoid of dominant weight, by $\omega_1,\ldots,\omega_n\in\Lambda^+$ the fundamental weights and by $\Phi=\{\alpha_1,\ldots,\alpha_n\}\subset\Lambda$ the set of simple roots (ordered in such a way $\alpha_i$ is orthogonal to $\omega_j$ whenever $i\neq j$), and by $\Lambda_a\subset \Lambda$ the sublattice generated by the simple roots, that is also the weight lattice for $G_a$. Given a dominant weight $\lambda\in\Lambda^+$, we will denote by $V_\lambda$ the irreducible representation of $G$ with highest weight $\lambda$, and by $\mathcal{P}_\lambda$ and $P_\lambda$ the sets
$$\mathscr{P_\lambda}=\{\lambda-\sum c_i\alpha_i\mid c_i\in\R^{\geq0}\}\cap\mathcal{W}$$
$$P_\lambda=\{\lambda-\sum c_i\alpha_i\mid c_i\in\N\}\cap\mathcal{W}.$$
\begin{definition}
Consider the $G\times G$ variety
$$G\times G\curvearrowright \prod_{i=1}^N\PP(End(V_{\omega_i})))$$
where the action is given by
$$(g,g')\cdot (\phi_1,\ldots,\phi_N)=(g^{-1}\circ\phi_1\circ g',\ldots,g^{-1}\circ\phi_N\circ g').$$
We call the space of \textit{complete collineations} the closure $\Omega_N$ of the $G\times G$ orbit of the element
$$\mathbb{I}=(Id_{V_{\omega_1}},\ldots,Id_{V_{\omega_N}})\in \prod_{i=1}^N\PP(End(V_{\omega_i})))$$
\end{definition}

\noindent We will list now some properties.

\begin{itemize}[leftmargin=*]
\item[i)] The stabilizer of $\mathbb{I}$ is the subgroup generated by the centers $Z(G)$ of the two copies of $G$ and the diagonal, and hence the orbit of $\mathbb{I}$ will be isomorphic to $G_a$. The space $\Omega_N$ contains then an open dense subset $U$ isomorphic to $G_a$.
\item[ii)] $\Omega_{N}$ is smooth of dimension $(N+1)^2-1$, and inherits an action of $G\times G$, for which it has finitely many orbits. Moreover, the complement $\Omega_{N}\setminus U$ is the union of $N$ irreducible smooth divisors $\Delta_1,\ldots,\Delta_N$ intersecting transversely; this is why it is called ``wonderful compactification''. The $G\times G$ orbits besides $U$ are just intersections of boundary divisors
$$\Delta_I=\bigcap_{i\in I}\Delta_i\quad \forall I\subseteq\{1,2,\ldots,N\}.$$
\item[iii)] (cf.~\cite{CSV}, Theorem~7.6) The Picard group $Pic(\Omega_{N})$ is freely generated by the pullbacks $\eta_1,\ldots,\eta_N$ of the hyperplane classes of the factors $\PP(End(V_{\omega_i})))$ for $i=1,\ldots,N$, and hence it can be naturally identified with the weight lattice $\Lambda$, by associating $\eta_i$ to $\omega_i$. Under this identification, the classes $\delta_i$ of the boundary divisors $\Delta_i$ go to the simple roots $\alpha_i$. The divisor classes $\eta_1,\ldots,\eta_N$ generate also the nef cone (cf. \cite{brion_enum}).
\item[iv)] (cf.~\cite{CSV}, Theorem~8.3) Given a line bundle $L_\lambda$ corresponding to a weight $\lambda\in\Lambda$, the group $G\times G$ acts on $H^0(\Omega_N,L_\lambda)$, and we have
\begin{equation}\label{sectionsomega}
H^0(\Omega_N,L_\lambda)\cong\bigoplus_{\mu\in P_\lambda}V_{\mu}\otimes V_{\mu}^*
\end{equation}
as representation of $G\times G$ (that acts on the summands coordinatewise). 
\end{itemize}

\begin{remark}\label{vanishingboundary}
It is possible to give the different summands in (\ref{sectionsomega}) a geometric interpretation. If $\lambda-\mu=\sum c_i\alpha_i$, then sections in $V_{\mu}\otimes V_{\mu}^*$ are going to vanish with order exactly $c_i$ along the boundary divisor $\Delta_i$. This follows from the algebraic Peter-Weyl theorem, for which
$$\C[G]\cong\bigoplus_{\mu\in\Lambda^+}V_{\mu}\otimes V_{\mu}^*$$
$$\C[G_a]\cong\bigoplus_{\mu\in\Lambda_a\cap\Lambda^+}V_{\mu}\otimes V_{\mu}^*$$
and extending these functions to $\Omega_N$.
\end{remark}

\subsection{Further geometric properties}

We will now describe some further geometric properties of $\Omega_N$. Let us consider the first factor $\PP(End(V_{\omega_1}))$, and let us consider again the orbit of the identity by the action of $G\times G$. This will be the open dense subset of endomorphisms of full rank $N+1$, and its closure will be the entire $\PP(End(V_{\omega_1}))$; let us consider the stratification by rank
$$\PP(End(V_{\omega_1}))\supset Z_N\supset \ldots \supset Z_1$$
where $Z_i$ consists of endomorphisms of rank $i$ or less. We have then the following theorem.
\begin{theorem}[cf. \cite{VainsCC}, Theorem 1]
The space of complete collineations can be obtained after a sequence of blow ups of $\PP(End(V_{\omega_1}))$ along (proper transforms of) $Z_1,Z_2,\ldots,Z_N$.
\end{theorem}
The exceptional divisors obtained after the blow up of the locus of rank $k$ matrices is exactly the boundary divisors $\Delta_k$ mentioned above. We can then consider elements of $\Omega_N$ as endomorphisms of $V_{\omega_1}$ ``enriched'' whenever the rank is not $N+1$ (or $N$). Considering $G$ as acting on $\PP^N$, we can look at the projection of $\Omega_N$ onto the factors $\PP(End(V_{\omega_i}))$ as the actions of $G$ on the Grassmannian $\mathbb{G}(i,N)$ of $i$-planes in $\PP^N$ (embedded into $\PP(V_{\omega_i})$ through the Pl\"ucker embedding) and its degenerations.

Another very useful property is the following, which we can use to prove that $\Omega_N$ is smooth and the boundary divisors intersect trasversely.

\begin{proposition}\label{affinepatch}[cf. \cite{CSV}, Propositions~2.3 and 2.8]
Let $H_1,\ldots,H_N$ be zeroes of the only $B\times B^-$-invariant sections in the classes $\eta_1,\ldots,\eta_N$. Then the complement
$$A=\Omega_N\setminus (H_1\cap\ldots\cap H_N)$$
is isomorphic to the affine space $\mathbb{A}^{N^2+2N}$, and there is a set of affine coordinates for which the intersections of the boundary divisors $\Delta_i\cap A$ are coordinate hyperplanes.
\end{proposition}
\subsection{The full flag variety}

The intersection of all boundary divisors of $\Omega_N$ (and its only closed $G\times G$-orbit) is isomorphic to a product of two copies of the full flag variety $G/B\times G/B$. This is in some way the ``core'' of $\Omega_N$, and most of the properties of $\Omega_N$ have an equivalent (mostly, simpler) for $G/B$. We will then list a few properties of $G/B$, that will be useful in the future.

\begin{itemize}
\item[i)] $G/B$ is smooth of dimension $N(N+1)/2$, and it lives naturally inside the product of Grassmannians
$$G/B\hookrightarrow\PP^N\times\GG(1,N)\times\GG(2,N)\times\ldots\times\GG(N-1,N)$$
\item[ii)] The Picard group $Pic(G/B)$, is freely generated by the pullbacks $\eta_1,\ldots,\eta_N$ of the hyperplane classes of the factors $\GG(i-1,N)$ for $i=1,\ldots,N$, and hence it can be naturally identified with the weight lattice $\Lambda$, by associating $\eta_i$ to $\omega_i$. The divisor classes $\eta_1,\ldots,\eta_N$ generate also the nef cone.
\item[iii)] Given a line bundle $L_\lambda$ corresponding to a weight $\lambda\in\Lambda$, the group $\wt{G}$ acts on $H^0(G/B,L_\lambda)$, and we have
$$H^0(G/B,L_\lambda)\cong V_{\lambda}$$
\end{itemize}

We will also need the following simple lemma. Given a closed subgroup $H\subset G$, and a representation $V$ of $G$, we will denote by $V^H$ the subspace of $H$-invariant vectors in the restriction of $V$ to $H$.

\begin{lemma}\label{simplelemma}
Let $H$ be a closed subgroup of $G$, and let $\lambda_1,\lambda_2$ be two dominant weights of $G$. Then, if $V_{\lambda_i}^H\neq 0$ for $i=1,2$, we have 
$$dim(V_{\lambda_1+\lambda_2}^H)\geq dim(V_{\lambda_1}^H)+dim(V_{\lambda_2}^H)-1.$$  
 \end{lemma}

\begin{proof}
Let us consider the map
$$\alpha:V_{\lambda_1}^H\times V_{\lambda_2}^H\cong H^0(G/B,L_{\lambda_1})^H\otimes H^0(G/B,L_{\lambda_2})^H\to H^0(G/B,L_{\lambda_1+\lambda_2})^H\cong V_{\lambda_1+\lambda_2}^H$$
Looking at vectors as sections of line bundles on $G/B$, and being $G/B$ irreducible, this map is never 0 on any pure tensor $v\otimes w$ for $v$ and $w$ nonzero.
We can then consider the associated map
$$\beta:\PP(V_{\lambda_1}^H)\otimes\PP(V_{\lambda_2}^H)\to\PP(V_{\lambda_1+\lambda_2}^H)$$ 
obtained mapping $([v],[w])$ to $[\alpha(v\otimes w)]$. In the intersection rings, the pullback by $\beta$ of the hyperplane section $\zeta_{12}$ of $\PP(V_{\lambda_1+\lambda_2}^H)$ pulls back to the sum $\zeta_1+\zeta_2$ of the two hyperplane sections of $\PP(V_{\lambda_1}^H)$ and $\PP(V_{\lambda_1}^H)$ in $A^1(\PP(V_{\lambda_1}^H)\times \PP(V_{\lambda_1}^H))$, because it comes from the bilinear map $\alpha$. For this to be a ring homorphisms, we need to have
$$(\zeta_1+\zeta_2)^{dim(V_{\lambda_1+\lambda_2}^H)}=0 \in A^*(\PP(V_{\lambda_1}^H)\times \PP(V_{\lambda_1}^H))=\Z[\zeta_1,\zeta_2]/(\zeta_1^{dim(V_{\lambda_1}^H)},\zeta_1^{dim(V_{\lambda_2}^H)})$$
but that happens if and only if $dim(V_{\lambda_1+\lambda_2}^H)\geq dim(V_{\lambda_1}^H)+dim(V_{\lambda_2}^H)-1$, as we needed to show.
\end{proof}

The above argument for the tensor product of two vector spaces goes all the way back to H. Hopf, that in \cite{hopf} proved this more in general using cohomology rings.

\section{Complete homogeneous varieties}\label{GIT}

\subsection{Homogeneous varieties}

Let $X$ be a subvariety of $\PP^N$, and let $Stab(X)$ be its stabilizer in $G$. The homogeneous space $G/Stab(X)$ is a parameter space for all possible $G$-translates of the variety $X$. We are interested in the case of $X$ satisfying the four properties below; we will call such an $X$ \textit{homogeneous variety}, even though this more usually means just the second part of property (ii). The wording \textit{homogeneous space} will be used when referring to the variety arising as quotient of any linear group by any closed subgroup.

\begin{definition}\label{homogdef}
We will call $X$ a \textbf{quasihomogeneous variety} if:
\begin{itemize}
\item[(i)] $X$ is a projective reduced subvariety of $\PP^N$, and it is nondegenerate (i.e. it is not contained in a hyperplane of $\PP^N$);
\item[(ii)] the stabilizer $Stab(X)\subset G$ of $X$ acts transitively on $X$;
\item[(iii)] the identity connected component $H$ of $Stab(X)$ is a semisimple group, and $Stab(X)$ is generated by $H$ and $Z(G)$.
\end{itemize}
We will call $X$ \textbf{homogeneous} if it satisfies also
\begin{itemize}
\item[(iv)] the bilateral action of $B\times H$ on $G$ has a point with finite stabilizer in $B\times H$.
\end{itemize}
\end{definition}

This definition might seem overly redundant, but we are going to need all these conditions to apply what comes next in the following sections. There are probably though some conditions that might be relaxed, causing some doable changes in what comes next; for instance, requiring $H$ to be just reductive (instead of semisimple) would need a treatment including also characters of $H$; we don't know though of any new example arising in doing that. Let us show some examples.

\begin{example}\label{quadric}
Let $X$ be a quadric hypersurface in $\PP^N$ for $N\geq2$. When $N$ is even, the stabilizer $Stab(X)$ is generated by the othogonal group $SO_{N+1}$ and of the center $Z(G)$, and hence $X$ will be quasihomogeneous, with $H$ equal to $SO_{N+1}$. When $N$ is odd, on the other hand, the product $SO_{N+1}\cdot Z(G)$ has index 2 in $Stab(X)$, and so $X$ will not be quasihomogeneous because it does not satisfy (iii). Let us show now that if $N$ is even, $X$ is also homogeneous. Consider the multiplication map $B\times H\to G$ (sending $(b,h)$ to $b^{-1}h$). The differential at the origin is injective, because the Lie algebra of $H$ (that is mapped into the antisymmetric $N+1\times N+1$ matrices) and the Lie algebra of $B$ (into the upper triangular traceless matrices) have disjoint images; this proves that the identity in $G$ has orbit of dimension equal to the dimension of $B\times H$, and hence that $X$ is homogeneous. This differential is also surjective, so $B\times H$ will also have a dense orbit. This is in some sense the best we can hope for, and plenty has been said about this case in the past (cf. \cite{CSV}), also in a similar fashion to what will follow in the next Sections (cf. \cite{brion_enum}). This entire work can be intended as an extension of all of this.
\end{example}

\begin{example}\label{rnc}
Let $X$ be a rational (projectively) normal curve of degree $N$ in $\PP^N$, where $N\geq 3$ (the case $N=2$ is covered by the previous example). In this case, the stabilizer $Stab(X)$ will be generated by the center $Z(G)$ and a group $H$, that will be $SL_2$ if $N$ is odd and $PGL_2$ if $N$ is even. The embedding $H\to G$ is given looking at the standard representation of $G$ as the $N$-th symmetric power of the standard representation of $SL_2$. Properties (i), (ii) and (iii) of being quasihomogeneous are again obvious, so let's prove (iv). For a general $g\in G$, its stabilizer in $B\times H$ is the same as the intersection $K=gBg^{-1}\cap H$, the set of all elements of $G$ stabilizing both $X$ and a general full flag of subvarieties $\Gamma_0,\ldots,\Gamma_{N-1}$ in $\PP^N$, where $\Gamma_k$ has dimension $k$. The group $K$ will stabilize also the intersection $\Gamma_{N-1}\cap X$, that is composed by $N$ general points on $X$; so, the restriction map $r:K\to Aut(X)$ to the automorphism group of $X$ can have at most a finite image (because on $X\cong \PP^1$ it has to stabilize a set of $N\geq 3$ points). But $r$ is also injective, because $X$ is nondegenerate, so the group $K$ has to be finite.
\end{example}

\begin{example}\label{Veronese}
We can extend the previous example to every Veronese embedding $X\subset \PP^N$ obtained as the embedding of $\PP^n$ through the entire linear series of degree $d$ (and hence $N=\binom{n+d}{n}-1$), let's suppose $n,d\geq 2$; notice that $X$ is a nondegenerate subvariety of $\PP^N$ of degree $d^n$. As before, it is easy to verify that $X$ is homogeneous and that $H$ will be isomorphic to $SL_{n+1}$ (or a quotient by a subgroup of the center of it). Let's now verify (iv) in a similar way as the previous example; the group $K=gBg^{-1}\cap H $ will stabilize $\Gamma_{N-n}\cap X$, that consists of $d^n$ general points on $X\cong\PP^n$. Notice that $d^n\geq n+2$ for all couples $n,d\geq 2$, so the restriction $r:K\to Aut(X)$ has a finite image again, and as before $K$ must be finite.
\end{example}

\begin{example}\label{multiSegreVeronese}
Even more in general, we can consider the multi-Segre-Veronese embeddings $X$ obtained as the image of
$$\PP^{n_1}\times\ldots\times\PP^{n_k}\xrightarrow{\phi_{\OO(d_1,\ldots,d_k)}}\PP^N \quad\quad N=\binom{n_1+d_1}{d_1}\cdots\binom{n_k+d_k}{d_k}-1.$$
Whenever $(n_i,d_i)\neq (n_j,d_j)\  \forall i\neq j$, the quotient $Stab(X)/Z(G)$ is connected, and $H$ will be isomorphic to (a quotient by a subgroup of the center of) $\bigtimes_1^k SL_{n_i+1}$ and (i), (ii), (iii) are easily verified. Notice that $X$ has degree equal to $$\frac{(\sum n_i)!}{\prod (n_i!)}\prod_1^k d_i^{n_i},$$ so the intersection with the general linear subspace $\Gamma_{N-\sum n_i}$ will be composed by this many points. As before, let's consider the restriction $r$ from $K=gBg^{-1}\cap H$ to $Aut(X)$; for it to have finite image, it is sufficient to have all the restrictions $r_i:K\to Stab(\PP^{n_i})$ having finite image. The numerical condition we need is then just
$$\frac{(\sum n_i)!}{\prod (n_i!)}\prod_1^k d_i^{n_i}\geq max\{n_i\}+2.$$
We can suppose $k\geq2$. Then, after some calculations, the equality above is true whenever $\prod d_i\geq2$, or $k\geq 3$, or $min\{n_i\}\geq 2$, and in all these cases we can conclude these are homogeneous varieties. The only cases for which this doesn't happen are the Segre varieties $\PP^1\times\PP^n$ embedded by the complete $(1,1)$ series (we can suppose $n\geq2$ because the case $n=1$ is covered by Example~\ref{quadric}); in this case, to prove that the restriction $r_2:K\to Aut(\PP^n)$ has a finite image, we need an extra step. In this case, the intersection $\Gamma_{N-n-1}\cap X$ is composed by $n+1=N-n$ points; from easy fact about the geometry of the Segre embedding, the intersection $\Gamma_{N-n}\cap X=C$ will be a rational normal curve of degree $n+1$ in $\Gamma_{N-n}=\Gamma_{n+1}$, and its projection to $\PP^n$ will be a rational normal curve as well (of degree $n$). The action of $K$ on $C$ has to fix the $n+1$ points $\Gamma_{N-n-1}\cap X$, and $n\geq 2$, so the image $K\to Aut(C)$ will have finite image; but $C$ spans the entire $\PP^n$, so $K$ will have finite image in $Aut(\PP^n)$, and hence in $Aut(X)$, that completes the proof.
\end{example}

Other quasihomogeneous varieties are, for example, Grassmannians in their Pl\"ucker embedding, and products of any of the above. We are not sure whether these are homogeneous too (the Grassmannian $\GG(1,3)$ is because of Example~\ref{quadric}). Also, we don't know of any quasihomogeneous varieties that are not homogeneous; we have in fact the following conjecture.

\begin{conjecture}
All quasihomogeneous varieties are homogeneous.
\end{conjecture}

Most likely, proving such conjecture would require a complete classification of quasihomogeneous varieties.

From now on, we will denote by $H$ the identity component of the stabilizer $Stab(X)$ of a homogeneous variety $X$.

\begin{remark}
Our aim is to study the quotient $G/Stab(X)$ as a family of the translates of $X$. In the cases we have left off because $Stab(X)/Z(G)$ was not connected (such as quadric hypersurfaces when $N$ is odd and Segre embeddings with identical components), we can instead consider the quotient $G/[Z(G)\cdot(Stab(X)_0)]$, and everything that follows will apply to that quotient too. Many properties of $G/Stab(X)$ can be then deduced from properties of $G/[Z(G)\cdot(Stab(X)_0)]$, because it is just a finit cover of it. We are not anyways going to focus our attention to any of these cases.
\end{remark}

\subsection{The quotient $G/H$}

Let us focus our attention on $G/Stab(X)$. Notice that it is isomoprhic to $G_a/H_a$ (where $H_a$ is the adjoint form of $H$) because  $$Stab(X)/Z(G)=Z(G)\cdot H/Z(G)\cong H/(Z(G)\cap H),$$
and the center of $H$ is contained into the center of $G$. The quotient $G/Stab(X)$ still has a ``left'' action of $G$. We can then describe its coordinate ring in the following way, that follows from the algebraic Peter-Weyl theorem.

$$\C[G/Stab(X)]=\C[G]^{Stab(X)}=\bigoplus_{\lambda\in\Lambda^+_a}V_\lambda\otimes (V_{\lambda}^*)^{Stab(X)}$$
where the action of $G$ is on the left coordinates; only weights in $\Lambda_a$ will appear, because $Stab(X)$ contains the center $Z(G)$, and those are the only weighs for which we have vectors invariant for $Z(G)$. Because of this, we will instead consider the ring of semiinvariant functions 
$$\C[G]^{(Stab(X))}=\{f\in\C[G] : \exists \chi:Stab(X)\to\C^*\ |\ h\cdot f=\chi(h)f\ \forall h\in Stab(X)\}.$$
This ring can also be interpreted as the Cox ring of $G/Stab(X)$, and this is the main reason we are interested in this ring rather than the previous one. Because of the fact that $H$ is semisimple (and hence it has no nontrivial character) and $Stab(X)/H$ is a finite group, we have
$$\C[G]^{(Stab(X))}=\C[G]^H\cong \C[G/H].$$
We will then focus our attention more on the quotient $G/H$ (that will be a finite cover of $G/Stab(X)$).
We will now follow very closely Section~5 of \cite{timashev_book}.

\begin{definition}
We will denote by $\Lambda^+(G/H)$ the set of dominant weights $\lambda\in\Lambda$ of $G$ for which $(V_{\lambda}^*)^H\neq 0$, and by $\Lambda(G/H)$ its $\Z$-span in $\Lambda$. We will call the \textbf{rank} of $G/H$ the rank $r(G/H)$ of $\Lambda(G/H)$. We will call the \textbf{complexity} of $G/H$ the smallest codimension $c(G/H)$ of a $B$-orbit in $G/H$.
\end{definition}

Notice that $\Lambda^+(G/H)$ is a semigroup by Lemma~\ref{simplelemma}, and it is finitely generated because of Proposition~5.15 on \cite{timashev_book}. The rank and the complexity are very important invariants for the geometry of the homogeneous space $G/H$; in simple words, the rank tells about how many representations of $G$ have $H$-invariants, and the complexity about how large this invariant spaces are. We have in fact the following, that basically follows from Theorem~5.16 of \cite{timashev_book}.

\begin{proposition}\label{growth}
For every $\lambda\in\Lambda^+$, then $dim(V_{k\lambda}^H)=O(k^{c(G/H)})$. If $\lambda$ is in the interior of the cone generated by $\Lambda^+(G/H)$, then the complexity $c(G/H)$ is the smallest real number $c$ such that $dim(V_{k\lambda}^H)=O(k^c)$.
\end{proposition}

\begin{proof}
The only addition to the proof of Theorem~5.16 of \cite{timashev_book} is the fact that for any $\lambda$ in the interior, the bound is sharp; Theorem~5.16 of \cite{timashev_book} only shows \textit{one} such dominant weight, that we will call $\lambda_0$. Given any dominant weight $\lambda$ in the interior of the cone, suppose we can find a nonzero space $V_{h\lambda-\lambda_0}^H$ for a given $h>0$. Then, we would get $dim(V_{h\lambda}^H)\geq dim(V_{\lambda_0}^H)$ by Lemma~\ref{simplelemma}, and equally $dim(V_{kh\lambda}^H)\geq dim(V_{k\lambda_0}^H)$ because $V_{k(h\lambda-\lambda_0)}^H\neq 0$ as well, that gives $dim(V_{k\lambda}^H)=O((k/h)^c)=O(k^c)$. The fact that there is a nonzero $V_{h\lambda-\lambda_0}^H$ just follows from the fact that $\lambda$ lies in the interior of the cone, and the fact that $\Lambda^+(G/H)$ is finitely generated as a semigroup. In fact, we can write $\lambda$ as a combination of the all generators of $\Lambda^+(G/H)$ with positive rational coefficients (if we couldn't, $\lambda$ would have to lie on the boundary); hence, a positive multiple $h\lambda$ of this will have integer coefficients that would dominate those of $\lambda_0$, so that $h\lambda-\lambda_0$ also belong to $\Lambda^+(G/H)$.
\end{proof}

The case of complexity zero is called \textit{spherical}, and much is known in this case (see for instance \cite{brion_enum}, \cite{timashev_book} for a general theory for spherical varieties). Quadric hypersurfaces are examples of homogeneous varieties with complexity zero, as we are about to see. In fact, complexity and rank are easily given by property (iv) in Definition~\ref{homogdef}.

\begin{lemma}\label{complexitylemma}
The rank of $G/H$ is $N$, while its complexity is equal to
$$dim(G)-dim(H)-dim(B)=dim(G/H)-N(N+3)/2$$
\end{lemma} 

\begin{proof}
From property (iv) of Definition~\ref{homogdef}, we have an $B\times H$-orbit on $G$ of dimension $dim(B)+dim(H)$. This implies that there is a $B$-orbit of dimension $dim(B)$ in $G/H$, that is equivalent to the claim about the complexity. About the rank, Proposition~5.6 of \cite{timashev_book} tells us that the rank is the difference between the minimal codimension of an $U$-orbit and the complexity. But there is an $U$-orbit in $G/H$ of dimension $dim(U)$ (because there is for $B$), hence we have
$$r(G/H)=dim(G)-dim(H)-dim(U)-c(G/H)=dim(B)-dim(U)=N$$
as needed.
\end{proof}

So, the semigroup $\Lambda^+(G/H)$ spans (as a group) a finite index subgroup of the weight lattice $\Lambda$ of $G$. The question whether $\Lambda^+(G/H)$ could span (as a semigroup) the entire $\Lambda^+$ up to a finite index is an interesting one (and has consequences in what will follow). Let us give a definition about it.

\begin{definition}\label{specialdef}
We will call $X$ \textbf{special} if it quasihomogeneous as of Definition~\ref{homogdef}, and it satisfies one of the following equivalent conditions:
\begin{itemize}
\item each dominant weight $\lambda$ of $G$ has a positive multiple $k\lambda$ such that $(V_{k\lambda}^*)^H\neq 0$, for $H$ the stabilizer of $X$ in $G$;
\item the action of $H$ on each of $\PP^N,\GG(1,N),\ldots,\GG(N-1,N)$ has an invariant hypersurface.
\end{itemize}
\end{definition}

The equivalence of the two definition follows easily from Lemma~\ref{simplelemma}, since the spaces of sections of line bundles on $\GG(i-1,N)$ are just given by the representations $V_{k\omega_{i}}$.

Unfortunately, not all homogeneous varieties are special, as shown by next example.

\begin{example}
Let $X\subset \PP^5$ be the threefold obtained after the Segre embedding of $\PP^1\times\PP^2$, that is a quasihomogeneous variety as we have see in Example~\ref{multiSegreVeronese}. Then $H\cong SL_2\times SL_3$, and the only two $H$-orbits in $\PP^5$ are $X$ and its complement $\PP^5\setminus X$, so there is not an invariant hypersurface, that means that there is no invariant vector in any $V_{k\omega_1}$ for each $k>0$. To see this, notice that points in $\PP^5$ are points of
$$\PP(H^0(\OO_{\PP^1\times\PP^2}(1,1))^*)\cong\PP(H^0(\OO_{\PP^1}(1))^*\otimes H^0(\OO_{\PP^2}(1))^*)\cong$$
$$\cong\PP(Hom(H^0(\OO_{\PP^1}(1)),H^0(\OO_{\PP^2}(1))^*))$$

whose stratification in $SL_2\times SL_3$-orbits is just given by the rank, that can be either 1 or 2, giving us respectively $X$ and $\PP^5\setminus X$. This argument can be in fact generalized to any Segre embedding of any product $\PP^a\times\PP^b$ for $a\neq b$.
\end{example}

The next lemma gives a big help in proving that homogeneous varieties are special.

\begin{lemma}\label{dimension}
Let $H$ be a reductive group acting on an irreducible projective variety $Y$, such that $dim(H)\leq dim(Y)$. Then $Y$ has an $H$-invariant hypersurface.
\end{lemma}

\begin{proof}
Let us split the situation in two cases: the case of where there is a dense orbit in $Y$, and the case where there isn't. If there is a dense orbit $U$, then we need to have $dim(H)=dim(Y)$, and $U\cong H/F$ where $F$ is a finite subgroup, stabilizing a point of $U$; notice now that $H$ has torsion Picard group, and so will $U$, and hence a simple intersection theory argument shows that its complement of $U$ in $Y$ has to have codimension 1. In case the smaller codimension for $H$-orbits is $d$, we can then just consider a general subvariety of $W\subset Y$ of dimension $d-1$, and consider the closure of its orbit $\overline{H\cdot W}$ to get an invariant hypersurface.
\end{proof}

In this way, we can prove that some homogeneous varieties are special.

\begin{corollary}\label{specialhomog}
The Veronese varieties $\nu_d(\PP^n)\subset \PP^{\binom{n+d}{d}-1}$ are special in the sense of Definition~\ref{specialdef} for $n=1,2$ and for any $d\geq 3$.
\end{corollary}

\begin{proof}
We have shown that Veronese varietes are homogeneous in Example~\ref{Veronese}, with $H$ being (possibily a quotient by a subgroup of the center of) $SL_{n+1}$.
For the case $d\geq 3$ and any $n\geq 1$, we have
$$dim(\GG(i,\binom{n+d}{d}-1))\geq \binom{n+d}{d}-1\geq \binom{n+3}{3}-1=$$
$$=(n+1)\frac{n^2+5n+6}{6}-1\geq (n+1)\frac{6n+6}{6}-1=(n+1)^2-1=dim SL_{n+1}$$
so that we can apply again Lemma~\ref{dimension} and get the invariant hypersurfaces.

The only two remaining cases are $(n,d)=(1,2)$ and $(n,d)=(2,2)$. In the former, $X$ is a smooth conic that is already an hypersurface, and the dual conic is the hypersurface we look for in $\GG(1,2)=\PP^{2*}$. If $(n,d)=(2,2)$, $X$ is the Veronese surface in $\PP^5$, and we need to find $SL_3$-invariant hypersurfaces in all Grassmannians $\GG(i,5)$ for $i=0,\ldots,4$; notice that we can apply Lemma~\ref{dimension} in all cases besides $\PP^5$ and $\GG(4,5)$, but in these cases we can consider the secant variety to $X$, that is well known to be an hypersurface in $\PP^5$, and the set of all hyperplanes that are tangent to $X$, that is an hypersurface in $\GG(4,5)$; both of these are clearly $SL_3$-invariant.
\end{proof}

For the remaining cases of $d=2$ and $n\geq 3$, it would again be enough to find invariant hypersurfaces in $\PP^N$ and $\GG(N-1,N)$, because all other cases are covered by Lemma~\ref{dimension}.

Other examples of homogeneous varieties that are special are quadric hypersurfaces (cf. \cite{CSV}),.

\subsection{Compactification(s)}

Let us consider now a homogeneous variety $X$, with identity component of its stabilizer $H$. We will now construct compactifications of the quotient $G/Stab(X)\cong G_a/H_a$ taking the G.I.T. quotient of $\Omega_N$ by $H$. The reference we will use for G.I.T. will be \cite{newstead}. We will consider the action of $H$ on $\Omega_N$ obtained restricting the action of $G\times G$ to the subgroup $\{0\}\times H$. By \cite{KKLV}, Theorem~2.4, for any choice of a line bundle $L$ on $\Omega_N$ there is a power of it that is $H$-linearizable (because a power of $L$ will be linearizable for the universal cover of $H$), and in a unique way because of $H$ does not have nontrivial characters. We are ready now to give our main definition. 

\begin{definition}\label{maindef}
Given a ($H$-linearized) line bundle $L$ on $\Omega_N$, we will call space of \textbf{complete homogeneous varieties} the G.I.T. quotient
$$M_L=\Omega_N/\!\!/_{_{\!\! L}}H=Proj\left(\bigoplus_{k\in\N} H^0(\Omega_N,L^{\otimes k})^H\right)$$
\end{definition}

The notation including the linearization $L$ is unfortunately necessary, because different linearizations could (and will) lead to different spaces. Let us start with a pivotal example. 

\begin{example}\label{pivotal}
Let us take $N=2$ and $X$ to be a smooth conic in $\PP^2$, so that $H$ is isomorphic to $SO_3\cong PGL_2\subset SL_3$. In this case, as we already mentioned, the homogeneous space $G_a/H_a=PGL_3/PGL_2$ is \textit{spherical}, and much is known about its compactifications and their geometric properties.

In this case, for any choice of $L$ ample, we obtain the well known space of complete conics. For $L$ multiple of $\eta_1$, on the other hand, we will get the more usual compactification $\PP^5$ of the space of conics, and for a multiple of $\eta_2$ the space of dual conics $\PP^{5*}$. The idea of looking at complete conics as G.I.T. quotient of the space of complete collineations appears in \cite{kannan}, and in some sense that paper was the very first inspiration for this work. This is also the main reason why we use the word ``complete'' in our definition.
\end{example}

We will now describe some properties of the spaces $M_L$, trying to follow in the next proposition properties i),ii),iii),iv) of $\Omega_N$. We will call $H$-nef a line bundle $L$ on $\Omega_N$ having a multiple $L^{\otimes k}$ with an $H$-invariant section not vanishing on any boundary component; equivalently, through the correspondence of $Pic(\Omega_N)$ with the weight lattice $\Lambda$, and Remark~\ref{vanishingboundary}, $L$ is $H$\textbf{-nef} if the corresponding weight has a multiple in $\Lambda^+(G/H)$; if $X$ is special as in Definition~\ref{specialdef}, all line bundles that are nef on $\Omega_3$ are $H$-nef. We will denote the unstable, semistable and stable loci of $\Omega_N$ by the $L$-linearized $H$-action respectively by $\Omega_N^{us}(L),\Omega_N^{ss}(L),\Omega_N^{s}(L)$, and by $\pi_L$ the quotient map $\pi_L:\Omega_N^{ss}(L)\to M_L$.

\begin{proposition}\label{4points}
Suppose $L$ is $H$-nef. Then
\begin{itemize}
\item[i)] $M_L$ is a projective variety of dimension $d=dim(G)-dim(H)$, with an action of $G$, and with an open dense orbit isomorphic to $G_a/H_a\cong G/Stab(X)$.
\item[ii)] The unstable locus $\Omega_N^{us}(L)$ does not contain any of the boundary divisors. Then, if $L$ is ample on $\Omega_N$, the irreducible components of codimension 1 in $M_L\setminus (U/H)$ are exactly the images by $\pi_L$ of the boundary divisors $\Delta_i\cap\Omega_N^{ss}(L)$ whose general point is stable.
\item[iii)] We have
$$Pic(M_L)\xrightarrow{\pi_L^*} Pic(\Omega_N^{ss}(L))\xleftarrow{restr}Pic(\Omega_N)$$
where the restriction is an isomorphism and $\pi_L^*$ is injective. If furthermore $L$ is ample on $\Omega_N$ and $\Omega_N^{ss}(L)=\Omega_N^{s}(L)$, then $M_L$ has at most finite quotient singularities, and $\pi_L^*$ has a finite index sublattice as image.
\item[iv)] Let $L_\lambda$ be a line bundle on $M_L$ corresponding to a weight $\lambda$ by the previous embedding. We have then
\begin{equation}\label{sectionsML}
H^0(M_L,L_\lambda)=\bigoplus_{\mu\in P_{\lambda}}V_{\mu}\times (V_{\mu}^*)^H
\end{equation}
where the sum is intended as a decomposition in $G$-irreducible representations, where $G$ acts on the left factor of every summand.
\end{itemize}
\end{proposition}

\begin{proof}\textbf{i).}
First of all, notice that we have a section $s_0$ of a sufficient high power $L^{\otimes k}$ that vanishes only along the boundary components $\Delta_i$, and on all of them. This is true because we can always write any dominant weight as a rational positive linear combination of simple roots. The section $s_0$ will also of course be ${H}$-invariant. We can then complete $s_0$ to a basis $\{s_0,s_1,\ldots,s_r\}$ of $H^0(\Omega_N,L^{\otimes k})$, and embed $(\Omega_N)_{s_0}=U\cong G_a$ into $\A^r$ through coordinates $s_i/s_0$. Inside $(\Omega_N)_{s_0}$, all $H$-orbits are closed and disjoint, because they are in $G_a$; this proves directly that all points of $U$ are stable, and hence that the quotient will contain $U/H$ as an open dense subset. Then, the action of $G\times\{0\}$ commutes with the action of $\{0\}\times H$, and hence $M_L$ will have a $G$-action and the quotient map will be $G$-equivariant (and $U/H$ will be an open dense orbit). Notice that i) is true more in general if we drop the hypothesis for $L$ to be $H$-nef, and just be nef on $\Omega_N$; we are not though very interested in this case in general, because of the what comes after in this Proposition.

\noindent\textbf{ii).} To prove the statement, it is sufficient to find, for any nef line bundle $L$, an $H$-invariant section (eventually of a multiple $L^{\otimes k}$) that does not vanish identically on any boundary divisor., and this is given by the definition of $L$ to be $H$-nef, together with property iv) of $\Omega_N$.
For the second part, if the linearization is ample, a subvariety in a G.I.T. quotient is mapped into a subvariety of the same codimension if and only if its general point is stable.

\noindent\textbf{iii).} First of all, the statement on the restriction is obvious because $\Omega_N$ is smooth and the complement $\Omega_N^{us}(L)$ has codimension 2 or more. Then, Kempf descent lemma gives us that $\pi_L^*$ is injective; in particular, the remark after Theorem~2.3 of \cite{kempf_lemma} that says that the preimage of a line bundle on $\Omega_N$ with total space $E$ can only consist of the line bundle on $M_L$ with total space $F=E/\!\!/_{_{\!L}}\wt{H}$, where $\wt{H}$ is the universal cover of $H$ (notice that because $\wt{H}$ is semisimple there is a unique possible action of $\wt{H}$ on $E$, and hence a unique possible such $F$). If $\Omega_N^{ss}(L)=\Omega_N^{s}(L)$, furthermore, Luna's slice \'etale theorem gives us that $M_L$ has finite quotient singularities, and Kempf descent lemma again gives us the finite index.

\noindent\textbf{iv).} To prove this result, we just need to prove that $$H^0(M_L,L_\lambda)\cong H^0(\Omega^{ss}_{N}(L),\pi_L^*L_\lambda)^H=H^0(\Omega_{N},L_\lambda)^H.$$
Consider the total space $E$ of $\pi_L^*L_\lambda$ on $\Omega^{ss}_{N}(L)$ and $F=E/\!\!/_{_{\!L}}\wt{H}$ of $L_\lambda$ on $M_L$. Looking at sections as subvarieties of $E$ and $F$, it is clear that sections of $F$ correspond to $H$-invariant sections of $E$, just through $\pi_L$.
\end{proof}

\begin{remark}\label{remarkvanishing}
Exactly as in Remark~\ref{vanishingboundary}, if $\lambda-\mu=\sum c_i\alpha_i$, then sections in $V_{\mu}\otimes( V_{\mu}^*)^H$ are going to vanish with order exactly $c_i$ along the boundary divisor $E_i$ that is image of the boundary divisor $\Delta_i$ form $\Omega_N$.
\end{remark}

\begin{remark}\label{nefcase}
When talking about linearizations of the action of a group, it is usually assumed for the line bundle to be ample, so in our case any linear combination $\sum a_i\eta_i$ with positive coefficients. The definition can be extended to any line bundle $L$; in this case, though, the main theory of the GIT theory (the existence of the stable and semistable loci, the properties of the quotient map with closed orbits, the Hilbert-Mumford criterion) still apply, but to the birational model $$Proj\bigoplus_{k\geq 0} H^0(\Omega_N,L^{\otimes k})$$
that can be different from $\Omega_N$. In case $L$ is nef (and still $H$-nef), that means $\sum a_i\eta_i$ with nonnegative coefficients, $L$ is also globally generated, so $\Omega_N$ maps to $Proj\bigoplus H^0(\Omega_N,L^k)$, and we can recover in $\Omega_N$ notions of $L$-stable and $L$-semistable points, just pulling back them from this model; they won't anymore satisfy the Hilbert-Mumford criterion, and the map $\pi_L:\Omega_N\to M_L$ will not satisfy the properties about closed orbits anymore. The above Proposition still applies entirely to $M_L$ in this nef case, besides for the second part of ii) and iii). In the former case, the irreducible components of codimension 1 in $M_L\setminus (U/H)$ are in 1-1 correspondence with (and image by $\pi_L$ of) the boundary divisors $\Delta_i\cap\Omega_N^{ss}(L)$ that map birationally to $Proj\bigoplus H^0(\Omega_N,L^k)$ and whose general point is stable. In the latter, the condition $\Omega_N^{ss}(L)=\Omega_N^{s}(L)$ still implies that $M_L$ will have finite quotient singularities, but $\pi^*_L$ will have finite index image only in 
$$Pic(Proj\bigoplus H^0(\Omega_N,L^k))\subset Pic(\Omega_N)$$
\end{remark}

\begin{remark}
It might happen that the condition $\Omega_N^{ss}(L)=\Omega_N^{s}(L)$ never holds, and hence we never have the Picard groups of $M_L$ that is a finite index subgroup of $\Omega_N$. We have the chance though to get any divisor class (up to a positive multiple) of $\Omega_N$ appearing in a G.I.T. model $M_L$. In fact, any G.I.T. quotient has a tautological line bundle that pulls back to (a sufficiently high power of) the linearization itself; to get a class $L_0$ of $\Omega_N$ in the model, it would be sufficient to consider the model $M_{L_0}$.
\end{remark}

The question whether $\Omega_N^{ss}(L)=\Omega_N^{s}(L)$ ever happens is an interesting one; in the case of twisted cubics, that we will see in Section~\ref{TC}, we will answer this question entirely, and try to set up a method that could solve the question more generally.

\subsection{Intersection theory}\label{inters_theory}

The main aim of compactifying $G_a/H_a$ is to be able to use intersection theory to solve the following problem. Let $s$ be the dimension of $G_a/H_a$.

\begin{problem}\label{opendivisors}
Given a divisor $D^\circ\subset G_a/H_a$, in how many points $s$ general $G$-translates of it are going to intersect?
\end{problem}

Many enumerative questions from algebraic geometry can be reduced to this. Given a compactification $G_a/H_a\subset M$, one could try to find on $M$ the top intersection product of the closure $D$ of $D^\circ$
$$D^{s}=\int_M c_1(\OO(D))^{s}$$
and then hope that the linear system $|D|$ is base point free and is generated by $G$-translates of $D$.

A way to partially solve the first of these two issues that might occur, is to consider the \textit{volume} of a divisor, defined as
$$\text{vol}(D)=s!\cdot\limsup_{k\to\infty}\frac{h^0(M,\OO(kD))}{k^{s}}.$$
For $D$ nef on $M$, the two notions agree, $vol(D)=D^{s}$, from the Asymptotic Riemann-Roch formula (see Corollary~1.4.38 of \cite{positivity_1}). For $D$ effective, this notion is (asymptotically) equal to the \textit{moving self-intersection number} $D^{[s]}$, that is defined (when $H^0(\OO(D))$ is large enough) as the intersection of $s$ general divisors linearly equivalent to $D$, outside of the base locus $B(D)$ (see Definition 11.4.10 of \cite{positivity_2}). We have in fact
$$vol(D)=\lim_{k\to\infty}\frac{(kD)^{[s]}}{k^s}$$

So, there is some reason to think that the volume function might be even more powerful in answering enumerative questions, and we will see that in Section~\ref{TC}. Another reason this seems the right choice, is that in Proposition~\ref{4points} we saw that the space of sections of $L_\lambda$ on $M_L$ is in fact independent on $L$. So, the volume function will be independent on the choice of linearization $L$ on $\Omega_N$ and the model $M_L$ obtained. The volume is in some sense then a more natural concept to consider, more intrinsic. Notice that on a model $M_L$, the line bundle $L$ (or the right multiple that is defined) will be ample; another way to rephrase why the volume is an intrinsic object is that the volume gives us the self intersection of $L$ on $M_L$; for a line bundle, the volume function will \textit{choose} for us the G.I.T. model where $L$ has the best behaviour (meaning, where it has no stable base locus), and calculate its self intersection there. In Section~\ref{TC}, in the explicit case of twisted cubics, we will investigate this further, posing some questions about what are some more geometric implications of the volume function.

The last, and more important reason why we will consider the volume, is that we have actually a recipe to calculate it explicitly, using formula (\ref{sectionsML}). We have in fact, as a corollary of Proposition~\ref{4points} iv), for $D$ in the linear series of a line bundle $L_\lambda$,
$$\text{vol}(D)=d!\cdot\limsup_{k\to\infty}\frac{\sum_{\mu\in P_{k\lambda}}dim(V_\mu)dim(V_\mu^H)}{k^{d}},$$
so that we can reduce the volume calculation to a problem about representations of $G$ and $H$-invariants. Going to a limit, we will now define asymptotic dimensions $\dimas$ and $\dimasi$ that will be more useful.

\begin{definition}\label{simas}
The asymptotic dimension $\dimas$ is equal to
$$\dimas=\lim_{k\to\infty}\frac{dim(V_{k\mu})}{k^{N(N+1)/2}}$$
\end{definition} 

\begin{remark}
The dimension of $V_\mu$ is given by the Weyl dimension formula
$$dim(V_\mu)=\prod_{\alpha}\frac{<\rho+\mu,\alpha>}{<\rho,\alpha>}$$
where the sum is over all positive roots, and $\rho$ is half the sum of all positive roots. It is clear then that this is a polynomial function on $\Lambda^+$ of degree equal to the number of positive roots, that is $N(N+1)/2$; hence, the limit in the above definition exists, and the asymptotic dimension $\dimas$ will then just be the homogeneous top degree part of it.
\end{remark}

For $\dimasi$, the situation is much more complicated (in fact, the next two sections will be just devoted to calculations about it). Let us start with a structure proposition about the behavior of $dim(V_{\mu}^H)$. We will need a few definitions in convex geometry before.

\begin{definition}\label{convex}\  
Let $A=[a_1,\ldots,a_m]$ be a list of vectors in a lattice $\Lambda$ of rank $d$, sitting inside and spanning a vector space $V=\Lambda_{\R}$. Suppose their convex hull does not contain 0. We will denote by $C(A)$ the pointed cone generated by them (their $\R^+$-span), and by $\Lambda(A)$ the sublattice of $\Lambda$ that they span over $\Z$.
A vector in $C(A)$ is called $\textbf{regular}$ if it doesn't lie in any cone $C(B)\subsetneq C(A)$ for a sublist $B\subsetneq A$, and it is called \textbf{singular} otherwise. A \textbf{big cell} is the closure $\mathfrak{c}$ of a connected components of the set of regular vectors, and the \textbf{chamber complex} associated with $A$ is the set $\mathcal{C}(A)$ of all big cells.
\end{definition}

The next Proposition will be proved in Section~\ref{asymptotics}, using the theory of vector partition functions.

\vspace{0.2cm}

\noindent\textbf{Proposition~\ref{capitalM}.} 
\textit{Let $M:\Lambda^+\to\N$ be the function associating every dominant weight $\lambda$ of $G$ the number $dim(V_\lambda^H)$. Then there exists a list $A$ of vectors in $\Lambda^+$ such that
\begin{itemize}
\item[(i)] The function $M$ is nonzero only on $\Lambda(A)\cap\Lambda^+$.
\item[(ii)] In each big cell $\mathfrak{c}$ of the chamber complex $\mathcal{C}(A)$, we have
$$M(\lambda) = q^{\mathfrak{c}}(\lambda) + b^{\mathfrak{c}}(\lambda) \quad \forall \lambda\in\Lambda(A)\cap\mathfrak{c}$$
where
\begin{itemize}
\item $q^{\mathfrak{c}}$ is a quasi-polynomial; that means, there is a finite index sublattice $\Lambda(A,\mathfrak{c})$ of $\Lambda(A)$ such that $q^{\mathfrak{c}}$ agrees on a different polynomial on each coset $\lambda+\Lambda(A,\mathfrak{c})$ for $\lambda\in\Lambda(A)$.
\item $b^{\mathfrak{c}}$ is a function that is zero on $\mathfrak{c}\cap\Lambda(A)\setminus(\mathfrak{c}+\beta)\cap\Lambda(A)$, where $\beta$ is an element of $\Gamma$ such that $M(\beta)$ is nonzero.
\end{itemize}
\end{itemize}
}

With the results obtained so far, we can improve the above proposition a little bit.

\begin{proposition}\label{capitalM2}
In Proposition~\ref{capitalM}, the cone $C(A)$ and the lattice $\Lambda(A)$ have the same dimension $s$ as $\Lambda$. Furthermore, in each big cell $\mathfrak{c}$, the top degree part $q_{top}^{\mathfrak{c}}$ of $q^{\mathfrak{c}}$ is in fact a polynomial (that means, it is the same polynomial in any coset $\lambda+\Lambda(A,\mathfrak{c})$) and it has degree $s-N(N+3)/2$.
\end{proposition}

\begin{proof}
The fact that $C(A)$ has full rank comes directly from the fact that the rank of $G/H$ is exactly $N$, as shown in Lemma~\ref{complexitylemma}; in particular, $C(A)$ is exactly the $\R^{>0}$-span of the semigroup $\Lambda^+(G/H)$. 
Suppose now, for a given cell $\mathfrak{c}$, that $q^\mathfrak{c}$ is a quasi polynomial having different top degree parts $p_1$ and $p_2$ on two different cosets $\lambda_1+\Lambda(A,\mathfrak{c})$ and $\lambda_2+\Lambda(A,\mathfrak{c})$. Let us now take elements $\lambda_{21}\in\lambda_2-\lambda_1+\Lambda(A)$ and $\lambda_{12}\in\lambda_1-\lambda_2+\Lambda(A)$ such that $M(\lambda_{12})$ and $M(\lambda_{21})$ are both nonzero. We have then, for each $\lambda$ for which $M(\lambda)\neq 0$, 
$$M(\lambda)\leq M(\lambda+\lambda_{12})-M(\lambda_{12})-1$$
$$M(\lambda)\leq M(\lambda+\lambda_{21})-M(\lambda_{21})-1$$
from Lemma~\ref{simplelemma}. Applying the first to elements of $\lambda\in\lambda_2+\Lambda(A,\mathfrak{c})$ going to infinity, we prove that $p_2\leq p_1$ on the entire $\mathfrak{c}$, applying the second to elements of $\lambda\in\lambda_1+\Lambda(A,\mathfrak{c})$ going to infinity, we prove that $p_1\leq p_2$; we get then $p_1=p_2$. The fact that the degree is equal to the complexity $s-N(N+3)/2$ just follows from Proposition~\ref{growth}.
\end{proof}

\begin{definition}\label{dimasi}
Let $c$ be the index of $\Lambda(A)\cap\Lambda_\alpha$ in $\Lambda_\alpha$, where $A$ comes from Proposition~\ref{capitalM}. For each big cell $\mathfrak{c}$ of the decomposition given in \ref{capitalM}, let again $q_{top}^{\mathfrak{c}}$ be the top degree homogeneous part of $q^{\mathfrak{c}}$, that is a polynomial by Proposition~\ref{capitalM2}  Then, for a dominant weight $\mu$ in a big cell $\mathfrak{c}$ of the decomposition given in \ref{capitalM}, we define the asymptotic dimension as 
$$\dimasi=\frac{1}{c}\limsup_{k\to\infty}\frac{q_{top}^{\mathfrak{c}}(k\mu)}{k^{s-N(N+3)/2}}=q_{top}^{\mathfrak{c}}(\mu)/c$$
\end{definition}

We are now ready to state the main theorem of this section. Notice that we can extend the functions $\dimas$ and $\dimasi$ to any $\mu\in\Lambda_{\R}$, because they are just (piecewise) polynomial functions. We will also equip the vector space $\Lambda_{\R}$ of the Euclidean metric for which the fundamental weights $\omega_i$ form an orthonormal basis (and the simple roots form a parallellotope of volume $N+1$). Finally, remember that 
$$\mathscr{P_\lambda}=\{\lambda-\sum c_i\alpha_i\mid c_i\in\R^{\geq0}\}\cap\mathcal{W}.$$

\begin{theorem}\label{main}
Let $D$ be a Cartier divisor on $M_L$, such that $\OO(D)$ corresponds to the weight $\lambda$. We have then
$$\text{vol}(D)=\frac{s!}{N+1}\cdot\int_{\mathscr{P}_\lambda}\dimas\dimasi d\mu.$$
\end{theorem}

Before proving this theorem (whose proof will still not be complete until we prove Proposition~\ref{capitalM}) let us give one example.

\begin{example}
Let us consider Example~\ref{pivotal}. Let us pick coordinates $x,y$ for the basis $\omega_1,\omega_2$ of $\Lambda_{\R}$. From the Weyl dimension formula, we have
$$dim(V_{x,y})=\frac{(x+1)(y+1)(x+y+1)}{2}\quad \dimm(V_{x,y})=\frac{xy(x+y)}{2},$$
and from \cite{CSV} we have that $dim(V_{x,y}^H)=1$ iff $x$ and $y$ are even integers, so that we have $\dimm(V_{x,y}^H)=1/4$, since in this base the simple roots are the vectors $(2,-1)$ and $(-1,2)$ and we have
$$\Lambda_\alpha=\{(i,j) : 3|i+j\}$$
$$\Lambda(A)=\{(i,j) : 2|i,2|j\}$$
and the index of $\Lambda(A)\cap\Lambda_{\alpha}$ in $\Lambda_{\alpha}$ is 4. 
The volume of a divisor $D$ corresponding to a dominant weight $(a,b)$ then becomes
$$\text{vol}(D)=\frac{5!}{3}\int_{\mathscr{P}_{a,b}}\frac{xy(x+y)}{2}\cdot\frac{1}{4}dxdy=\frac{1}{32}(a^5+10a^4b+40a^3b^2+40a^2b^3+10ab^4+b^5).$$
Considering for instance the divisor $$D=\overline{\{\text{conics containing a point }p\}},$$
using some simple test curves we get $\lambda=(2,0)$, and we get the number 1 of intersection of 5 general translates of $D$ (that gives us the fact that there is one conic through 5 general points); considering the divisor $$D=\overline{\{\text{conics tangent to a conic }C\}},$$ we get $\lambda=(4,4)$, and the number 3264 of conics tangent to 5 general conics. More precise statements about the effects of Theorem~\ref{main} on enumerative geometry will be shown in Section~\ref{TC}.
\end{example}

\begin{remark}
This idea of making the volume function into an integral for enumerative geometry of (compactifications of) homogeneous spaces is not new. It has been developed for spherical varieties (where $B$ acts on $G/H$ with a dense orbit) in \cite{brion_enum}, and in the case of complexity 1 (where $B$ acts on $G/H$ with an orbit of codimension 1) in \cite{timashev_book}. In \cite{timashev_book}, a way is paved to start the calculations in the general case of any complexity; as far as we know, this is the first serious attempt to carry out the computation until the end in situations of higher complexity.
\end{remark}

We are now ready to prove Theorem~\ref{main}.

\begin{proof}
We have
$$\frac{\text{vol}(D)}{s!}=\limsup_{k\to\infty}\frac{\sum_{\mu\in P_{k\lambda}}dim(V_\mu)dim(V_\mu^H)}{k^d}.$$
In each big cell $\mathfrak{c}$ of the decomposition coming from Proposition~\ref{capitalM}, we have now
$$dim(V_\mu)dim(V_\mu^H)=p^{\mathfrak{c}}(\mu)+r^{\mathfrak{c}}(\mu)+s^{\mathfrak{c}}(\mu),$$

\begin{itemize}
\item $p^{\mathfrak{c}}(\mu)=\dimas q_{top}^{\mathfrak{c}}$ is a piecewise polynomial, of degree $N(N+1)/2+d-N(N+3)/2=d-N$.
\item $r^{\mathfrak{c}}(\mu)=(dim(V_\mu)-\dimas)q^{\mathfrak{c}}+dim(V_\mu)(q^{\mathfrak{c}}_{top}-q^{\mathfrak{c}})$ is a quasi polynomial of degree strictly less than $d-N$.
\item $s^{\mathfrak{c}}(\mu)=dim(V_{\mu})b^{\mathfrak{c}}(\mu)$ is still a function that is zero on $\mathfrak{c}\cap\Lambda(A)\setminus(\mathfrak{c}+\gamma^{\mathfrak{c}})\cap\Lambda(A)$, where $\gamma^{\mathfrak{c}}$ is an element of $\Lambda(A)$ such that $M(\gamma^{\mathfrak{c}})$ is nonzero.
\end{itemize}
Let us first show now that the components $s^{\mathfrak{c}}$ and $r^{\mathfrak{c}}$ are negligible. 
Notice that we have, for any $\mu$ such that $dim(V_\mu^H)\neq 0$,
$$dim(V_\mu)dim(V_\mu^H)\leq dim(V_{\mu+\gamma})dim(V_{\mu+\gamma}^H)-dim(V_\gamma)dim(V_\gamma^H)-1$$
following the same argument as in Lemma~\ref{simplelemma}. We also have
$$dim(V_{\mu+\gamma})dim(V_{\mu+\gamma}^H)=p^{\mathfrak{c}}(\mu)+r^{\mathfrak{c}}(\mu)+\underbrace{(p^{\mathfrak{c}}(\mu+\gamma)+r^{\mathfrak{c}}(\mu+\gamma)-p^{\mathfrak{c}}(\mu)-r^{\mathfrak{c}}(\mu))}_{=u^{\mathfrak{c}}(\mu)}$$
where $u^{\mathfrak{c}}(\mu)$ will be a quasipolynomial of degree $<d-N$ as well.
We have then 
$$\limsup_{k\to\infty}\frac{\sum_{\mu\in P_{k\lambda}\cap\mathfrak{c}}dim(V_\mu)dim(V_\mu^H)}{k^d}\leq$$
$$\leq\limsup_{k\to\infty}\frac{\sum_{\mu\in P_{k\lambda}\cap\mathfrak{c}}p^{\mathfrak{c}}(\mu)+r^{\mathfrak{c}}(\mu)+u^{\mathfrak{c}}(\mu)}{k^d}=\limsup_{k\to\infty}\frac{\sum_{\mu\in P_{k\lambda}\cap\mathfrak{c}}p^{\mathfrak{c}}(\mu)}{k^d}$$
where the last step is because $r^{\mathfrak{c}}(\mu)$ and $u^{\mathfrak{c}}(\mu)$ have degree smaller than $p^{\mathfrak{c}}(\mu)$.
Let now $a$ be an integer such that $\gamma^{\mathfrak{c}}+P_{(k-a)\lambda}\subset P_{k\lambda}$ for each $k\geq a$; this exists because $\gamma^{\mathfrak{c}}$ is in the span of the simple roots $\alpha_i$. We have then
$$\limsup_{k\to\infty}\frac{\sum_{\mu\in P_{k\lambda}\cap\mathfrak{c}}dim(V_\mu)dim(V_\mu^H)}{k^d}\geq$$
$$\geq \limsup_{k\to\infty}\frac{\sum_{\mu\in (\gamma^{\mathfrak{c}}+P_{(k-a)\lambda})\cap\mathfrak{c}}dim(V_\mu)dim(V_\mu^H)}{k^d}=$$
$$=\limsup_{k\to\infty}\frac{\sum_{\mu\in P_{(k-a)\lambda}\cap\mathfrak{c}}p^{\mathfrak{c}}(\mu)+r^{\mathfrak{c}}(\mu)+u^{\mathfrak{c}}(\mu)}{k^d}=$$
$$=\limsup_{k\to\infty}\left(\frac{k}{k+a}\right)^d\frac{\sum_{\mu\in P_{k\lambda}\cap\mathfrak{c}}p^{\mathfrak{c}}(\mu)}{k^d}=\limsup_{k\to\infty}\frac{\sum_{\mu\in P_{k\lambda}\cap\mathfrak{c}}p^{\mathfrak{c}}(\mu)}{k^d}.$$
Joining everything together, we get
$$\limsup_{k\to\infty}\frac{\sum_{\mu\in P_{k\lambda}\cap\mathfrak{c}}dim(V_\mu)dim(V_\mu^H)}{k^d}=\limsup_{k\to\infty}\frac{\sum_{\mu\in P_{k\lambda}\cap\mathfrak{c}}p^{\mathfrak{c}}(\mu)}{k^d}.$$
The polynomial $p^{\mathfrak{c}}(\mu)$ is now homogeneous of degree $d-N$.
We have then
$$\limsup_{k\to\infty}\frac{\sum_{\mu\in P_{k\lambda}\cap\mathfrak{c}}p^{\mathfrak{c}}(\mu)}{k^d}=\limsup_{k\to\infty}\frac{\sum_{\mu\in \left(\frac{1}{k}P_{k\lambda}\right)\cap\mathfrak{c}}p^{\mathfrak{c}}(\mu)}{k^{N}}.$$
We can then interpret this as a Riemann sum over the volume $\mathcal{P}_{\lambda}$. When summing over $\frac{1}{k}P_{k\lambda}$, we are splitting $\mathcal{P}_{\lambda}$ in polytopes that are fundamental chambers for the lattice $\frac{1}{k}\Lambda_a\cap\Lambda(A)$ (where $\Lambda(A)$ is again the lattice coming from Proposition~\ref{capitalM}), so they have volume $c\cdot(N+1)/k^N$, where $c$ is as is Definition~\ref{dimasi}. We get then
$$\limsup_{k\to\infty}\frac{\sum_{\mu\in \left(\frac{1}{k}P_{k\lambda}\right)\cap\mathfrak{c}}p^{\mathfrak{c}}(\mu)}{k^{N}}=$$
$$=\frac{1}{N+1}\limsup_{k\to\infty}\left(\sum_{\mu\in \left(\frac{1}{k}P_{k\lambda}\right)\cap\mathfrak{c}}\frac{c(N+1)}{k^N}\dimas \frac{q_{top}^{\mathfrak{c}}}{c}\right)=$$
$$=\frac{1}{N+1}\limsup_{k\to\infty}\left(\sum_{\mu\in \left(\frac{1}{k}P_{k\lambda}\right)\cap\mathfrak{c}}\frac{c(N+1)}{k^N}\dimas\dimasi\right)=$$
$$=\frac{1}{N+1}\int_{\mathcal{P}_{\lambda}\cap\mathfrak{c}}\dimas\dimasi d\mu$$Summing over all big cells $\mathfrak{c}$, we get the result.
\end{proof}

\section{Invariant theory}\label{invariant_theory}

\subsection{Formal series and characters}

This section heavily relies on the use of generating functions and formal power series. Let us give some definitions.

\begin{definition}
Let $\Gamma$ be a monoid, and $R$ be an integral domain. We will denote by $R[z^\Gamma]$ the polynomial ring generated by the monomials $z^\gamma$, that satisfy the formal relations $z^\gamma z^{\gamma'}=z^{\gamma+\gamma'}$. We will also define $R[[z^\Gamma]]$ as the set of power series
$$\sum_\Gamma a_\gamma z^\gamma.$$
\end{definition}

\begin{remark}\label{fractions}
If $\Gamma$ is isomorphic as a monoid to a product of copies of $\N$, then the product on $R[[z^\Gamma]]$ is well-defined, and we can talk about the ring of formal series.
Moreover, if $r\in R$ and $\gamma\in\Gamma$ nonzero, the element $1-rz^\gamma$ has an inverse in $R[[z^\Gamma]]$, given by
$$\frac{1}{1-rz^\gamma}=\sum_{k=0}^{\infty}r^kz^{k\gamma}$$
and we will use the fractional notation to indicate such series.
\end{remark}

\begin{example}
Let $\Lambda_G$ be the weight lattice of a semisimple group $G$, and let $V$ be a representation of $G$; we can then associate to $V$ an element $\chi_V^G(x)\in\Z[x^\Gamma]$, called the character of $V$ (the construction is through the diagonalization of the action of a maximal torus in $G$). The character of the direct sum of two representations is the sum of the two characters, and the character of the tensor product is the product of the characters.
\end{example}

\begin{definition}
Let $\Lambda_G$ be again the weight lattice of a semisimple group $G$, and $\Lambda^+_G$ the monoid of dominant weights. Let us denote by $V_\lambda$ the irreducible representation of $G$ with highest weight $\lambda\in\Lambda^+$. We will call the \textbf{generating function for $G$-characters} the formal series
$$\Xi_G(x,z)=\sum_{\lambda\in\Lambda^+}\chi^G_{V_\lambda}(x)\cdot z^\lambda\in\Z[x^\Lambda][[z^{\Lambda^+}]]$$
in the ring of formal series over $\Lambda^+$ having as coefficients polynomials in $\Z[x^\Lambda]$.
\end{definition}

The next step in this section will be to find a formula for $\Xi_G(x,z)$ in a closed form, that means as a finite formula including fractional terms as in Remark~\ref{fractions}.

\subsection{Generating function for $G$-characters}

\begin{example}\label{sl3char} Let us start finding $\Xi_G(x,z)$ in an example, when $G=SL_3$; this will shed some light on how to do it in general.

The key point is the following formula (where $S^aV$ is the $a$-th symmetric power of a vector space $V$).
\begin{equation} \label{tricksl3}
V_{a\omega_1+b\omega_2}=(S^aV_{\omega_1}\otimes S^bV_{\omega_2})/(S^{a-1}V_{\omega_1}\otimes S^{b-1}V_{\omega_2})
\end{equation}

Using this formula, we easily get
$$\Xi_{G}(x,z)=(1-z^{\omega_1+\omega_2})\sum_{a,b}\chi^G_{S^a V_{\omega_1}}(x)\cdot\chi^G_{S^b V_{\omega_2}}(x)\cdot z^{a\omega_1+b\omega_2}=$$
$$=(1-z^{\omega_1+\omega_2})\left(\sum_{a}\chi^G_{S^a V_{\omega_1}}(x)z^{a\omega_1}\right)\left(\sum_{b}\chi^G_{S^b V_{\omega_2}}(x)\cdot z^{b\omega_2}\right)$$

Expressions of the form $\sum_{a}\chi_G(Sym^a V_{\omega_1})(x)z^{a\omega_1}$ are now easy to express in closed form, using the fact that
\begin{equation} \label{tricksym}
\sum_{a}\chi^G_{S^a V}(x)z^{a\omega_i}=\prod_{\gamma\text{ weight of }V}\frac{1}{1-x^\gamma z^{\omega_i}}.
\end{equation}

We get then

$$\Xi_{SL_3}(x,z)=(1-z^{\omega_1+\omega_2})\prod_{\gamma\text{ weight of }V_{\omega_1}}\frac{1}{1-x^\gamma z^{\omega_1}}\prod_{\gamma\text{ weight of }V_{\omega_2}}\frac{1}{1-x^\gamma z^{\omega_2}}.$$

\end{example}

For bigger $G$, unfortunately, we don't have anymore a formula as simple as (\ref{tricksl3}). In order to be able to generalize this calculation, we need to look at our problem in a different, more geometric, way. Let us now give a few further definitions.

\begin{definition}
Let $G$ be a semisimple group, and $\Gamma$ a monoid; we will denote by $\Gamma$\textbf{-graded} $G$\textbf{-module} a (possibly infinite dimensional) vector space $V$ with the following properties:
\begin{itemize}
\item there is a $\Gamma$-grading $$V=\bigoplus_{\gamma\in\Gamma} V_\gamma$$ where all $V_\gamma$ are finite dimensional;
\item there is a linear $G$-action respecting the grading.
\end{itemize}
If furthermore there is a product on $V$ such that $V_\gamma\cdot V_{\gamma'}\subseteq V_{\gamma+\gamma'}$, then we will call $V$ a $\Gamma$\textbf{-graded} $G$\textbf{-algebra}. 
\end{definition}

\begin{definition}
Let $V$ be a $\Gamma$\textbf{-graded} $G$\textbf{-module}, and let $\Lambda_G$ be the weight lattice of $G$. The \textbf{character} of $V$ will be the formal series
$$\chi^G_V(x,z)=\sum_{\gamma\in\Gamma}\chi_G(V_\gamma)(x)\cdot z^\gamma\in \Z[\Lambda_G][[\Gamma]].$$
\end{definition}

What follows now comes from \cite{private}. References for it can be found on \cite{weymanbook}, Chapter 5, on \cite{eisenbudbook}, Chapter 21 and on \cite{pany}.

Let $V_1,\ldots,V_k$ be $k$ representations of $G$; we can then consider the $\N^k$-graded $G$-module given by $$R=\C[\bigoplus_{i=1}^kV_i]$$
with the obvious multigrading such that the degree $(a_1,a_2,\ldots,a_k)$ piece is $\bigotimes Sym^{a_i}V_i$; using the notation $z^{(a_1,a_2,\ldots,a_k)}=z_1^{a_1}z_2^{a_2}\cdots z_k^{a_k}$, we get that in this case
$$\chi^G_R(x,z)=\prod_{i=1}^k\left(\prod_{\gamma\text{ weight of }V_{i}}\frac{1}{1-x^\gamma z_i}\right)$$
because again of Formula~(\ref{tricksym}) in Example~\ref{sl3char}.
The ring $R$ can be thought as the ring of multihomogeneous polynomials on the (affine) vector space $\bigoplus_{i=1}^kV_i$. Consider now a $G$-invariant multihomogeneous subvariety $X\subset\bigoplus_{i=1}^kV_i$, that can be thought as the multicone over an invariant projective subvariety $\PP(X)\subset\bigtimes_1^k\PP V_i$. Let us now consider the ring $\C[X]$ of regular functions on it, that will be of course graded as well; notice that its $(a_1,a_2,\ldots,a_k)$-degree piece can be identified with the space of sections
$$H^0(\PP(X),\OO_{\bigtimes_1^k\PP V_i}(a_1,a_2,\ldots,a_k)|_{\PP(X)}).$$
$\C[X]$ will also be a (cyclic) $R$-module; we can consider the minimal resolution of $\C[X]$ in free $R$-modules
\begin{equation}\label{resolution}
0\to F_r\to\ldots\to F_2\to F_1\to R\to\C[X]\cong R/\mathcal{I}(X)\to0.
\end{equation}
Every module $F_i$ can be thought as isomorphic to $M_i\otimes_{\C}R$, where $M_i$ is a finite dimensional (as a vector space) $\N^k$-graded $G$-module, whose basis is a set of generators for $F_i$ as $R$-module. The lenght $r$ of the resolution will be equal to the codimension of $X$ in $\bigoplus_{i=1}^kV_i$. We get then the following formula for the character of $\C[X]$
$$\chi^G_{\C[X]}(x,z)=\chi^G_R\cdot\left(1-\chi_G(M_1)(x,z)+\ldots+(-1)^k\chi_G(M_k)(x,z)\right).$$
If furthermore $\C[X]$ is a Gorenstein ring, the dualizing module of $\C[X]$ is isomorphic to $\C[X]$ (up to a shift given by the multidegree of the relative canonical module $\Omega$), and hence the resolution is going to have the symmetry relations
\begin{equation}\label{duality1}
M_{r-j}=M_j^*\otimes M_r \quad \forall j=1,\ldots,r-1
\end{equation}
\begin{equation}\label{duality2}
M_r=(M_0)_{-\Omega}
\end{equation}
where the dual of a module is intended to have the opposite multigrading (so, possibly becoming negative), and where $(M_0)_{-\Omega}$ is the module $M_0$ shifted in multidegree by $-\Omega$.

We are interested in applying this machinery to the specific case of the full flag variety $G/B$, that naturally lives inside $\bigtimes_1^n\PP V_{\omega_i}$. We will consider as $X$ the multicone $\widehat{G/B}\subset\bigoplus V_{\omega_i}$. Notice that the $\N^k$ of the grading can be identified with the set of dominant weights $\Lambda^+$. The rational functions on the multicone $\widehat{G/B}$ are the sections on $G/B$ of the restrictions of line bundles on $\bigtimes_1^n\PP V_{\omega_i}$. We then have
$$\chi^G_{\C[\widehat{G/B}]}(x,z)=\sum_{\lambda}\chi^G_{V_\lambda}(x)\cdot z^{\lambda}=\Xi_G(x,z)$$
Given that $\C[\widehat{G/H}]$ is Gorenstein (from \cite{pany}), this gives us an explicit way of calculating $\Xi_G(x,z)$ as a polynomial in $x$ and $z$ times $\chi^G_{R}(x,z)$, that is equal to
$$\chi^G_R(x,z)=\prod_{i=1}^n\left(\prod_{\gamma\text{ weight of }V_{\omega_i}}\frac{1}{1-x^\gamma z_i}\right).$$
We will conclude this section with two examples.

\begin{example}
Let us work again on Example~\ref{sl3char}. In that case, the flag variety $G/B$ is a $(1,1)$ hypersurface in $\bigtimes_1^n\PP V_{\omega_i}$ that is just the product of two copies of $\PP^2$. We have then, by adjunction,
$\Omega=(-1,-1)$, and a resolution that is forced to be
$$0\to M_1\otimes_{\C}R=(V_0)_{1,1}\otimes_{\C}R \to R\to\C[\widehat{G/B}]\cong R/\mathcal{I}(\widehat{G/B})\to0$$
that gives us again
$$\Xi_G(x,z)=(1-z_1z_2)\prod_{\gamma\text{ weight of }V_{\omega_1}}\frac{1}{1-x^\gamma z_1}\prod_{\gamma\text{ weight of }V_{\omega_2}}\frac{1}{1-x^\gamma z_2}.$$
\end{example}

The previous case is pretty lucky because $G/B$ is an hypersurface, that immediately gives a very simple resolution. Let us now analyze a slightly more complicated case.

\begin{example}\label{resolutionSL4}
Let now $G$ be $SL_4$; we will denote the irreducible representation $V_{a\omega_1+b\omega_2+c\omega_3}$ by $V_{a,b,c}$ (or $V_{abc}$), and any vector $V$ space lifted by a degree $(d,e,f)$ by $V_{(d,e,f)}$ (or $V_{(def)}$). The flag variety $G/B$ in this case has codimension 5 in $\bigtimes_1^n\PP V_{\omega_i}$, so we should expect a resolution of lenght 5. About $\Omega$, it is well known that the canonical bundle on $G/B$ is the restriction of the line bundle $(-2,-2,-2)$ from $\bigtimes_1^n\PP V_{\omega_i}$, while the canonical of $\bigtimes_1^n\PP V_{\omega_i}$ is $(-4,-6,-4)$; $\Omega$ will be then $(-2,-4,-2)$.
We can the construct the resolution in 2 different ways; either using some software capable of finding syzygies (plus some understanding of what $G$-representation are the graded pieces of the modules $M_i$), or a more direct way, by inclusion/exclusion, that is the one we will follow.
It is clear that the first piece of the resolution is always $R$, that means the module $V_{0,0,0}$ will be just the trivial representation $M_0$ at degree 0. Let us now consider the following piece
$$M_1\otimes_{\C}R\xrightarrow{f} R\to R/\mathcal{I}(\widehat{G/B})\to 0$$
We want to find the minimal $M_1$ such that $f$ is surjective. Let us consider now the restriction of this sequence in small degrees.
$$\begin{matrix}
degree &\quad M_1\otimes_\C R &\to & R &\to& R/\mathcal{I}&\to& 0 \\
(0,0,0)&\quad ? &\to& V_{000} &\xrightarrow{\sim}& V_{000} &\to& 0 \\
(1,0,0)&\quad ? &\to& V_{100} &\xrightarrow{\sim}& V_{100} &\to& 0 \\
(0,1,0)&\quad ? &\to& V_{010} &\xrightarrow{\sim}& V_{010} &\to& 0 \\
(0,0,1)&\quad ? &\to& V_{001} &\xrightarrow{\sim}& V_{001} &\to& 0 \\
(2,0,0)&\quad ? &\to& S^2V_{100}=V_{200} &\xrightarrow{\sim}& V_{200} &\to& 0 \\
(0,2,0)&\quad ? &\to& S^2V_{010}=V_{020}\oplus V_{000} &\to& V_{020} &\to& 0 \\
(2,0,0)&\quad ? &\to& S^2V_{001}=V_{002} &\xrightarrow{\sim}& V_{002} &\to& 0 \\
(1,1,0)&\quad ? &\to& V_{100}\otimes V_{010} = V_{110}\oplus V_{001} &\to& V_{110} &\to& 0 \\
(1,0,1)&\quad ? &\to& V_{100}\otimes V_{001} = V_{101}\oplus V_{000} &\to& V_{101} &\to& 0 \\
(1,1,0)&\quad ? &\to& V_{010}\otimes V_{001} = V_{011}\oplus V_{100} &\to& V_{011} &\to& 0 
\end{matrix}$$
It is clear that in order to be surjective, $M_1$ has to contain at least the module $$(V_{0,0,0})_{(0,2,0)}\oplus(V_{0,0,1})_{(1,1,0)}\oplus(V_{0,0,0})_{(1,0,1)}\oplus(V_{1,0,0})_{(0,1,1)}.$$
It is in fact true that if $M_1$ equals the module above, then the map $M_1\otimes_\C R \to R$ is indeed surjective. To prove this, notice that it is enough to check it only up to degree $(2,4,2)$, because from Formula~(\ref{duality1}) no $M_i$ can contain anything of degree higher than $(2,4,2)$ (and by higher we mean having any coordinate larger).
In the same way, we can construct the module $M_2$; modules $M_3$ and $M_4$ will then follow from Formula~(\ref{duality1}). We get then
$$M_0=(V_{000})_{(000)}$$
$$M_1=(V_{001})_{(110)}\oplus(V_{100})_{(011)}\oplus(V_{000})_{(101)}\oplus(V_{000})_{(020)}$$
$$M_2=(V_{000})_{(210)}\oplus(V_{000})_{(012)}\oplus(V_{100})_{(120)}\oplus(V_{001})_{(021)}\oplus(V_{010})_{(111)}$$
$$M_3=(V_{000})_{(032)}\oplus(V_{000})_{(230)}\oplus(V_{001})_{(122)}\oplus(V_{100})_{(221)}\oplus(V_{010})_{(131)}$$
$$M_4=(V_{100})_{(132)}\oplus(V_{001})_{(231)}\oplus(V_{000})_{(141)}\oplus(V_{000})_{(222)}$$
$$M_5=(V_{000})_{(242)}.$$
This gives us in the end a closed formula for $\Xi_G(x,z)$ in the case of $SL_4$.
\end{example}

It should be noted that a different approach to this problem, despite probably leading to the same (or probably more complicated) calculations, using Gel'fand-Tsetlin theory, is carried out in \cite{erass}.

\subsection{Generating function for $H$-invariants}

We will now consider a closed subgroup $H\subset G$; we will assume it is also semisimple. From now on, to indicate the weight lattice $\Lambda$, the set of dominant weights $\Lambda^+$, the set of all (positive, simple) roots $\Phi$ (respectively $\Phi^+,\Phi^s$) we will use a subscript to make explicit the group we are talking about. 

The choice of a Borel and a Cartan subgroup for $G$ will be made in such a way it extends a choice of Borel and Cartan for $H$ (that is always possible, by Lie's theorem); in this way, we have a well defined restriction $\pi:\Lambda_G\to\Lambda_H$ that maps $\Lambda^+_G$ onto $\Lambda^+_H$. Let $\pi_x$ be the map
\begin{align*}
\pi_x:\Z[x^{\Lambda_G}][[z^{\Lambda_G^+}]]&\to \Z[x^{\Lambda_H}][[z^{\Lambda_G^+}]] \\
x^\lambda z^\mu &\mapsto x^{\pi(\lambda)}z^\mu
\end{align*}

The aim of this subsection is to find and express in a closed form the following power series.

$$\Xi_H^G(z)=\sum_{\lambda\in\Lambda^+_G}dim(V_\lambda^H) z^\lambda.$$

\begin{proposition}\label{constantterm}
Let $G,H$ as before. Then the series $\Xi_H^G(z)$ is the constant term in the $x$ variables of the series
$$\left(\prod_{\alpha\in\Phi^+_H}(1-x^{-\alpha})\right)\cdot\pi_x (\Xi_G(x,z))\in \Z[x^{\Lambda_H}][[z^{\Lambda_G^+}]].$$
\end{proposition}

\begin{proof}
Given a $\lambda\in\Lambda^+_G$, the coefficient of $z^\lambda$ in $\pi_{x} (\Xi_G(x,z))$ is just the character in $\Z[x^{\Lambda_H}]$ of the representation $V_\lambda$ of $G$ restricted to $H$ (notice that on $H$ it does not need to be irreducible anymore!). Consider now the additive operator $CT_\Phi$ on $\Z[x^{\Lambda_H}]$ obtained multiplying by $\prod_{\Phi^+_H}(1-x^{-\alpha})$ and taking the constant term in $x$. Consider now the character $\chi^H_{V_\mu}(x)$ of an irreducible representation $V_\mu$ of $H$. By the Weyl character formula, we have
$$\prod_{\Phi^+_H}(1-x^{-\alpha})\chi^H_{V_\lambda}(x)=\sum_{w\in W_H}x^{w(\mu+\rho)-\rho}$$
where $W$ is the Weyl group of $H$, and $\rho=(\sum_{\Phi^+}\alpha)/2$ is the semisum of all positive roots. The only possibility for which this has a constant term is when $w(\mu+\rho)-\rho=0$, that happens if and only if $\mu=0$ and $w$ is the identity. This shows that the operator $CT_\Phi$ is zero on every $\chi^H_{V_\mu}(x)$, besides when $\mu=0$, where its value is 1. This operator, then, on any $H$-representation $V$ calculates the dimension of the $H$-invariants, and this ends the proof.
\end{proof}

The problem of finding the constant term in a multivariate Laurent series is not immediate (and the fact that all happens in a ring of formal series over other variables makes it even harder); a very general way is using a theory of formal multivariate residues; the only references we could find (and they are rather incomplete) are \cite{guocedef}, \cite{guocepap1}, \cite{guocepap2}, where a theory of Malcev-Neumann series is developed. We also believe that following those papers is the only solid hope to extend the calculations of this paper to more complicated cases. In the case of $G=SL_2$, there is a more direct way, that does not use Proposition~\ref{constantterm}, through the results on invariants in binary forms in \cite{brion_residue}; this is the one we will follow for the main case we will analize in detail. We will consider the case of $G=SL_4$ and $H=SL_2$, with the embedding given looking at the standard representation $V_{\omega_1}$ of $SL_4$ as the third symmetric power of the standard representation of $SL_2$. We have the following result.

\begin{theorem}\label{mess}
Let $G=SL_4$ and $H=SL_2$ be as above. Let us denote by $z_1,z_2,z_3$ the monomials $z^{\omega_1},z^{\omega_2},z^{\omega_3}$ of . The generating function $\Xi_H^G$ for $H$-invariants is then
\begin{equation}\label{sl4formula}
\Xi_H^G=\frac{P}{U}
\end{equation}
where
\begin{align*}
P=&1 - z_1^2 z_2 + z_1^4 z_2^2 + z_1^5 z_2 z_3 + z_1 z_2^2 z_3 + z_1^3 z_2^2 z_3 - z_1^7 z_2^2 z_3 - z_1^3 z_2^3 z_3 + z_1^2 z_3^2 - \\
 &-z_2 z_3^2 + z_1^2 z_2 z_3^2 + 2 z_1^2 z_2^2 z_3^2 - z_1^4 z_2^2 z_3^2 - z_1^6 z_2^2 z_3^2 - z_1^4 z_2^3 z_3^2 + z_1^3 z_3^3 + \\
 &+ z_1^3 z_2 z_3^3 - z_1^5 z_2 z_3^3 + z_1 z_2^2 z_3^3 + z_1^3 z_2^2 z_3^3 - 3 z_1^5 z_2^2 z_3^3 - z_1 z_2^3 z_3^3 - z_1^3 z_2^3 z_3^3 -\\ 
&- z_1^5 z_2^3 z_3^3 + z_1^7 z_2^3 z_3^3 + z_1^3 z_2^4 z_3^3 + z_1^4 z_2 z_3^4 - z_1^6 z_2 z_3^4 + z_2^2 z_3^4 - z_1^2 z_2^2 z_3^4 - \\
&-2 z_1^4 z_2^2 z_3^4 - z_1^6 z_2^2 z_3^4 + z_1^8 z_2^2 z_3^4 - z_1^2 z_2^3 z_3^4 + z_1^4 z_2^3 z_3^4 + z_1^5 z_3^5 + z_1 z_2 z_3^5 - \\
&-z_1^3 z_2 z_3^5 - z_1^5 z_2 z_3^5 - z_1^7 z_2 z_3^5 - 3 z_1^3 z_2^2 z_3^5 + z_1^5 z_2^2 z_3^5 + z_1^7 z_2^2 z_3^5 - z_1^3 z_2^3 z_3^5 + \\
&+z_1^5 z_2^3 z_3^5 + z_1^5 z_2^4 z_3^5 - z_1^4 z_2 z_3^6 - z_1^2 z_2^2 z_3^6 - z_1^4 z_2^2 z_3^6 + 2 z_1^6 z_2^2 z_3^6 + z_1^6 z_2^3 z_3^6 -\\
&- z_1^8 z_2^3 z_3^6 + z_1^6 z_2^4 z_3^6 - z_1^5 z_2 z_3^7 - z_1 z_2^2 z_3^7 + z_1^5 z_2^2 z_3^7 + z_1^7 z_2^2 z_3^7 + z_1^3 z_2^3 z_3^7 +\\
&+ z_1^4 z_2^2 z_3^8 - z_1^6 z_2^3 z_3^8 + z_1^8 z_2^4 z_3^8\\
 U=&(1-z_1^4)(1-z_1^3z_3)(1-z_1z_3^3)(1-z_3^4)(1-z_2)(1-z_2^3)(1-z_1^2z_2)(1-z_2z_3^2)
\end{align*}
\end{theorem}

\begin{proof}
Given a polynomial $p$ in $k$ variables $z_1,\ldots,z_k$, we will say it has multidegree $(a_1,\ldots,a_k)$ if the degree as a polynomial in $z_i$ is $a_i$ for every $i=1\ldots,k$. We will say it has multidegree at most $(a_1,\ldots,a_k)$ if the degree in $z_i$ is at most $a_i$ for every $i=1\ldots,k$. We will say a formal series \textit{has denominator} a given polynomial $r$ if it can be written in a form with $r$ as denominator and a polynomial as numerator.

We will first prove that $\Xi_H^G(z)$ has the form

$$\Xi_H^G(z)=\frac{Q}{U\cdot V\cdot (1-z_2)}$$

where $Q$ is a polynomial in $z_1,z_2,z_3$ having multidegree at most $(30,30,30)$, and 
\begin{align*}
V=&(1-z_1^2)(1-z_2^2)(1-z_3^2)(1-z_1^2z_2^3)(1-z_1^4z_2)(1-z_1^4z_2^3)\cdot\\
\cdot&(1-z_1z_3)(1-z_2z_3^4)(1-z_2^3z_3^2)(1-z_2^3z_3^4)
\end{align*}
After that, proving the theorem will just become a matter of checking the equality in a finite number of cases.

To prove the claim, we will use the results in \cite{brion_residue} to find the generating function of the multi-Hilbert function (or Poincar\'e series as it is called in \cite{brion_residue}) of the ($\N^4$-graded) algebra
$$\C[V_3\oplus V_4\oplus V_3\oplus V_1]^H$$
where $V_i$ is the irreducible representation of $SL_2$ having highest weight $i$ times the fundamental weight of $SL_2$. Following Theorem 1 in \cite{brion_residue}, clearing all denominators, we get that the denominator will be $U\cdot V\cdot W$,\footnote{we used the fact that if we apply $\Phi_{h_1,h_2,h_3,h_4}$ (as defined in \cite{brion_residue}), for every factor in the denominator $(1-z_1^{a_1}z_2^{a_2}z_3^{a_3}z_4^{a_4})$, we get a factor $(1-(z_1^{a_1/h_1}z_2^{a_2/h_2}z_3^{a_3/h_3}z_4^{a_4/h_4})^k)$, where $k$ is the smallest integer such that $ka_i/h_i$ is an integer for every $i$.} where
\begin{align*}
W=&(1-z_1z_4)(1-z_1z_4^3)(1-z_2z_4^2)(1-z_2z_4^4)(1-z_3z_4)(1-z_3z_4^3).
\end{align*}
Notice that  $U\cdot V\cdot W$ has multidegree $(25,24,25,14)$; as a Corollary of Theorem 2 in \cite{brion_residue}, we get that the numerator is a polynomial $Q$ in $z_1,z_2,z_3,z_4$ that has multidegree $(21,19,21,12)$. Notice also that we have
\begin{equation}\label{multihilb1}MultiHilb(\C[V_3\oplus (V_4\oplus V_0)\oplus V_3]^H)=\left.\frac{Q}{UVW(1-z_2)}\right\vert_{z_4=0}\end{equation}
\begin{equation}\label{multihilb2}MultiHilb((V_k\otimes\C[V_3\oplus (V_4\oplus V_0)\oplus V_3])^H)=\left.\frac{\frac{1}{k!}\frac{d^k}{dz_4^k}\frac{Q}{W}}{UV(1-z_2)}\right\vert_{z_4=0},\end{equation}
because of the fact that the representation $V_k\otimes V_h$ of $SL_2$ contains one invariant for $k=h$ and zero otherwise, and so to find invariants we need to look at the coefficient of $z_4^k$. After setting $z_4=0$, the denominator of (\ref{multihilb1}) and (\ref{multihilb2}) is exactly $UV(1-z_2)$, as claimed above; in fact, 
$$\frac{d^k}{dz_4^k}\frac{Q}{W}=\frac{Q^{(k)}}{W^k}$$
where $Q^{(k)}$ is a polynomial of multidegree at most $(21+2k,19+2k,21+2k,12+13k)$ (because the multidegree of $W$ is (2,2,2,14) and we take the derivative in $z_4$); then, setting $z_4=0$ will kill exactly all the factors in $W$.

The function $\Xi_H^G(z)$ is the generating function of the multi-Hilbert function of the multi-graded ring $\C[\widehat{G/B}]^H$.
Consider now the resolution (\ref{resolution}) of $\C[\widehat{G/B}]$ in free $R=\C[\bigoplus V_{\omega_i}]$-modules, and let's take the $H$-invariants of it; we get then
$$0\to (M_r\otimes_{\C}R)^H\to\ldots\to (M_1\otimes_{\C}R)^H\to R^H\to\C[\widehat{G/B}]^H\to0$$
so, again, the multiHilbert function for $\C[\widehat{G/B}]^H$ will be the alternating sum of the functions of the modules $(M_i\otimes R)^H$. Notice that $$R^H=\C[V_{100}|_{SL_2}\oplus V_{010}|_{SL_2} \oplus V_{001}|_{SL_2}]^H=\C[V_3\oplus (V_4\oplus V_0)\oplus V_3]^H$$
For the other terms $M_i$, from the calculations in Example~\ref{resolutionSL4}, we only get (lifts of) representations $V_{000}$, $V_{100}$ or $V_{001}$ (so, $V_3$ on $H$) and $V_{010}$ (so, $V_4\oplus V_0$ on $H$). These are just combinations of modules  of type $(V_k\otimes\C[V_3\oplus (V_4\oplus V_0)\oplus V_3])^H$, with $k$ at most 4. We then get immediately that the denominator of $MultiHilb(\C[\widehat{G/B}]^H)$ is $UV(1-z_2)$ claimed above (because all summands do). For a piece of type $(V_k)_{(a_1,a_2,a_3)}$, the numerator have multidegree $(a_1+21+2k,a_2+19+2k,a_3+21+2k)$ (because lifting the multdegree by $(a_1,a_2,a_3)$ means multiplying by $\zzz{a_1}{a_2}{a_3}$ and increasing the multidegree). Looking at all the pieces in the modules $M_i$ in Example~\ref{resolutionSL4}, the highest multidegree we get is at $(V_4)_{(1,3,1)}$, where we get a polynomial $R$ in $z_1,z_2z_3$ of multidegree $(30,30,30)$, as claimed.

Suppose now to know that the series
\begin{equation}\label{manychecks}
U\cdot\Xi_H^G(z)-p(z_1,z_2,z_3)
\end{equation}
has no term for all multidegrees where all $z_1,z_2,z_3$ have degree less then or equal to 30. This series has then the form $z_1^{31}S_1+z_2^{31}S_2+z_3^{31}S_3$, where $S_1,S_2,S_3$ are three series in $z_1,z_2,z_3$. Using $\Xi_H^G(z)=R/UV(1-z_2)$, we get
$$\frac{R}{V(1-z_2)}-P=z_1^{31}S_1+z_2^{31}S_2+z_3^{31}S_3$$\
and hence
$$R-P\cdot V(1-z_2)=(z_1^{31}S_1+z_2^{31}S_2+z_3^{31}S_3)\cdot V(1-z_2)$$
but now $P$ has multidegree $(8,4,8)$ and $V(1-z_2)$ has multidegree $(13,17,13)$, so in the LHS there is just a polynomial of multidegree at most $(30,30,30)$. The RHS needs then to be zero, and the theorem is proved.

It only remains to check that the expression (\ref{manychecks}) has no monomial where $z_1,z_2,z_3$ have exponents $\leq 30$. This can be done quit explicitely in a finite amount of time; the only nontrivial step is finding an explicit algorithm to find the coefficient of a monomial in $\Xi_H^G(z)$, that is, the dimension of a given $(V_{abc}|_{SL_2})^{SL_2}$), and this is the content of the following lemma. The rest is just a long calculation, that we did using a computer (in fact, in a few different ways), but it is also possible to do by hand.
\end{proof}

\begin{lemma}\label{singlecases}
Let $SL_2$ embed into $SL_{N+1}$ as the (identity component of the) stabilizer of a rational normal curve $X$, and let $V_{a_1\ldots a_N}$ be the irreducible representation of $SL_{N+1}$ of highest weight $\sum a_i\omega_i$. Then the dimension of $SL_2$-invariants in $V_{a_1\ldots a_N}$ is the constant term in the Laurent polynomial
\begin{equation}\label{weyl2}
(1-t^{-2})\prod_{1\leq i\leq j\leq n}\frac{t^{\sum_{k=i}^j(a_k+1)}-t^{-\sum_{k=i}^j(a_k+1)}}{t^{j-i+1}-t^{-(j-i+1)}}
\end{equation}
\end{lemma}
\begin{proof}
Let $t=t^{\omega}\in\Z[t^{\Lambda_{SL_2}}]$, for $\omega$ the only dominant weight. We will prove that the character of the restriction $V_{a_1\ldots a_N}|_{SL_2}$ is the entire product in the formula. Then, multiplying by $(1-t^{-2})$ and taking the constant term will give us the dimension of the $SL_2$ invariants, as in the proof of Proposition~\ref{constantterm}. Let now $\eta_0,\ldots,\eta_N$ be the weights of $SL_{N+1}$ of the standard representation $V_{\omega_1}$ (so they sum up to zero), and notice that their projections by  $\pi:\Lambda_{SL_{N+1}}\to\Lambda_{SL_2}$ satisfy $\pi(\eta_i)=t^{N-2i}$. Let $z_i=z^{\eta_i}\in\Z[z^{\Lambda_{SL_{N+1}}}]$. Then, the Weyl character formula tells us that
$$\chi(V_{a_1\ldots a_N})=\frac{\prod_{i=0}^N z_i^{-(N-i)}}{\prod_{\alpha\in\Phi^+}(1-z^{-\alpha})}\sum_{\sigma\in\Sigma_{N+1}}\left(sgn(\sigma)\prod_{i=0}^N z_i^{\sum_{j=\sigma(i)+1}^{N}(a_j+1)}\right)$$
where $\Sigma_{N+1}$ is the set of all permutations of the set $\{0,1,\ldots,N\}$, and $sgn(\sigma)$ is the sign of a permutation $\sigma$. Projecting to $\Z[t^{\Lambda_{SL_2}}]$, is it clear that the first factor becomes just the product of all denominators of (\ref{weyl2}), because $$\prod_{i=0}^N z_i^{-(N-i)}=\prod_{\alpha\in\Phi^+}z^{\alpha/2},$$
and because the positive root $\sum_{k=i}^j\alpha_i$ is sent by $\pi$ to $t^{2(j-i+1)}$.
About the second factor, we have
$$\pi\left(\sum_{\sigma\in\Sigma_{N+1}}\left(sgn(\sigma)\prod_{i=0}^N z_i^{\sum_{j=\sigma(i)+1}^{N}(a_j+1)}\right)\right)=$$
$$=\sum_{\sigma\in\Sigma_{N+1}}\left(sgn(\sigma)\prod_{i=0}^N t^{(N-2i)\sum_{j=\sigma(i)+1}^{N}(a_j+1)}\right)$$
With a little bit of imagination, the second factor can be imagined as the determinant of the following matrix.
\begin{center}
$$\left[\begin{array}{c c c c}
t^{N\sum_{j=1}^{N}(a_j+1)} & t^{(N-2)\sum_{j=1}^{N}(a_j+1)} & \cdots  & t^{-N\sum_{j=1}^{N}(a_j+1)}\\
t^{N\sum_{j=2}^{N}(a_j+1)} & t^{(N-2)\sum_{j=2}^{N}(a_j+1)} & \cdots  & t^{-N\sum_{j=2}^{N}(a_j+1)}\\
\vdots & \vdots & \ddots  & \vdots \\
t^{N(a_N+1)} & t^{(N-2)(a_N+1)} & \cdots  & t^{-N(a_N+1)}\\
1 & 1 & \cdots  & 1 \\
\end{array}\right]$$
\end{center}
This matrix can be interpreted as a Vandermonde matrix, and its determinant is exactly the product of all the numerators of (\ref{weyl2}).
\end{proof}

\begin{remark}
The formulas in \cite{brion_residue} could be used directly to find the explicit formula, instead of just bounding the degree; unfortunately we were not able to produce computational tools able to sustain such a complicated calculation.
\end{remark}

\section{Vector partition functions and splines}\label{asymptotics}

\subsection{Vector partition functions}

The final aim for this chapter is to find the function $\dimasi$ from Section~\ref{inters_theory} explicitly, that is the asymptotic value for the coefficients of $\Xi_H^G$. Our first step will be to prove the next proposition, that we already used in Section~\ref{inters_theory}.

\begin{proposition}\label{capitalM}
Let $M:\Lambda^+\to\N$ be the function associating every dominant weight $\lambda$ of $G$ the number $dim(V_\lambda^H)$. Then there exists a list $A$ of vectors in $\Lambda^+$ such that
\begin{itemize}
\item[(i)] The function $M$ is zero outside $\Lambda(A)\cap\Lambda^+$.
\item[(ii)] In each big cell $\mathfrak{c}$ of the chamber complex $\mathcal{C}(A)$, we have
$$M(\lambda) = q^{\mathfrak{c}}(\lambda) + b^{\mathfrak{c}}(\lambda) \quad \forall \lambda\in\Lambda(A)\cap\mathfrak{c}$$
where
\begin{itemize}
\item $q^{\mathfrak{c}}$ is a quasi-polynomial; that means, there is a finite index sublattice $\Lambda(A,\mathfrak{c})$ of $\Lambda(A)$ such that $q^{\mathfrak{c}}$ agrees on a different polynomial on each coset $\lambda+\Lambda(A,\mathfrak{c})$ for $\lambda\in\Lambda(A)$.
\item $b^{\mathfrak{c}}$ is a function that is zero on $\mathfrak{c}\cap\Lambda(A)\setminus(\mathfrak{c}+\beta)\cap\Lambda(A)$, where $\beta$ is an element of $\Gamma$ such that $M(\beta)$ is nonzero.
\end{itemize}
\end{itemize}
\end{proposition}

We will need first a few definitions.

\begin{definition}
Let $A=[a_1,\ldots,a_m]$ be a list of vectors in a lattice $\Lambda$ sitting inside a vector space $V=\Lambda_\R$ of rank $d$, such that 0 is not in their convex hull. We will call \textbf{vector partition function} the function $\mathcal{T}_{A}:\Lambda\to\N$ assigning to every vector $\lambda\in\Lambda$ the number of solutions in nonnegative integers $c_1\ldots,c_k$ of the equation $$c_1a_1+\ldots+c_ka_k=\lambda.$$
\end{definition}

If we denote by $\Lambda^+\subset\Lambda$ the semigroup generated by the $\lambda_i$, we have
$$\prod_{i=1}^k\frac{1}{1-z^{\lambda_i}}=\sum_{\lambda\in\Lambda^+}\mathcal{T}_{A}(\lambda)z^\lambda\in\Z[[z^{\Lambda^+}]].$$

Let $A=[a_1,\ldots,a_m]$ be as before. Let $\Lambda(A)$ be the $\Z$-span of the vectors in $\mathcal{A}$. Recall from Definition~\ref{convex} the cone $C(A)$, the definition of a big cell $\mathfrak{c}$, and the chamber complex $\mathcal{C}(A)$. We have then the following, that is Theorem~1 in \cite{sturmfels}.

\begin{theorem}\label{sturmfelsvpf}
Let $\mathcal{T}_{A}$ be the partition function for a list of vectors $A=[a_1,\ldots,a_m]$ spanning the entire vector space $V$ (or rank $d$) they lie in. Then on each big cell $\mathfrak{c}$ of $\mathcal{C}(A)$ there is a polynomial $p$ of degree $m-d$, and a quasipolynomial $q$ of degree strictly less, such that
$$\mathcal{T}_{A}(\lambda)=p(\lambda)+q(\lambda)\quad \forall \lambda\in\Lambda(A)\cap \mathfrak{c}.$$
\end{theorem}

Let us now prove Proposition~\ref{capitalM}.

\begin{proof}[Proof of Proposition~\ref{capitalM}]
Consider the ring $\C[\widehat{G/B}]$ as before, with its multi-grading indexed by $\Lambda^+_G$. It is a finitely generated algebra, with an action of $H$ reductive, and hence the ring of invariants  $\C[\widehat{G/B}]^H$ will be finitely generated as well, with the same grading (because the $H$-action respects it). Notice that $\Xi_H^G(z)$ is just the multi-Hilbert polynomial of this ring. Let $f_1,\ldots,f_m$ be generators of this ring, of multidegrees $\lambda_1,\ldots,\lambda_m$ (that we can choose to be homogeneous) and let us consider a finite free $S=\C[f_1,\ldots,f_m]$-resolution of $\C[\widehat{G/B}]^H$
$$0\to F_r\to\ldots\to F_1\to F_0=S\to \C[\widehat{G/B}]^H\to 0.$$
Denoting by $\{g_{i,j}\}_{i=1}^{i_j}$ a set of generators for $F_j$ (that we can choose to be homogeneous with respect to the $\Lambda^+$-grading), of degrees $\lambda_{i,j}$, we get
$$\Xi_H^G(z)=\frac{\sum_{j=0}^r(-1)^j\sum_{i=1}^{i_j}z^{\lambda_{i,j}}}{\prod_{k=1}^s(1-z^{\lambda_k})}.$$
Let now $\Lambda$ be the $\Z$-span of the vectors $\lambda_i$ (notice that the $\lambda_{i,j}$ belong to it as well). Let $\mathcal{T}_{A}$ be the partition function for the list $A=[\lambda_1,\ldots,\lambda_s]$.
The coefficient of $z^\lambda$ in $\Xi_H^G(z)$ (that is also $dim(V_\lambda^H)$ ) is then
$$\sum_{j=0}^r(-1)^j\sum_{i=1}^{i_j}\mathcal{T}_{A}(\lambda-\lambda_{i,j}).$$
In each big cell $\mathfrak{c}$, let $\gamma^{\mathfrak{c}}$ be a vector in $\Lambda(A)\cap\mathfrak{c}$ such that $\gamma^{\mathfrak{c}}-\lambda_{i,j}$ is still in $\mathfrak{c}$ for every $i,j$. On $\Gamma\cap(\mathfrak{c}+\gamma^{\mathfrak{c}})$, then, the function $dim(V_\lambda^H)$ is a the result of a difference operator applied to a quasi polynomial, that is a quasi polynomial again. We get then our result.
\end{proof}

\subsection{Splines}

The next topic in this chapter will be devoted to finding explicitly these functions. For our practical use, as seen in Proposition~\ref{capitalM2}, we will only need the top degree part of the polynomials $q_\mathcal{C}$ from Proposition~\ref{capitalM}. To be able to do this, we will introduce the concept of \textit{spline}, that is a continuous counterpart of the (discrete) vector partition function, together with its ``box'' equivalent. Let us fix a basis for $\Lambda$, and the associated Euclidea metric on $V$.

\begin{definition}
Let $A=[a_1,\ldots,a_m]$ be a list of vectors in a lattice $\Lambda$ sitting inside a vector space $\Lambda_\R=V$ of rank $d$, whose $\R$-span is the entire $\Lambda_\R$, and such that 0 is not in their convex hull. The \textbf{multivariate spline} is the function $T:\Lambda_\R\to\R$ such that, for every function $f\in\C^{\infty}_c(\Lambda_\R)$ with compact support, we have
$$\int_{\Lambda_{\R}}f(x)T(x)dx=\int_0^{\infty}\cdots\int_0^{\infty}f(\sum_{i=1}^m t_i\lambda_i)dt_1\cdots dt_m.$$
Similarly, the \textbf{box spline} is the function  $B:\Lambda_\R\to\R$ such that, for every function $f\in\C^{\infty}_c(\Lambda_\R)$ with compact support, we have
$$\int_{\Lambda_{\R}}f(x)B(x)dx=\int_0^{1}\cdots\int_0^{1}f(\sum_{i=1}^m t_i\lambda_i)dt_1\cdots dt_m.$$
\end{definition}

In the case where the $\R$-span of $A$ is not the entire $V$, the splines can be defined as distributions, i.e. operators on a suitable space of functions.

Splines have been quite widely studied; for a resource containing most of the results, we refer to \cite{topics}. Given an element $a\in\Lambda$, we will denote by $\partial_a$ the derivative in that direction on differentiable functions on $V$. Given a list $B$ of vectors in $\Lambda$ we will denote by $\partial_B=\prod_{a\in B}\partial_a$. We have then the following. 

\begin{proposition}\label{factsplines}
Let $A=[a_1,\ldots,a_m]$ be as before. Then
\begin{itemize}
\item[i)] (cf. \cite{topics}, Th.~9.7, Proposition~7.17) the spline $T_A$ agrees, in every big cell $\mathfrak{c}$, with a polynomial $T_A^{\mathfrak{c}}$ of degree $n-d$;
\item[ii)]  (cf. \cite{topics}, Proposition~7.15) if A is a basis for $V$, and $c$ is the determinant of the matrix whose columns are the vectors of $A$, then $T_A=\frac{1}{c}\chi_{C(A)}$, where $\chi_{C(A)}$ is the characteristic function of the cone $C(A)$; 
\item[iii)]  (cf. \cite{topics}, Lemma~7.23 and Proposition~7.14) for each element $a\in A$ such that $A\setminus\{a\}$ still spans the entire $V$, we have that $T_A$ is continuous in the direction of $a$, and $\partial_a T_A^{\mathfrak{c}}=T_{A\setminus\{a\}}^{\mathfrak{c}}$; 
\item[iv)]  (cf. \cite{topics}, Theorem~11.16) if $B\subset A$ is such that $A\setminus B$ does not span $V$, we have $\partial_B T_A^{\mathfrak{c}}=0$ for each big cell $\mathfrak{c}$.
\end{itemize}
\end{proposition}

A further property that we will need are the following.

\begin{proposition}\label{hyperplane}
Suppose two big cells $\mathfrak{c}_1$ and $\mathfrak{c}_2$ intersect in an $s-1$-dimensional locus, and let $H$ be the span of it. Suppose $A$ contains $k$ vectors outside of the hyperplane $H$, and let $L$ be a linear form vanishing on $H$. Then 
$$L^{k-1}|T_A^{\mathfrak{c}_1}-T_A^{\mathfrak{c}_1}$$
as polynomials. This is also true if we take one of the big cells to be the complement $\mathfrak{c}_0=\overline{V\setminus C(A)}$, for which of course $T_A^{\mathfrak{c}_0}=0$.
\end{proposition}

\begin{proof}
Let us prove it by induction on $k$. It is obvious that the statement is equivalent to asking the function to be of class $C^{k-2}$ along $\mathfrak{c}_1\cap\mathfrak{c}_2$ in any direction transversal to $H$. Let $a$ be a vector in $A$ that does not belong to $H$.
If $k=2$, then $A\setminus\{a\}$ still spans the entire $V$, so $T_A$ is continuous in the direction of $a$ (that is transverse to $a$) by Proposition~\ref{factsplines} iii). If $k>2$, the function $\partial_a T_A$ will be of class $C^{k-3}$ in the direction of $a$ along $\mathfrak{c}_1\cap\mathfrak{c}_2$, so $T_A$ will be of class $C^{k-2}$. The same argument applies also at the boundary of $C(A)$.   
\end{proof}

Based on Definition~\ref{dimasi}, we need only the top degree part of the polynomials the partition function agrees with. In this sense, the spline is exactly what we are looking for.

\begin{proposition}\label{vpfspline}
Let $\mathcal{A}=[a_1,\ldots,a_m]$ be as before, and let $e$ be the index of $\Lambda(A)$ in $\Lambda$. Let us denote by $\mathcal{T}_A^{\mathfrak{c},top}$ the top degree part of the function $\mathcal{T}_A$ agrees with on $\Lambda(A)\cup\mathfrak{c}$, as in Theorem~\ref{sturmfelsvpf}. We have then
$$T_A=\frac{\mathcal{T}_A^{\mathfrak{c},top}}{e}$$
\end{proposition}

\begin{proof}
In this proof, we will refer directly to notation and results from \cite{DM}. Suppose $\Lambda(A)=\Lambda$; then all the lists $A_z$ (as in Section~5 of \cite{DM}) are all strictly smaller than $A$; this implies that the component $E\mathcal{T}_A^{\mathfrak{c}}$ has (quasi)polynomial components of degree strictly less than $m-d$ (cf. (5.5) of \cite{DM}), and hence $P\mathcal{T}_A^{\mathfrak{c}}$ always contains the top degree part of $\mathcal{T}_A^{\mathfrak{c}}$. So, Proposition~5.3 of \cite{DM} gives us the result.
If $\Lambda(A)\subsetneq\Lambda$, let us try to apply the above argument to $\Lambda(A)$ instead of $\Lambda$; while the values of $\mathcal{T}_A$ are not affected by the choice of the lattice, the function $T_A$ does, because we are using a different Euclidean metric for the integral in its definition; in particular, we are applying a linear transformation of determinant $1/e$, and the new spline $T_A^{(\Lambda(A))}$ will be equal to $eT_A$. We have then, from the first part, that $\mathcal{T}_A^{\mathfrak{c},top}=T_A^{(\Lambda(A))}=eT_A$ as needed.
\end{proof}

\subsection{Asymptotics for $SL_2$-invariants in representations of $SL_4$}

In this section we will keep following Theorem~\ref{mess}, and produce a formula for the function $\dimasi$ in the case of $X$ a twisted cubic; this will make us able to evaluate explicitely the volume of divisors on $M_L$ in this case, by Theorem~\ref{main}.

From now on in this section, $G$ will be $SL_4$, and $H$ will be the identity connected component in the stabilizer of a twisted cubic. Remember $P$ and $U$ from Theorem~\ref{mess}. We have

$$\Xi^G_H=\frac{P}{U}=\frac{P_1}{U_1}+\frac{P_2}{U_2}+\frac{P_3}{U_3}+\frac{P_4}{U_4},$$
where
\begin{align*}
P_1=&-z_1^2 z_2 - z_1^3 z_3 - z_1 z_2 z_3 + z_1^5 z_2 z_3 + z_1 z_2^2 z_3 - z_1^2 z_3^2 + z_1^6 z_3^2 - z_2 z_3^2 + \\
  &+z_1^2 z_2 z_3^2 + z_1^4 z_2 z_3^2 - z_1 z_3^3 + z_1^3 z_3^3 + z_1^5 z_3^3 + z_1^3 z_2 z_3^3 + z_1^3 z_2^2 z_3^3 + z_1^4 z_3^4 + \\ 
 &+z_1^2 z_2 z_3^4 + z_1^4 z_2 z_3^4 + z_1^4 z_2^2 z_3^4 + z_1^3 z_3^5 +  z_1^5 z_3^5 + z_1 z_2 z_3^5 + z_1^2 z_3^6 + z_1^6 z_2^2 z_3^6\\
P_2=&-1 + z_1^2 z_2 + z_1 z_2 z_3 + z_2 z_3^2 \\
P_3=&1 + z_1^3 z_3 + z_1^2 z_3^2\\
P_4=&1 + z1^2 z3^2 + z1 z3^3\\
U_1=&(1-z_1^4)(1-z_1^3z_3)(1-z_1z_3^3)(1-z_3^4)(1-z_2)(1-z_2^3) \\
U_2=&(1-z_1^4)(1-z_3^4)(1-z_2)(1-z_2^3)(1-z_1^2z_2)(1-z_2z_3^2) \\
U_3=&(1-z_1^4)(1-z_1^3z_3)(1-z_3^4)(1-z_2)(1-z_2^3)(1-z_1^2z_2) \\
U_4=&(1-z_1^4)(1-z_1z_3^3)(1-z_3^4)(1-z_2)(1-z_2^3)(1-z_2z_3^2)
\end{align*}

Let now $A_i$ be the list of exponents of the $z$ variables in the factors in $U_i$, thought as elements of $\N^3$. Using as basis $\omega_1,\omega_2,\omega_3$ of $\Lambda_G$, and coordinates $x_1,x_2,x_3$, we can talk about $\dimasx$ as a function in the variables $x_1,x_2,x_3$, that by Propositions~\ref{capitalM} and \ref{capitalM2} is a piecewise polynomial function of $3$. We have the following lemma.

\begin{lemma}\label{a0}
We have
$$\dimasx=24T_{A_1}(x)+4T_{A_2}(x)+6T_{A_3}(x)+6T_{A_4}(x).$$
\end{lemma}

\begin{proof}
We will use Proposition~\ref{vpfspline}. Given a series
$$\frac{\sum_B c_bz^b}{\prod_A(1-z^a)},$$
supposing that all elements of $B$ lie in $\Lambda(A)$ and that $\sum_B c_b\neq 0$, then the coefficient of $z^\lambda$ is given by
$$\sum_B c_b\mathcal{T}_A(\lambda-b).$$
On $\Lambda(A)$, this will be a piecewise quasipolynomial, whose top degree part will be 
$$\sum_B c_b\mathcal{T}_A^{top}(\lambda-b)=\sum_B c_b\mathcal{T}_A^{top}(\lambda)=\left(\sum_B c_b\right)T_A(\lambda)|\Lambda/\Lambda(A)|.$$
For $\frac{P_i}{U_i}$ with $i=2,3,4$, then, we have $\Lambda(A_i)=\{(i,j,k): 2|i+k\}$, and the sum of the three top degree parts is
$$4T_{A_2}+6T_{A_3}+6T_{A_4}.$$
For $\frac{P_1}{U_1}$, we have $\Lambda(A_1)=\{(i,j,k): 4|i+k\}$. The coefficients of the series $\frac{1}{U_1}$, then, will agree on $\Lambda(A_1)$ with a piecewise quasipolynomial having as top degree term $4 T_{A_1}$. The numerator, then, has some monomials (whose sum of coefficient is 6) that lie still in $\Lambda(A_1)$, and the other monomials (whose sum of coefficients is still 6) that have the exponent in the nonzero class in $\Lambda(A_2)/\Lambda(A_1)$; so, the coefficients of the series $\frac{P_1}{U_1}$ will agree on the entire $\Lambda(A_2)$ with a function having top degree 
$$6\cdot 4 T_{A_1}.$$
Notice that $\Lambda(A_2)$ is contained in the root lattice $\Lambda_\alpha=\{(i,j,k): 4|i+2k+3k\}$, so the number $c$ in Definition~\ref{dimasi} is 1, and this concludes the proof.
\end{proof}

It remains now only to find the four multivariate splines for the lists $A_1,A_2,A_3,A_4$, and we will use Proposition~\ref{hyperplane}. Everything will happen inside the Weyl chamber $\mathcal{W}$ of $G$, that is the positive octant of a three dimensional vector space in this case; we will only show its intersection with the hyperplane $x+y+z=1$, that is a triangle. We will denote by $\partial_i$ the derivative $\partial/\partial x_i$.

\begin{example}\label{a1}
We have $$A_1=\left[
\left(\begin{array}{c}4\\0\\0\end{array}\right),
\left(\begin{array}{c}3\\0\\1\end{array}\right),
\left(\begin{array}{c}1\\0\\3\end{array}\right),
\left(\begin{array}{c}0\\0\\4\end{array}\right),
\left(\begin{array}{c}0\\1\\0\end{array}\right),
\left(\begin{array}{c}0\\3\\0\end{array}\right)
\right]$$
\begin{center}
\begin{tikzpicture}
[scale=0.8,transform shape]
\tikzstyle{circ}=[draw,circle,fill=black,minimum size=4pt,
                            inner sep=0pt]
\draw (0,0) -- (4,6) -- (8,0) -- (0,0);
\draw (2,0) -- (4,6) -- (6,0);
\draw (0,0) node (1) [circ,label=180:$\left(\begin{array}{c}4\\0\\0\end{array}\right)$]{};
\draw (2,0) node (1) [circ,label=270:$\left(\begin{array}{c}3\\0\\1\end{array}\right)$]{};
\draw (6,0) node (1) [circ,label=270:$\left(\begin{array}{c}1\\0\\3\end{array}\right)$]{};
\draw (8,0) node (1) [circ,label=0:$\left(\begin{array}{c}0\\0\\4\end{array}\right)$]{};
\draw (4,6) node (1) [circ,label=30:$\left(\begin{array}{c}0\\1\\0\end{array}\right)$]{};
\draw (4,6) node (1) [circ,label=150:$\left(\begin{array}{c}0\\3\\0\end{array}\right)$]{};
\draw (2,1) node (1) [label=135:$\mathfrak{c}_1$]{};
\draw (4,1) node (1) [label=90:$\mathfrak{c}_2$]{};
\draw (6,1) node (1) [label=45:$\mathfrak{c}_3$]{};
\end{tikzpicture}
\end{center}

\noindent We have that $T_{A_1}^{\mathfrak{c}_1}$ is divided by $x_3^2$ and $x_2$, by Proposition~\ref{hyperplane}, so it will be of the form $ax_2x_3^2$. From Proposition~\ref{factsplines} ii) and iii), we have $$\partial_2(\partial_1+3\partial_3)(3\partial_1+\partial_3)T_{A_1}^{\mathfrak{c}_1}=T_{\Tiny{\left[
\left(\!\!\!\begin{array}{c}4\\0\\0\end{array}\!\!\!\right),
\left(\!\!\!\begin{array}{c}0\\0\\4\end{array}\!\!\!\right),
\left(\!\!\!\begin{array}{c}0\\3\\0\end{array}\!\!\!\right)
\right]}}^{\mathfrak{c}_1}=\frac{1}{4\cdot 4\cdot 3}$$
that gives us $a=1/288$. The difference $T_{A_1}^{\mathfrak{c}_2}-T_{A_1}^{\mathfrak{c}_1}$ is a multiple of  $(x_1-3x_3)^2$ by Proposition~\ref{hyperplane}, so it will be of the form $l(x)(x_1-3x_3)^2$, for $l(x)$ a linear form. We also know that $T_{A_1}^{\mathfrak{c}_2}$ needs to be symmetrical in $x_1$ and $x_3$, and this leaves as only choice $l(x)=-x_2/2304$. By symmetry, we have $T_{A_1}^{\mathfrak{c}_3}-T_{A_1}^{\mathfrak{c}_2}=x_2(3x_1-x_3)^2/2304$ too.
\end{example}

\begin{example}\label{a2}
We have $$A_2=\left[
\left(\begin{array}{c}4\\0\\0\end{array}\right),
\left(\begin{array}{c}0\\0\\4\end{array}\right),
\left(\begin{array}{c}0\\1\\0\end{array}\right),
\left(\begin{array}{c}0\\3\\0\end{array}\right)
\left(\begin{array}{c}2\\1\\0\end{array}\right),
\left(\begin{array}{c}0\\1\\2\end{array}\right),
\right]$$
\begin{center}
\begin{tikzpicture}
[scale=0.8,transform shape]
\tikzstyle{circ}=[draw,circle,fill=black,minimum size=4pt,
                            inner sep=0pt]
\draw (0,0) -- (4,6) -- (8,0) -- (0,0);
\draw (0,0) -- (6.666,2) -- (1.333,2) -- (8,0);
\draw (0,0) node (1) [circ,label=180:$\left(\begin{array}{c}4\\0\\0\end{array}\right)$]{};
\draw (8,0) node (1) [circ,label=0:$\left(\begin{array}{c}0\\0\\4\end{array}\right)$]{};
\draw (4,6) node (1) [circ,label=30:$\left(\begin{array}{c}0\\1\\0\end{array}\right)$]{};
\draw (4,6) node (1) [circ,label=150:$\left(\begin{array}{c}0\\3\\0\end{array}\right)$]{};
\draw (1.333,2) node (1) [circ,label=150:$\left(\begin{array}{c}2\\1\\0\end{array}\right)$]{};
\draw (6.666,2) node (1) [circ,label=30:$\left(\begin{array}{c}0\\1\\2\end{array}\right)$]{};
\draw (4,0.2) node (1) [label=90:$\mathfrak{c}_1$]{};
\draw (2.1,0.8) node (1) [label=135:$\mathfrak{c}_2$]{};
\draw (5.9,0.8) node (1) [label=45:$\mathfrak{c}_3$]{};
\draw (4,1.2) node (1) [label=90:$\mathfrak{c}_4$]{};
\draw (4,3) node (1) [label=90:$\mathfrak{c}_5$]{};
\end{tikzpicture}
\end{center}

\noindent We have that $T_{A_2}^{\mathfrak{c}_1}$ is divided by $x_2^3$, by Proposition~\ref{hyperplane}, so it will be of the form $ax_2^3$. From Proposition~\ref{factsplines} ii) and iii), we have $$\partial_2(\partial_2+2\partial_3)(2\partial_1+\partial_2)T_{A_2}^{\mathfrak{c}_1}=T_{\Tiny{\left[
\left(\!\!\!\begin{array}{c}4\\0\\0\end{array}\!\!\!\right),
\left(\!\!\!\begin{array}{c}0\\0\\4\end{array}\!\!\!\right),
\left(\!\!\!\begin{array}{c}0\\3\\0\end{array}\!\!\!\right)
\right]}}^{\mathfrak{c}_1}=\frac{1}{4\cdot 4\cdot 3}$$
that gives us $a=1/288$ again. Similarly, we have
\begin{align*}
T_{A_2}^{\mathfrak{c}_2}-T_{A_2}^{\mathfrak{c}_1}&=a_{12}(2x_2-x_3)^3\\
T_{A_2}^{\mathfrak{c}_4}-T_{A_2}^{\mathfrak{c}_2}&=a_{24}(x_1-2x_2)^3\\
T_{A_2}^{\mathfrak{c}_3}-T_{A_2}^{\mathfrak{c}_4}&=a_{43}(2x_2-x_3)^3\\
T_{A_2}^{\mathfrak{c}_1}-T_{A_2}^{\mathfrak{c}_3}&=a_{31}(x_1-2x_2)^3
\end{align*}
for suitable constants $a_{12},a_{24},a_{43},a_{31}$. Summing the four equations, we get also $a_{12}=-a_{43}$ and $a_{24}=-a_{31}$. The difference $T_{A_2}^{\mathfrak{c}_1}-T_{A_2}^{\mathfrak{c}_1}$ is a multiple of  $(x_1-3x_3)^2$ by Proposition~\ref{hyperplane}, so it will be of the form $l(x)(x_1-3x_3)^2$, for $l(x)$ a linear form. Imposing $x_3|T_{A_2}^{\mathfrak{c}_2}$ and $x_1|T_{A_2}^{\mathfrak{c}_4}$, we get $a_{12}=-1/2304$ and $a_{31}=1/2304$. Imposing also $x_1x_3|T_{A_2}^{\mathfrak{c}_5}$, we also get 
$$T_{A_2}^{\mathfrak{c}_5}-T_{A_2}^{\mathfrak{c}_4}=-(x_1-2x_2+x_3)^3/2304.$$
\end{example}

\begin{example}\label{a3}
We have $$A_3=\left[
\left(\begin{array}{c}4\\0\\0\end{array}\right),
\left(\begin{array}{c}0\\0\\4\end{array}\right),
\left(\begin{array}{c}0\\1\\0\end{array}\right),
\left(\begin{array}{c}0\\3\\0\end{array}\right)
\left(\begin{array}{c}2\\1\\0\end{array}\right),
\left(\begin{array}{c}0\\1\\2\end{array}\right),
\right]$$
\begin{center}
\begin{tikzpicture}
[scale=0.8,transform shape]
\tikzstyle{circ}=[draw,circle,fill=black,minimum size=4pt,
                            inner sep=0pt]
\draw (0,0) -- (4,6) -- (8,0) -- (0,0);
\draw (4,6) -- (2,0) -- (1.333,2) -- (8,0);
\draw (0,0) node (1) [circ,label=180:$\left(\begin{array}{c}4\\0\\0\end{array}\right)$]{};
\draw (2,0) node (1) [circ,label=270:$\left(\begin{array}{c}3\\0\\1\end{array}\right)$]{};
\draw (8,0) node (1) [circ,label=0:$\left(\begin{array}{c}0\\0\\4\end{array}\right)$]{};
\draw (4,6) node (1) [circ,label=30:$\left(\begin{array}{c}0\\1\\0\end{array}\right)$]{};
\draw (4,6) node (1) [circ,label=150:$\left(\begin{array}{c}0\\3\\0\end{array}\right)$]{};
\draw (1.333,2) node (1) [circ,label=150:$\left(\begin{array}{c}2\\1\\0\end{array}\right)$]{};
\draw (1,0.2) node (1) [label=90:$\mathfrak{c}_1$]{};
\draw (2.4,0.8) node (1) [label=135:$\mathfrak{c}_2$]{};
\draw (4,0.2) node (1) [label=45:$\mathfrak{c}_3$]{};
\draw (2.5,2.7) node (1) [label=90:$\mathfrak{c}_4$]{};
\draw (4.5,2.5) node (1) [label=90:$\mathfrak{c}_5$]{};
\end{tikzpicture}
\end{center}

\noindent We have that $T_{A_3}^{\mathfrak{c}_1}$ is divided by $x_2^2x_3$, by Proposition~\ref{hyperplane}, so it will be of the form $ax_2^2x_3$. From Proposition~\ref{factsplines} ii) and iii), we have $$\partial_2(3\partial_1+\partial_3)(2\partial_1+\partial_2)T_{A_3}^{\mathfrak{c}_1}=T_{\Tiny{\left[
\left(\!\!\!\begin{array}{c}4\\0\\0\end{array}\!\!\!\right),
\left(\!\!\!\begin{array}{c}0\\0\\4\end{array}\!\!\!\right),
\left(\!\!\!\begin{array}{c}0\\3\\0\end{array}\!\!\!\right)
\right]}}^{\mathfrak{c}_1}=\frac{1}{4\cdot 4\cdot 3}$$
that gives us $a=1/96$. We have then
\begin{align*}
T_{A_3}^{\mathfrak{c}_2}-T_{A_3}^{\mathfrak{c}_1}&=a_{12}(x_1-2x_2-3x_3)^3\\
T_{A_3}^{\mathfrak{c}_3}-T_{A_3}^{\mathfrak{c}_2}&=l(x)(x_1-3x_3)^2\\
T_{A_3}^{\mathfrak{c}_4}-T_{A_3}^{\mathfrak{c}_2}&=a_{24}(x_1-2x_2)^3\\
\end{align*}
for $a_{12}$ a constant and $l(x)$ a linear form. Imposing that $x_2^2|T_{A_3}^{\mathfrak{c}_3}$, we get that $l(x)=a_{12}(-x_1+6x_2+3x_3)$; imposing that $x_3|T_{A_3}^{\mathfrak{c}_4}$, we also get $a_{24}=-a_{12}$. For the same argument as in the previous example, we have

\begin{align*}
T_{A_3}^{\mathfrak{c}_5}&=T_{A_3}^{\mathfrak{c}_4}+T_{A_3}^{\mathfrak{c}_3}-T_{A_3}^{\mathfrak{c}_2}=\\
&=a_{12}[(x_1-2x_2-3x_3)^3-(x_1-2x_2)^3+(-x_1+6x_2+3x_3)(x_1-3x_3)^2]+x_2^2x_3/96
\end{align*}
Imposing it to be divided by $x_1^2$, we get $a_{12}=1/3456$, that completes this case as well.
\end{example}

The case of $T_{A_4}$ can be obtained just switching $x_1$ and $x_3$ in the above.
We reached then the following, that is a direct consequence of Lemma~\ref{a0} and Examples~\ref{a1},~\ref{a2},~\ref{a3}. 

\begin{corollary}\label{finalformula}
In the coordinates $x_1,x_2,x_3$ as above, the piecewise polynomial function $\dimasx$ agrees with a different polynomial of degree 3 in each of the following cells.
\begin{center}
\begin{tikzpicture}
[scale=0.8,transform shape]
\tikzstyle{circ}=[draw,circle,fill=black,minimum size=4pt,
                            inner sep=0pt]
\draw (0,0) -- (4,6) -- (8,0) -- (0,0);
\draw (4,6) -- (2,0) -- (1.333,2) -- (6.666,2) -- (6,0) -- (4,6);
\draw (0,0) node (1) [circ,label=180:$\left(\begin{array}{c}4\\0\\0\end{array}\right)$]{};
\draw (2,0) node (1) [circ,label=270:$\left(\begin{array}{c}3\\0\\1\end{array}\right)$]{};
\draw (6,0) node (1) [circ,label=270:$\left(\begin{array}{c}1\\0\\3\end{array}\right)$]{};
\draw (8,0) node (1) [circ,label=0:$\left(\begin{array}{c}0\\0\\4\end{array}\right)$]{};
\draw (4,6) node (1) [circ,label=30:$\left(\begin{array}{c}0\\1\\0\end{array}\right)$]{};
\draw (4,6) node (1) [circ,label=150:$\left(\begin{array}{c}0\\3\\0\end{array}\right)$]{};
\draw (1.333,2) node (1) [circ,label=150:$\left(\begin{array}{c}2\\1\\0\end{array}\right)$]{};
\draw (6.666,2) node (1) [circ,label=30:$\left(\begin{array}{c}0\\1\\2\end{array}\right)$]{};
\draw (1,0.2) node (1) [label=90:$\mathfrak{c}_1$]{};
\draw (2.4,0.8) node (1) [label=135:$\mathfrak{c}_2$]{};
\draw (2.5,2.7) node (1) [label=90:$\mathfrak{c}_3$]{};
\draw (4,0.5) node (1) [label=90:$\mathfrak{c}_4$]{};
\draw (4,2.5) node (1) [label=90:$\mathfrak{c}_5$]{};
\draw (7,0.2) node (1) [label=90:$\mathfrak{c}_6$]{};
\draw (5.6,0.8) node (1) [label=45:$\mathfrak{c}_7$]{};
\draw (5.5,2.7) node (1) [label=90:$\mathfrak{c}_8$]{};
\end{tikzpicture}
\end{center}
We have, furthermore,
$$\dimasx^{\mathfrak{c}_4}=\frac{24x_1x_2x_3-x_2^3-3x_2(x_1+x_3-x_2)^2}{288}$$
$$\dimasx^{\mathfrak{c}_3}-\dimasx^{\mathfrak{c}_2}=\dimasx^{\mathfrak{c}_5}-\dimasx^{\mathfrak{c}_4}=$$
$$=\dimasx^{\mathfrak{c}_8}-\dimasx^{\mathfrak{c}_7}=\frac{(2x_2-x_1-x_3)^3}{576}$$
$$\dimasx^{\mathfrak{c}_3}-\dimasx^{\mathfrak{c}_5}=\dimasx^{\mathfrak{c}_2}-\dimasx^{\mathfrak{c}_4}=\frac{(x_1-3x_3)^3}{576}$$
$$\dimasx^{\mathfrak{c}_8}-\dimasx^{\mathfrak{c}_5}=\dimasx^{\mathfrak{c}_7}-\dimasx^{\mathfrak{c}_4}=\frac{(x_3-3x_1)^3}{576}$$
$$\dimasx^{\mathfrak{c}_1}-\dimasx^{\mathfrak{c}_2}=-\frac{(x_1-2x_2-3x_3)^3}{576}$$
$$=\dimasx^{\mathfrak{c}_6}-\dimasx^{\mathfrak{c}_7}=-\frac{(x_3-2x_2-3x_1)^3}{576}$$
that completely determines it.
\end{corollary}

We can now finally apply Theorem~\ref{main}, to calculate actual numbers. To find the intergrals in question, we used Mathematica, and the following lines of code.

\begin{doublespace}
\noindent\(\pmb{g[\text{x$\_$},\text{y$\_$},\text{z$\_$}]\text{:=} (1/576)\text{UnitStep}[x]\text{UnitStep}[y]\text{UnitStep}[z]}\\
\pmb{(48x y z -2y{}^{\wedge}3-6y(x+z-y){}^{\wedge}2+\text{UnitStep}[2y-x-z](2y-x-z){}^{\wedge}3+}\\
\pmb{\text{UnitStep}[z-3x](z-3x){}^{\wedge}3+\text{UnitStep}[x-3z](x-3z){}^{\wedge}3-}\\
\pmb{\text{UnitStep}[x-2y-3z](x-2y-3z){}^{\wedge}3-\text{UnitStep}[z-2y-3x](z-2y-3x){}^{\wedge}3);}\\
\pmb{\text{dimasVHx}=\text{PiecewiseExpand}[g[x,y,z]];}\\
\pmb{\text{dimasVx}=x y z(x+y)(y+z)(x+y+z)/12;}\\
\pmb{P[\text{a$\_$},\text{b$\_$},\text{c$\_$}]\text{:=}}\\
\pmb{\text{ImplicitRegion}[\{x>0,y>0,z>0, x+2y+3z<a+2b+3c,3x+2y+z<3a+2b+c,}\\
\pmb{x+2y+z<a+2b+c\},\{x,y,z\}];}\\
\pmb{\text{VOL}[\text{a$\_$},\text{b$\_$},\text{c$\_$}]\text{:=}}\\
\pmb{12!*(1/4)*(\text{Integrate}[\text{PiecewiseExpand}[\text{dimasVHx}*\text{dimasVx}],\{x,y,z\}\in P[a,b,c]])}\)
\end{doublespace}

Given explicit integer values for $a,b,c$, the program takes around 15 minutes to run and give the numerical answer.

We will see in the next section some enumerative consequences of these calculations, and we will speculate a little bit about why we could have expected a picture as in Corollary~\ref{finalformula}.

\section{The twisted cubics case}\label{TC}

We will now express some consequences of the entire work so far, to the case of $X$ being a twisted cubic, $G=SL_4$ and $H=SL_2$. As seen in Example~\ref{rnc}, $X$ is a homogeneous variety, and by Corollary~\ref{specialhomog} it is special. We will indicate by $L_{a_1,a_2,a_3}$ or $L_{a_1a_2a_3}$ the line bundle corresponding to the weight $a_1\omega_1+a_2\omega_2+a_3\omega_3$, and by $M_{a_1,a_2,a_3}$ or $M_{a_1a_2a_3}$ the G.I.T. quotient obtained with the $H$-linearized line bundle $L_{a_1,a_2,a_3}$.

\subsection{Two stratifications of $\Omega_3$}

We will now describe two stratification of $\Omega_3$, that will make us able to describe explicitly the stable and semistable loci for the various linearizations, and further geometric properties of the quotients $M_{a_1a_2a_2}$.

Let us pick a basis $e_1,e_2,e_3,e_4$ for the vector space $V$ on which $G$ acts on, in such a way $X$ can be parametrized as $[t^3,t^2s,ts^2,s^3]$ in $\PP(V)$. Let us set $e_{ij}=e_i\wedge e_j\in\wedge^2 V$ and $e_{ijk}=e_i\wedge e_j\wedge e_k\in\wedge^3 V$. Elements of $\Omega_3$ can then be represented as (projectivization of) triples of matrices $$(A_1,A_2,A_3) \quad \begin{cases}A_1\in SL(<e_1,e_2,e_3,e_4>)=SL(V)\\A_2\in SL(<e_{12},e_{13},e_{14},e_{23},e_{24},e_{34}>)=SL(\wedge^2V)\\A_3\in SL(<e_{123},e_{124},e_{134},e_{234}>)=SL(\wedge^3V)\end{cases}.$$

We already have the first stratification, that is given by $G$-orbits, that we denoted by $U$, $\Delta_1$, $\Delta_2$, $\Delta_3$ and all their intersections that we denoted by $\Delta_{12},\Delta_{13},\Delta_{23},\Delta_{123}$. For each of these strata, we can inquire who the three matrices $A_1,A_2,A_3$ are, their ranks, and how they relate to each other.

For the matrix $A_1$, let us consider the projective (rational) morphism $\PP(V)\to\PP(V)$; we will indicate by \textit{kernel} of $A_1$ the projectivization of the actual linear kernel (so, rather the locus where the projective morphism is not defined) and by \textit{image} the image of the projective morphism. We will indicate points by $p,q$, lines by $L,M$ and planes by $H,K$. For $A_2$, notice that the projective morphism $\PP(\wedge^2V)\to\PP(\wedge^2V)$ associated will carry the Grassmannian of lines $\GG(1,3)$ into itself, because $A_2$ conserves pure tensors in $\wedge^2V$; we will indicate by \textit{kernel} the intersection of the projectivization of the actual linear kernel and $\GG(1,3)$, and by image the intersection of the projective image with $\GG(1,3)$, expressed as Schubert cycles. For $A_3$, we will indicate by \textit{kernel} the locus where the projective morphism $\PP(\wedge^3V)\to\PP(\wedge^3V)$ is not defined and by \textit{image} the image of the projective morphism again; this will be expressed as dual varieties, because we have $\PP(\wedge^3V)\cong\PP(V^*)$; for a plane $H\subset \PP(V)$, the cycle $H^*$ will be a point in $\PP(\wedge^3V)$, and analogously $L^*$ will be a line and $p^*$ will be a plane. Let us now analyze all different cases.
\begin{center}
\begin{tabular}{ | c | c | c | c | c | c | c | c | c | c | c | c | c | }
  \hline
  & \multicolumn{3}{|c|}{$A_1$}& \multicolumn{3}{|c|}{$A_2$}& \multicolumn{3}{|c|}{$A_3$} & \\
  \hline
  & rk & ker & im & rk & ker & im & rk & ker & im & extra \\
  \hline
  $U$ & 4 & $\emptyset$ & $\PP V$ & 6 & $\emptyset$ & $\PP\wedge^2V$ & 4 & $\emptyset$ & $\PP\wedge^3V$ & \\
  \hline
  $\Delta_1$ & 1 & $H$ & $q$ & 3 & $\Sigma_{1,1}(H)$ & $\Sigma_{2}(q)$ & 3 & $H^*$ & $q^*$ & \\
  \hline
  $\Delta_2$ & 2 & $L$ & $M$ & 1 & $\Sigma_{1}(L)$ & $\Sigma_{2,2}(M)$ & 2 & $L^*$ & $M^*$ & \\
  \hline
  $\Delta_3$ & 3 & $p$ & $K$ & 3 & $\Sigma_{2}(p)$ & $\Sigma_{1,1}(K)$ & 1 & $p^*$ & $K^*$ & \\
  \hline
  $\Delta_{12}$ & 1 & $H$ & $q$ & 1 & $\Sigma_{1}(L)$ & $\Sigma_{2,2}(M)$ & 2 & $L^*$ & $M^*$ & \!\!\mbox{\begin{tabular}{c} $L\subset H,$ \\ $q\in M$\end{tabular}}\!\! \\
  \hline
  $\Delta_{13}$ & 1 & $H$ & $q$ & 2 & \!\!\mbox{\begin{tabular}{c} $\Sigma_{1,1}(H)\cup$ \\ $\Sigma_{2}(q)$\end{tabular}}\!\! & $\Sigma_{2,1}(q,K)$ & 1 & $p^*$ & $K^*$ & \!\!\mbox{\begin{tabular}{c} $p\in H,$ \\ $q\in K$\end{tabular}}\!\! \\
  \hline
  $\Delta_{23}$ & 2 & $L$ & $M$ & 1 & $\Sigma_{1}(L)$ & $\Sigma_{2,2}(M)$ & 1 & $p^*$ & $K^*$ & \!\!\mbox{\begin{tabular}{c} $p\in L,$ \\ $M\subset K$\end{tabular}}\!\! \\
  \hline
  $\Delta_{123}$ & 1 & $H$ & $q$ & 1 & $\Sigma_{1}(L)$ & $\Sigma_{2,2}(M)$ & 1 & $p^*$ & $K^*$ &  \!\!\mbox{\begin{tabular}{c} $p\in L\subset H,$ \\ $q\in M\subset K$\end{tabular}}\!\!  \\
  \hline
\end{tabular}
\end{center}

A further, and quite neat, description of these strata is the following; the strata $\Delta_I$ is the $G$-orbit of the limit as the variables $u_i$ for $i\in I$ of the following family
\begin{equation}\label{neat}
\left[\begin{matrix}1&0&0&0\\0&u_1&0&0\\0&0&u_1u_2&0\\0&0&0&u_1u_2u_3\end{matrix}\right]\in G_a \cong  U\subset\Omega_3.
\end{equation}

We will now describe the second stratification, that is in some sense finer than this, and is related to the action of $H$.

Notice that the maximal torus $\left[\begin{smallmatrix}t&0\\0&t^{-1}\end{smallmatrix}\right]\subset H$ acts on $(A_1,A_2,A_3)$ as
\begin{equation}\label{torusacting}
\left(A_1\cdot \left[\begin{smallmatrix}t^3&0&0&0\\0&t&0&0\\0&0&t^{-1}&0\\0&0&0&t^{-3}\end{smallmatrix}\right],
A_2\cdot \left[\begin{smallmatrix}t^4&0&0&0&0&0\\0&t^2&0&0&0&0\\0&0&1&0&0&0\\0&0&0&1&0&0\\0&0&0&0&t^{-2}&0\\0&0&0&0&0&t^{-4}\end{smallmatrix}\right],
A_3\cdot \left[\begin{smallmatrix}t^3&0&0&0\\0&t&0&0\\0&0&t^{-1}&0\\0&0&0&t^{-3}\end{smallmatrix}\right]\right)
\end{equation}

So, every row of the matrices $A_i$ is multiplied by a certain power of $t$.
We are now ready to give our definition.

\begin{definition}
We will denote by $\Omega_3^{c_1,c_2,c_3}$ the closure of the $H$-orbit of the locus in $\Omega_3$ of triples of matrices such that the lowest nonzero row in $A_i$ gets multiplied by $t^{c_i}$.
\end{definition}

The usefulness of these strata is immediate; we have in fact the next Proposition, that is just a rephrasing of the Hilbert-Mumford criterion for stability and semistability.

\begin{proposition}\label{stablesemi}
Let us consider the $H$-linearization $L_{a_1a_2a_3}$ on $\Omega_3$. Then
$$\Omega_3^{s}(L_{a_1a_2a_3})=\Omega_3\setminus\left(\bigcup_{a_1c_1+a_2c_2+a_3c_3\geq0}\Omega_3^{c_1,c_2,c_3}\right)$$
$$\Omega_3^{ss}(L_{a_1a_2a_3})=\Omega_3\setminus\left(\bigcup_{a_1c_1+a_2c_2+a_3c_3>0}\Omega_3^{c_1,c_2,c_3}\right)$$
\end{proposition}

Let us now show a few examples of these strata.

\begin{example}
$\Omega_3^{-3,-4,-3}$ is just the entire $\Omega_3$, because the set of triples with the last rows of $A_1,A_2,A_3$ being nonzero is dense in $\Omega_3$ (and so will its $H$-orbit).
\end{example}

\begin{example}\label{boundarytc}
Let us now consider $\Omega_3^{-1,-4,-3}$, that is composed by all $H$-translates of elements for which the last row of $A_1$ is zero. The last row of $A_1$ being zero means that the kernel of $A_1$ needs to contain the point $[0,0,0,1]$, that lies on the twisted cubic $X$; any $H$-translate of it, hence, will have the kernel containing \textit{a} point of $X$, because $X$ is homogeneous for the action of $H$. This strata will contain hence the entire $\Delta_1$, because in that case the kernel is a plane $H$ that will of course contain a point of $X$; it will al course also contain the entire $\Delta_{12},\Delta_{23},\Delta_{123}$, but also the subsets of $\Delta_{2},\Delta_{3},\Delta_{23}$ for which the kernel of $A_1$ is either a line $L$ meeting $X$, or a point $p$ on the curve. Analogously, $\Omega_3^{-3,-2,-3}$ will contain the entire $\Delta_2$ but not the entire $\Delta_1$ and $\Delta_3$, and $\Omega_3^{-3,-4,-1}$ will contain the entire $\Delta_3$ but not the entire $\Delta_1$ and $\Delta_2$.
\end{example}

\begin{example}\label{boundarytc2}
Let's look at the other side of the spectrum, at $\Omega_3^{3,4,3}$; notice that this locus will \textit{always} be unstable for any $L$ nef on $\Omega_N$. $\Omega_3^{3,4,3}$ is the $H$-orbit of the set of matrices having only the first row that is nonzero. Any such object, for rank reasons, will have to lie necessarily in $\Delta_{123}$. Having only the first row nonzero means that the three kernels are $[<e_2,e_3,e_4>]$, $\Sigma_{1}([<e_3,e_4>])$ and $[<e_4>]^*$; these are a plane that \textit{osculates} (meeting once with multiplicity 3) $X$ at the point $[e_4]$, the tangent line at $[e_4]$, and $[e_4]$ itself; having $H$ acting on it, the three kernels can only become a triple osculating plane-tangent line-point at \textit{any} point of $X$.
\end{example}

It is actually possible to simplify this stratification a little bit. We will call a stratum $\Omega_3^{c_1,c_2,c_3}$ \textit{effective} if for each $\Omega_3^{c'_1,c'_2,c'_3}$ such that $c'_i\geq c_i$ and $(c_1,c_2,c_3)\neq(c'_1,c'_2,c'_3)$, we have 
$$\Omega_3^{c_1,c_2,c_3}\subsetneq\Omega_3^{c'_1,c'_2,c'_3}.$$

\begin{remark}\label{allstrata}
After a lot of calculations similar to what happened in Remarks~\ref{boundarytc}-\ref{boundarytc2}, it is possible to show that the only effective strata $\Omega_3^{c_1,c_2,c_3}$ are in the following 24 cases for $(c_1,c_2,c_3)$.
$$
\begin{matrix}
(3,4,3) & (3,4,1) & (3,2,3) & (1,4,3) \\
(1,4,1) & (3,2,-1) & (3,0,1) & (1,0,3) \\
(-1,2,3) & (3,0,-1) & (-1,0,3) & (1,-2,1) \\ 
(-1,2,-1) & (1,0,-3) & (-3,0,1) & (1,-2,-3) \\
(-1,0,-3) & (-3,0,-1) & (-3,-2,1) & (-1,-4,-1) \\
(-1,-4,-3) & (-3,-2,-3) & (-3,-4,-1) & (-3,-4,-3)
\end{matrix}
$$
It is possible to give geometric interpretation of all of these strata, as elements of a given boundary component, with the kernels in specific relation with the twisted cubic $X$. There are a few fascinating coincidences regarding this list; first, these elements are in 1-1 correspondence to the elements of the Weyl group $W$ of $G$, and to the Bruhat cells of $G/B$. Furthermore, if we describe them in a geometric way as in Remarks~\ref{boundarytc}-\ref{boundarytc2}, there is some unclear ``duality'' relationship between every nontrivial stratum $\Omega_3^{c_1,c_2,c_3}$ and its ``opposite'' $\Omega_3^{-c_1,-c_2,-c_3}$. More about it will come in the following subsection.
\end{remark}

\begin{remark}
To extend this stratification to the situation of any $G=SL_{N+1}$ and $H$ stabilizer of a homogeneous variety, the stratification would look like
$$\Omega_N^{c_1,\ldots,c_N}$$
where $c_1,\ldots,c_N$ are integral weights of $H$; in particular, $c_i$ would be one of the weights for the representation $V_{\omega_i}$ of $G$ restricted to $H$. The combinatorics of these spaces would get much more complicated, and very likely would contain a lot of information about the geometry of the spaces $M_L$ (possibly also about the volume function); we were not able to find some general statement about it, and that's why we have brought up this stratification only in the case of twisted cubics.
\end{remark}

\subsection{The spaces $M_L$}

Using these new strata, we can prove the following.

\begin{proposition}\label{completetc}
Let $L_{a_1a_2a_3}$ be any ample line bundle on $\Omega_3$. Then:
\begin{itemize}
\item[i)] The general point of each boundary divisor $\Delta_1,\Delta_2,\Delta_3$ lies in $\Omega_3^s(L_{a_1a_2a_3})$; as a consequence, $M_{a_1a_2a_3}$ will have three boundary irreducible components $E_1,E_2,E_3$ of codimension 1.
\item[ii)] If $(a_1,a_2,a_3)$ is general in $\N^3$ (that means, outside of the zero locus of a finite number of linear forms) then $$\Omega_3^s(L_{a_1a_2a_3})=\Omega_3^{ss}(L_{a_1a_2a_3});$$
as a consequence, if $(a_1,a_2,a_3)$ is general $M_{a_1a_2a_3}$ has finite quotient singularities and its Picard group can be identified with a finite index sublattice of $\Lambda(G)$.
\end{itemize} 
\end{proposition}

\begin{proof}
In Example~\ref{boundarytc}, we have seen that the general point of every boundary divisor lies in exactly one of the three strata $\Omega_3^{-1,-4,-3}$,$\Omega_3^{-3,-2,-3}$ and $\Omega_3^{-3,-4,-1}$. Notice that if $c_i\leq c'_i$ for $i=1,2,3$ then we clearly have $$\Omega_3^{c_1,c_2,c_3}\supseteq\Omega_3^{c'_1,c'_2,c'_3}.$$
So, points in the boundary divisor can't lie in any other strata $\Omega_3^{c_1,c_2,c_3}$, and hence they cannot lie in a strata for which $a_1c_1+a_2c_2+a_3c_3\geq0$ for any positive $a_1,a_2,a_3$; this proves the first part of i).
For the first part of ii), there is only a finite number of possibilities for $c_1,c_2,c_3$ (coming from Remark~\ref{allstrata})); for any $a_1,a_2,a_3$ such that $a_1c_1+a_2c_2+a_3c_3\neq0$ for each of these possibilities, then  
$\Omega_3^s(L_{a_1a_2a_3})=\Omega_3^{ss}(L_{a_1a_2a_3})$ from Proposition~\ref{stablesemi}. The second part of i) and ii) comes from Proposition~\ref{4points}, because from Proposition~\ref{specialhomog} we have $X$ is special, and all ample line bundles will be $H$-nef.
\end{proof}

As stated in Remark~\ref{nefcase}, when $L$ is not ample but just nef, it is still possible to understand the models $M_L$, but we some parts of Proposition~\ref{4points} don't quite work anymore. Let us do a couple of example to show what happens.

\begin{example}\label{m100}
Let's consider the line bundle $L_{100}$; in this case, the model $$Proj\bigoplus_{k\geq 0}H^0(\Omega_3,L_{k00})$$
is just isomorphic to $\PP^{15}=\PP(End(V))$, and the only boundary divisor of $\Omega_3$ that appears is just $\Delta_3$. Having $\PP^{15}$ Picard rank 1, we are forced to have the general point of $\Delta_3$ being stable, and the Picard group of $M_{100}$ to be a finite index subgroup of the Picard group of $\PP^{15}$ (because $M_{100}$ is projective). We have also no strictly semistable locus (because of no exponent of $t$ in (\ref{torusacting}) on $A_1$ being zero), so in the end $M_{100}$ will have only finite quotient singularities, have Picard rank 1, and the only boundary divisor $E_3$ (that is the quotient of $\Delta_3$).
\end{example}

\begin{example}\label{m010}
Let us now consider $L_{010}$. The model 
$$Proj\bigoplus_{k\geq 0}H^0(\Omega_3,L_{0k0})$$
is just obtained projecting $\Omega_3$ to the middle component $\PP^{35}=\PP(End(\wedge^2V))$. It has now two boundary components, corresponding to $\Delta_1$ and $\Delta_3$ on $\Omega_3$; in order to see this, notice that forgetting about $A_1$ and $A_3$ the stratum $\Delta_2$ has the same image as $\Delta_{123}$. This space won't very likely be smooth (it is possible to check it by an explicit calculation on tangent spaces at points of $\Delta_{123}$) as it won't most likely be the G.I.T. quotient
$$M_{010}=Proj\bigoplus_{k\geq 0}H^0(\Omega_3,L_{0k0})^H.$$
Notice that there will be some strictly semistable locus, because one of the power $t$ acts by is zero. Because of the fact that $M_{010}$ is projective, the general point of at least one of $\Delta_1$ and $\Delta_3$ has to be stable; by a symmetry argument, the general point of both will be, and we get the two boundary divisors $E_1$ and $E_3$ in $M_{010}$ as well. 
\end{example}

\begin{remark}\label{nefcases}
We can more in general describe how many boundary divisors we have for any $L$ nef, so allowing the $a_i$ to be zero as well. We just need to look at how many of the $\Delta_1,\Delta_2,\Delta_3$ survive in the model 
$$Proj\bigoplus_{k\geq 0}H^0(\Omega_3,L^{\otimes k}).$$
Following this rule, we have
\begin{itemize}
\item $E_1,E_2,E_3$ for each $a_1>0,a_2\geq0,a_3>0$,
\item $E_1,E_3$ for each $a_1a_3=0,a_2>0$,
\item $E_1$ when $a_1=a_2=0$,
\item $E_3$ when $a_2=a_3=0$. 
\end{itemize}
\end{remark}

We can also describe explicitly what line bundles $L$ give rise to different models $M_L$. We have in fact the following.

\begin{proposition}\label{models}
Let the below picture be the intersection of the positive octant in $\Z^3\cong Pic(\Omega_3)$ and the hyperplane with sum of coordinates 1.Then, the model $M_L$ depends only on the position of $L$ in the following chamber decomposition. Furthermore, $(a_1,a_2,a_3)$ satisfies the generality condition of Proposition~\ref{completetc}, (ii) if and only if it belongs to the interior of one of the following chambers.
\begin{center}
\begin{tikzpicture}
[scale=0.8,transform shape]\label{vargitt}
\tikzstyle{circ}=[draw,circle,fill=black,minimum size=4pt,
                            inner sep=0pt]
\draw (0,0) -- (4,6) -- (8,0) -- (0,0);
\draw (4,6) -- (2,0) -- (1.333,2) -- (6.666,2) -- (6,0) -- (4,6);
\draw (0,0) node (1) [circ,label=180:$\left(\begin{array}{c}4\\0\\0\end{array}\right)$]{};
\draw (2,0) node (1) [circ,label=270:$\left(\begin{array}{c}3\\0\\1\end{array}\right)$]{};
\draw (6,0) node (1) [circ,label=270:$\left(\begin{array}{c}1\\0\\3\end{array}\right)$]{};
\draw (8,0) node (1) [circ,label=0:$\left(\begin{array}{c}0\\0\\4\end{array}\right)$]{};
\draw (4,6) node (1) [circ,label=30:$\left(\begin{array}{c}0\\1\\0\end{array}\right)$]{};
\draw (4,6) node (1) [circ,label=150:$\left(\begin{array}{c}0\\3\\0\end{array}\right)$]{};
\draw (1.333,2) node (1) [circ,label=150:$\left(\begin{array}{c}2\\1\\0\end{array}\right)$]{};
\draw (6.666,2) node (1) [circ,label=30:$\left(\begin{array}{c}0\\1\\2\end{array}\right)$]{};
\draw (1,0.2) node (1) [label=90:$\mathfrak{c}_1$]{};
\draw (2.4,0.8) node (1) [label=135:$\mathfrak{c}_2$]{};
\draw (2.5,2.7) node (1) [label=90:$\mathfrak{c}_3$]{};
\draw (4,0.5) node (1) [label=90:$\mathfrak{c}_4$]{};
\draw (4,2.5) node (1) [label=90:$\mathfrak{c}_5$]{};
\draw (7,0.2) node (1) [label=90:$\mathfrak{c}_6$]{};
\draw (5.6,0.8) node (1) [label=45:$\mathfrak{c}_7$]{};
\draw (5.5,2.7) node (1) [label=90:$\mathfrak{c}_8$]{};
\end{tikzpicture}
\end{center}
\end{proposition}

\begin{proof}
This follows from Remark~\ref{allstrata}. The only way a quotient $M_L$ can change is if stable and semistable loci in $\Omega_3$ change; and that happens only whenever a nontrivial stratum moves in or out these loci. For each of these strata (only those with both positive and negative entries, because these are the ones for which the nef cone can achieve both positive and negative values) we get a hyperplane. In particular, the strata are
$$
\begin{matrix}
(3,2,-1) & (-1,2,3) & (3,0,-1) & (-1,0,3) & (1,-2,1) \\ 
(-3,-2,1) & (1,-2,-3) & (-3,0,1) &(1,0,-3) & (-1,2,-1)
\end{matrix}
$$
and (columnwise) they represent the 5 line segments dividing the nef cone in chambers in the picture. The second statement then follows as well.
\end{proof}

\begin{remark}
It is of course not a chance that the picture in Proposition~\ref{models} is the same as in Corollary~\ref{finalformula}. For a $L$ in a chamber $\mathfrak{c}$, the nef cone of $M_L$ will lift up to the chamber $\mathfrak{c}$ itself; so, on every chamber $\mathfrak{c}$ of this decomposition, the volume function has to agree with the self intersection product of $M_L$ for $L$ in $\mathfrak{c}$. So, the volume function will be polynomial in each of these chambers. 
\end{remark}

\subsection{Modular interpretation}

It is quite natural to ask what is the relation of these spaces $M_L$ with the different moduli spaces of twisted cubics that are already known, for instance the Hilbert scheme\footnote{We will call $Hilb_{3m+1}(\PP^3)^\circ$ the irreducible component in the Hilbert scheme $Hilb_{3m+1}(\PP^3)$ that is the closure of smooth curves.} $Hilb_{3m+1}(\PP^3)^\circ$ and the Kontsevich space of stable maps $\overline{\mathcal{M}}_{0,0}(\PP^3,3)$. Both these spaces contain an open set isomoprhic to $G_a/H_a$, and hence there will be birational morphisms from $M_L$ to either of them. We will now describe explicitly (up to codimension 1) the map $M_L\to Hilb_{3m+1}(\PP^3)^\circ$. Notice that the complement of $G_a/H_a$ in $Hilb_{3m+1}(\PP^3)^\circ$ is composed by two irreducible components $E_{\lambda}$ whose general point is a rational nodal plane cubic with an embedded point at the node, and $E_{\omega}$ whose general point is a union of a conic and a line secant to it. Everything we will say will hold in a similar matter for the Kontsevich space $\mathcal{M}_{0,0}(\PP^3,3)$ as well.

Let us consider the family

\begin{center}
\begin{tikzpicture}[description/.style={fill=white,inner sep=2pt}]
\matrix (m) [matrix of math nodes, row sep=1em,
column sep=2.5em, text height=1.5ex, text depth=0.25ex]
{ \Phi & \!\!\!\!\!\!\!\!\!\!\!\!\!\!\!\!\!\!=\{(g,p)\mid g^{-1}p\in X\}\subset G\times \PP^3 \\
G_a &  \\ };
\path[->,font=\scriptsize]
(m-1-1) edge node[auto] {$\pi$} (m-2-1);
\end{tikzpicture}
\end{center} 
that is clearly a projective algebraic subvariety of $G_a\times \PP^3$, and the fiber over every $g\in G_a$ is just the curve $g\cdot X$.
The action of $H$ on $G_a\times \PP^3$ just as the right multiplication on the first component leaves $\Phi$ invariant, because $h^{-1}g^{-1}p\in X\iff g^{-1}p\in X$. So, the quotient 
\begin{center}
\begin{tikzpicture}[description/.style={fill=white,inner sep=2pt}]
\matrix (m) [matrix of math nodes, row sep=1em,
column sep=2.5em, text height=1.5ex, text depth=0.25ex]
{ \Phi/H & \!\!\!\!\!\!\!\!\!\!\!\!\!\!\!\!\!\!\subset G_a/H_a\times \PP^3 \\
G_a/H_a & \!\!\!\!\!\!\!\!\!\!\!\!\!\!\!\!\!\!\!\!\!\!\!\!\!\!\!= G_a/H\\ };
\path[->,font=\scriptsize]
(m-1-1) edge node[auto] {$\pi$} (m-2-1);
\end{tikzpicture}
\end{center} 
will be a family of twisted cubics over $G_a/H_a$. By the modularity of $Hilb_{3m+1}(\PP^3)$, this gives us a map $f_{\pi}:G_a/H_a\to Hilb_{3m+1}(\PP^3)$, that we can extend to a rational map $f_L$ from any of the spaces $M_L$. Now, the undeterminacy locus of this map has to 

It is interesting to find out if the map is defined on the boundary divisors $E_1,E_2,E_3$, and what are their images, in terms of the two boundary divisors $E_{\lambda}$ and $\Delta_{\phi}$ of $Hilb_{3m+1}(\PP^3)^\circ$.

\begin{lemma}\label{divsurv1}
Let $L$ be ample on $\Omega_3$; then $f_L$ sends the general point of $E_3$ to the general point of $E_{\lambda}$, and it collapses the divisors $E_2$ and $E_1$. If $L$ is nef, these statements are still true for the divisors $E_i$ that appear in $M_L$, as in Remark~\ref{nefcases}.
\end{lemma}

\begin{proof}
Notice that if $L$ is ample, then $M_L$ will have three boundary divisors $E_1,E_2,E_3$ that are images of $\Delta_1,\Delta_2,\Delta_3$ from $\Omega_3$. Completing $f_\pi$ to $f_L$ is the same as completing the family $\Phi/H\to G_a/H_a$ as a flat family over $M_L$, that will be as well possible to do only outside of a locus of codimension at least 2. This is the same as completing the family $\Phi\to G_a$ as a family over $\Omega_3$; we can then use the formula \ref{neat} to construct arcs $\{g_t\}_{t\in\C^*}$ in $G_a$ having as limits the general points of $\Delta_1,\Delta_2,\Delta_3$, and find the flat limit $\lim_{t\to 0}(g_t\cdot X)$ of the twisted cubic $X$. For example, and arc in $G_a$ having limit a general point of $\Delta_3$ is a general $G$-translate of the arc $$\left\{\left[\begin{smallmatrix} 1 & 0 & 0 & 0 \\ 0 & 1 & 0 & 0 \\ 0 & 0 & 1 & 0 \\ 0 & 0 & 0 & t \end{smallmatrix}\right]\right\}.$$
This arc correspond to the projection away from a general point onto a general plane of $\PP^3$; the flat limit $\lim_{t\to 0}(g_t\cdot X)$ will then be a planar nodal cubic with an embedded point at the node; this proves that $f_L$ maps $E_3$ generically to $E_{\lambda}$. In the cases of $E_2$ and $E_1$, the arcs are general translates of (respectively)
$$\left\{\left[\begin{smallmatrix} 1 & 0 & 0 & 0 \\ 0 & 1 & 0 & 0 \\ 0 & 0 & t & 0 \\ 0 & 0 & 0 & t \end{smallmatrix}\right]\right\},\quad\left\{\left[\begin{smallmatrix} 1 & 0 & 0 & 0 \\ 0 & t & 0 & 0 \\ 0 & 0 & t & 0 \\ 0 & 0 & 0 & t \end{smallmatrix}\right]\right\};$$
in the first case, the arc represents the projection away from a general line onto a general line; the flat limit of the twisted cubic is then the curve whose ideal is the square of the ideal of a line (we'll call it \textit{triple line}); this shows that the general point of $E_2$ is sent by $f_L$ inside the locus in $Hilb_{3m+1}(\PP^3)^\circ$ of triple lines, that has dimension 4 (entirely contained in $E_{\omega}$, and disjoint from $E_{\lambda}$); this locus is also closed, so it will contain the image of the entire $E_2$ (or at least the part of if where $f_L$ is defined); $E_2$ is then collapsed by that. In the second case, the arc represents the projection away from a general plane onto a general point; the flat limit will then be the union of three not coplanar lines meeting at the same point (still contained in $E_{\omega}$, and whose closure intersect also $E_{\lambda}$). This locus in $Hilb_{3m+1}(\PP^3)^\circ$ has codimension 3, so $E_1$ will be shrunk as well. 
\end{proof}

\begin{remark}
In this way we can also show that $f_L$ cannot be a regular morphism; if it was, it would be surjective, but we don't reach the general point of the other boundary divisor $E_{\omega}$. To make the map regular, we would need to modify $M_L$ along the undeterminacy locus. To reach the general point of $E_{\omega}$, we would only need to modify $M_L$ along the points that have an arc $\{g_t\}$ with them as a limit, whose flat limit $\lim_{t\to 0}(g_t\cdot X)$ is the union of a conic and a line secant to it. Such arcs arise (in $G_a$) as the projection away from a point in $X$; their limits are the triples $(A_1,A_2,A_3)$ such that the kernel of $A_1$ is a point on $X$; in terms of the strata above, these points correspond to $\Omega_3^{-1,0,3}$. These points are in the (semi)stable locus for $L_{a_1a_2a_3}$ only whenever $3a_3-a_1<0$ ($3a_3-a_1\leq0$); in case this doesn't happen, the locus $\Omega_3^{1,0,-3}$ is actually composed by points that are limits of such arcs as well. The geometry of such points is more complicated; they belong to $\Delta_{12}$, and they are such that the kernel of $A_1$ is a plane containing a tangent line to $X$ at a point $p$, and such that the kernel of $A_3$ is the dual of a line containing $p$. This sort of duality  (as well as what happens more precisely when $3a_3-a_1=0$) will be explored later in this subsection.
\end{remark}

These are not the only known moduli spaces for twisted cubics. In \cite{netquadrics}, it is studied the moduli space of twisted cubics obtained considering nets of quadrics (another interpretation of this is as truncated Hilbert scheme $Tr_2Hilb_{3m+1}(\PP^3)^\circ$ at degree 2); in this space, there is only one boundary divisor, that corresponds to $E_{\omega}$ (because the net of quadrics is not able to recognize the cubic on the plane in $E_{\lambda}$). In \cite{vainsblowup} it is proved that $Hilb_{3m+1}(\PP^3)^\circ$ is just a blowup of $Tr_2Hilb_{3m+1}(\PP^3)^\circ$, whose exceptional divisor is of course $E_{\lambda}$ (but it was proved independently at the same time also by the authors of \cite{netquadrics}). A bit of history; using \cite{vainsblowup}, and starting from the Chow ring of $Tr_2Hilb_{3m+1}(\PP^3)^\circ$ obtained in \cite{netquadrics2}, people explicitly found the Chow ring of $Hilb_{3m+1}(\PP^3)^\circ$ using \cite{chow}.
The Chow groups were already calculated though (interpreting cycles as Schubert cycles) in \cite{chowhilb}.

Another branch of history goes all the way back to Schubert, \cite{Kalkul} in 1879, where plenty of enumerative answers for twisted cubics were given, seemingly without the use of a proper moduli spaces. The most impressive of this answers is the number $5\,819\,539\,783\,680$ of twisted cubics tangent to 12 general quadrics (later proved officially in \cite{sketchnumber}). Many tentatives to unwind Schubert's calculation have been done; he referred to 11 ``degenerations'' of a twisted cubic, or rather of triples of curves, consisting of a twisted cubic $C$, the curve $\Gamma\subset\G$ of all its tangent lines (a rational normal quartic curve), and the dual curve $C^*\subset\PP^3$ of all its osculating planes (another twisted cubic). This 11 degenerations of triples were then described explicitly by Alguneid in \cite{alguneid}, as triples of cycles, and by Piene in \cite{piene11} as triples of schemes. In \cite{piene11}, we find this beautiful picture that explains the situation much more than any word. Every triple should be read backwards too.

\vspace{-1cm}

\begin{center}
\includegraphics[,scale=0.5]{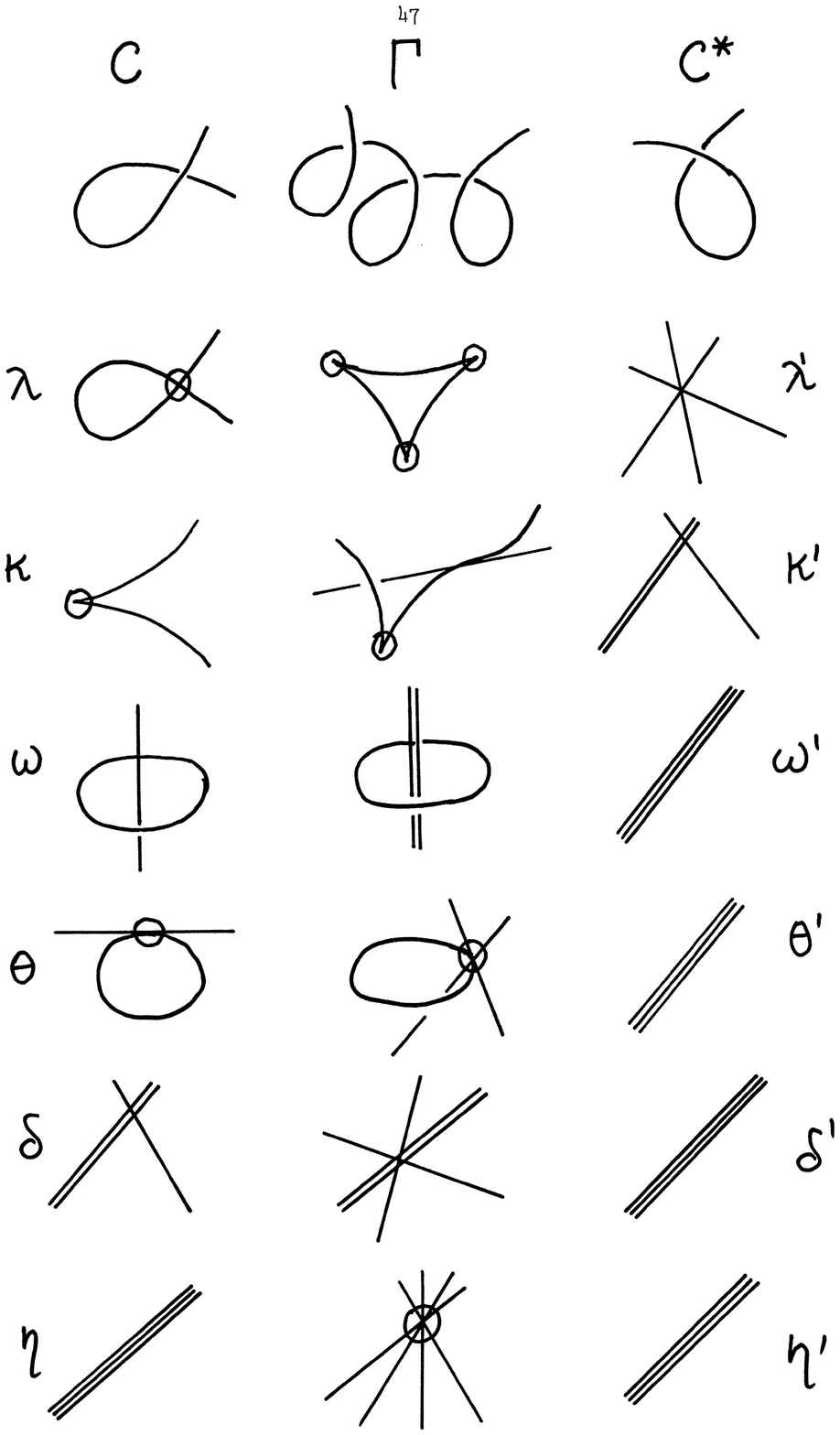}
\end{center}
\vspace{-1cm}
Let us be more precise: what Piene talks about in \cite{piene11} is to consider the multi-Hilbert scheme
$$MHilb\subset Hilb_{3m+1}(\PP^3)^\circ\times Hilb_{4m+1}(\G)^\circ\times Hilb_{3m+1}(\PP^{3*})^\circ$$
obtained considering the closure of triples $(C,\Gamma,C^*)$; this will again be a compactification of the space of twisted cubics, and the boundary outside $G_a/H_a$ will have (at least) divisorial components $E_{\lambda}$, $E_{\kappa}$, $E_{\omega}$, $E_{\theta}$, $E_{\delta}$, $E_{\eta}=E_{\eta'}$, $E_{\delta'}$, $E_{\theta'}$, $E_{\omega'}$, $E_{\kappa'}$, $E_{\lambda'}$. In \cite{pienerandom}, a further degeneration was discovered, still by Piene, and we will call it $E_{\zeta}$; the reason why this remained hidden to Schubert and Alguneid is because, to see $E_{\zeta}$ as a divisor, it is important to consider the nonreduced structure of the three curves; in this case, $C$ and $C*$ will be union of a simple line and a double line meeting (as in $E_{\kappa'}$) and $\Gamma$ will the the union of a single and a triple line. This degeneration is self dual, and so $E_{\zeta'}=E_{\zeta}$. We believe that these are the only irreducible components of the boundary of $MHilb$, even if we don't have a proof of that. Notice that under the projection $MHilb\to Hilb_{3m+1}(\PP^3)^\circ$, the divisors $E_{\lambda}$ and $E_{\omega}$ are sent birationally to their homonymous in $Hilb_{3m+1}(\PP^3)^\circ$, while all the others get contracted to smaller dimensional loci.

We can construct as before birational morphisms $g_L$ from $M_L$ to $MHilb$, starting from a family of triples
$$M\Phi\subset G_a\times\PP^3\times\G\times\PP^{3*}.$$

We can now wonder again where do the three divisors $E_1,E_2,E_3$ go under this morphism.

\begin{lemma}\label{divsurv2}
Let $L$ be ample on $\Omega_3$; then $g_L$ sends the general point of $E_3$ to the general point of $E_{\lambda}$, sends the general point of $E_2$ to the general point of $E_{\eta}$ and sends the general point of $E_1$ to the general point of $E_{\lambda'}$. If $L$ is nef, these statements are still true for the divisors $E_i$ that appear in $M_L$, as in Remark~\ref{nefcases}.
\end{lemma}

\begin{proof}
The proof goes as for the proof of Lemma~\ref{divsurv1}. We need to extend formula \ref{torusacting}, to get that the triples $(A_1,A_2,A_3)$ in a stratum $\Delta_I$ of $\Omega_3$ appear as $G$-translates of the limit, as $u_i\to0$ for all $i\in I$, of the triples
\begin{equation}\label{neatplus}
\left(\left[\begin{smallmatrix}1&0&0&0\\0&u_1&0&0\\0&0&u_1u_2&0\\0&0&0&u_1u_2u_3\end{smallmatrix}\right],\left[\begin{smallmatrix}1&0&0&0&0&0\\0&u_2&0&0&0&0\\0&0&u_2u_3&0&0&0\\0&0&0&u_1u_2&0&0\\0&0&0&0&u_1u_2u_3&0\\0&0&0&0&0&u_1u_2^2u_3\end{smallmatrix}\right],\left[\begin{smallmatrix}1&0&0&0\\0&u_3&0&0\\0&0&u_2u_3&0\\0&0&0&u_1u_2u_3\end{smallmatrix}\right]\right)
\end{equation}
applied to a general triple $(C,Gamma,C^*)$.
The three matrices will act respectively on the three factors $\PP^3,\G,\PP^{3*}$ moving the three curves $C,\Gamma,C^*$. Working case by case, the claim follows.\end{proof}

We have of course still morphisms $g_L$ that are pretty far from being regular (because we don't reach many of the boundary divisors of $MHilb$). It might be interesting to know whether the birational inverse of $g_L$ is regular, because it will be surjective on the general point of all boundary divisors.

There is more; the models $M_L$, in how they vary, contain in fact much more information. What follows comes from \cite{vardh} and \cite{varth}, and it is the main theorem of the theory of Variation of GIT. For every wall crossing in the picture (\ref{vargitt}) from a model $M_L$ with $L$ in a chamber $\mathfrak{c}$ to a model $M_{L'}$ with $L'$ in a chamber $\mathfrak{c}'$, we have a diagram

\begin{center}
\begin{tikzpicture}[description/.style={fill=white,inner sep=2pt}]
\matrix (m) [matrix of math nodes, row sep=1em,
column sep=1em, text height=1.5ex, text depth=0.25ex]
{ & \wt{M}_{\mathfrak{c},\mathfrak{c}'}  & \\
M_L &  & M_{L}' \\
& M_{L''} &  \\ };
\path[->,font=\scriptsize]
(m-1-2) edge node[auto] {} (m-2-1)
(m-1-2) edge node[auto] {} (m-2-3)
(m-2-1) edge node[auto] {} (m-3-2)
(m-2-3) edge node[auto] {} (m-3-2);
\end{tikzpicture}
\end{center} 

where $M_{L''}$ is the model obtained for a line bundle $L''$ lying on the wall $\mathfrak{c}\cap\mathfrak{c}'$ separating the two chambers, and $\wt{M}_{\mathfrak{c},\mathfrak{c}'}$ is the fibered product of the two morphisms. All of the maps are birational morphisms, and the locus where these map are not isomprhism is a square

\begin{center}
\begin{tikzpicture}[description/.style={fill=white,inner sep=2pt}]
\matrix (m) [matrix of math nodes, row sep=1em,
column sep=1em, text height=1.5ex, text depth=0.25ex]
{ & \wt{E}_{\mathfrak{c},\mathfrak{c}'}  & \\
E_L &  & E_{L}' \\
& E_{L''} &  \\ };
\path[->,font=\scriptsize]
(m-1-2) edge node[auto] {} (m-2-1)
(m-1-2) edge node[auto] {} (m-2-3)
(m-2-1) edge node[auto] {} (m-3-2)
(m-2-3) edge node[auto] {} (m-3-2);
\end{tikzpicture}
\end{center} 

whose maps are locally trivial fibrations with fiber weighted projective spaces, and $\wt{E}_{\mathfrak{c},\mathfrak{c}'}$ has codimension 1 in $\wt{M}_{\mathfrak{c},\mathfrak{c}'}$. We can now ask where these divisors $\wt{E}_{\mathfrak{c},\mathfrak{c}'}$ are sent into the multiHilbert scheme $MHilb$, as before. We have the following conjecture.

\begin{conjecture}
We have the following correspondences, between divisors $\wt{E}_{\mathfrak{c},\mathfrak{c}'}$ and boundary divisors in $MHilb$ (in the sense that that's where the general point is send to the map obtained extended the above one on $G_a/H_a$).
\begin{itemize}
\item $\wt{E}_{\mathfrak{c}_1,\mathfrak{c}_2}\to E_{\kappa}$
\item $\wt{E}_{\mathfrak{c}_6,\mathfrak{c}_7}\to E_{\kappa'}$
\item $\wt{E}_{\mathfrak{c}_2,\mathfrak{c}_4}=\wt{E}_{\mathfrak{c}_3,\mathfrak{c}_5}\to E_{\omega}$
\item $\wt{E}_{\mathfrak{c}_4,\mathfrak{c}_7}=\wt{E}_{\mathfrak{c}_5,\mathfrak{c}_8}\to E_{\omega'}$
\item $\wt{E}_{\mathfrak{c}_2,\mathfrak{c}_3}=\wt{E}_{\mathfrak{c}_4,\mathfrak{c}_5}=\wt{E}_{\mathfrak{c}_7,\mathfrak{c}_8}\to E_{\zeta}$
\end{itemize}
\end{conjecture}

We leave this as a conjecture, because we didn't carry out all the details to prove it; a proof would require an explicit description of the loci $\Omega_3^{c_1,c_2,c_3}$, and description of the limits of the triples $(C,\Gamma,C^*)$ for a family in $G_a$ approaching a general element of $\Omega_3^{c_1,c_2,c_3}$.
 
So, in this way we can see a few other of Schubert's degenerations, and the extra one from \cite{pienerandom} as well. Hence, we might be entitled to say that (at least a few of) Schubert's divisors are ``natural'' in some sort of way. It is still unclear if the remaining Schubert's divisors $\theta$,$\theta'$,$\delta$,$\delta'$ are hidden somewhere in this picture in any other way. For instance, they might occur at the two points where two walls meet, in some more complex variation of GIT statement. Or, we can notice how the strata are related to such boundary divisors, and hope that the correspondence continues in some way.
$$\begin{array}{| c | c | c | c |}\hline
(3,4,3) & (3,4,1) & (3,2,3) & (1,4,3)  \\ 
(-3,-4,-3) & (-3,-4,-1) & (-3,-2,-3) & (-1,-4,-3)  \\ 
G_a/H_a & E_{\lambda} & E_{\eta} & E_{\lambda'} \\ \hline
(-1,2,3) & (3,2,-1) & (-1,0,3) &  (3,0,-1) \\ 
(1,-2,-3) & (-3,-2,1) & (1,0,-3)& (-3,0,1)  \\
E_{\kappa} & E_{\kappa'} & E_{\omega} & E_{\omega'} \\ \hline
(1,-2,1) & (1,4,1) & (1,0,3) & (3,0,1) \\
(-1,2,-1) & (-1,-4,-1) & (-1,0,-3) & (-3,0,-1) \\
E_{\zeta} & ? & ? & ? \\ \hline
\end{array}$$

\begin{remark}\label{chowquotient}
There is a notion of Chow quotient (in \cite{chowquot}, defined in a specific case), that is in some sense the inverse limit of all the different models, so it should include all these divisors. We did not find in the literature any precise definition or study about this object in the case $H$ is not a torus. Such a variety would have 8 (or more) boundary components, and possibly a regular morphism from $MHilb$, that ``sees'' all these components. From what we have seen, it is possible that the volume function could be directly related to the intersection theory of this Chow quotient.
\end{remark}

\begin{remark}
Another geometric feature of $M_L$ that we can analyze is the amount of $G$-orbits that we have. The Hilbert scheme $Hilb_{3m+1}(\PP^3)^\circ$ is known to have only a finite number of $G$-orbits (cf. \cite{alggeo}), and the Kontsevich space $\overline{\mathcal{M}}_{0,0}(\PP^3,3)$ is known to have a one parameter family of $G$-orbits. We will show that $M_L$ has in fact a three dimensional family of $G$-orbits (that is the maximum allowed, because the complexity of $G/H$ is three, and because of Theorem 5.7 of \cite{timashev_book}). The subvariety $\Delta_{123}$ of $\Omega_3$ is isomorphic to the product $G/B\times B\backslash G$, with $H$ acting only on the second component; if the general point is $L$-stable, then its image in $M_L$ will be isomorphic to $G/B\times (B\backslash G)/\!\!/_{\!L}H$, with $G$ acting on the first factor only; $(B\backslash G)/\!\!/_{\!L}H$ will have dimension three, and this will give a 3-dimensional family of $G$-orbits. To see that the general point of $\Delta_{123}$ is stable, after some calculation it is possible to show that the general point of $\Delta_{123}$ belongs to the intersection
$$\Omega_3^{-3,-4,-1}\cap\Omega_3^{-3,-2,-3}\cap\Omega_3^{-1,-4,-3}$$
and to no other such stratum, and hence it will always be stable for any $L$ nef. These will also be the only closed $G$-orbits that $M_L$ will have. In case $L$ is just nef, a similar argument proves the same; the only difference is that those closed $G$-orbits will be of type $G/P$ for a different parabolic group.
\end{remark}

\subsection{The volume function and enumerative results}

We were not able to prove directly any enumerative result about twisted cubics; there is evidence though that the volume function is the right tool to use, at least in a few cases.

\subsubsection{The case (4,0,0): tangency to 12 planes}

The first problem we would like to solve is to find the number of twisted cubics that are tangent to 12 planes, that we know to be 56960 (cf. \cite{sketchnumber} and \cite{pandha}). For a plane $H$, we will denote by $D_H$ the divisor of $G_a/H_a$ consisting of all twisted cubics tangent to $H$. We can obtain it in the following way: consider the ($G$-equivariant) map
$$G_a/H_a\xrightarrow{f} \PP^{34}=\PP(H^0(\OO_{\PP^{3*}}(4)))=\PP(V_{004})$$
obtained sending a twisted cubic to the polynomial defining the set of all its tangent planes in $\PP^{3*}$, that is a degree 4 hypersurface. Any divisor $D_H$ can be obtained as the pullback of a hyperplane from $f$. Let us now pick a $M_L$ for $L$ general (in the interior of a chamber of \ref{models}). Let us extend $f$ to a morphism $\tilde{f}$ the entire $M_L$ (there could be an undeterminacy locus of codimension 2 or more), and pullback $\OO_{\PP(V_{004})}(1)$ to a line bundle $L_{a_1a_2a_3}$;  the sections we pullback through this map will be a subspace of $H^0(M_L,L_{a_1a_2a_3})$ isomorphic to $V_{400}$ as a representation of $G$. Notice that these sections don't vanish on any of the boundary components (because their vanishing loci are closures of loci in $G_a/H_a$). The only possibility then is for this line bundle to be $L_{400}$, because of Property (iv) of $\Omega_3$ and of $M_L$. Notice that we have
$$H^0(M_L,L_{400})\cong V_{400}\oplus V_{020}\oplus V_{000}$$
using formula (\ref{sectionsML}) (and a few applications of Lemma~\ref{singlecases}). From Remark~\ref{remarkvanishing}, the sections in $V_{400}$ are those coming from $\PP(V_{004})$, the sections in $V_{020}$ will vanish twice on $E_1$, and the section in $V_{000}$ vanishes only on the boundary, once along $E_3$, twice along $E_2$, and three times along $E_1$. Now, when we evaluate the volume $vol(D_H)$, using Theorem~\ref{main} with $\lambda=4\omega_1$, it gives the answer we know is right, 56960.

There are though three major problems.

\begin{itemize}
\item[i)] We are hoping to use the fact that $vol(D_H)=D_H^{[12]}$, that as we have seen is true only asymptotically. Going a little bit deeper, this statement is true as soon as the base locus $B(D_H)$ is equal to the stable base locus $\bf{B}\!\!$ $(D_H)$. This is not immediate; in fact, we have
\begin{align*}
H^0(M_L,L_{400})&\cong V_{400}\oplus V_{020}\oplus V_{000}\\
H^0(M_L,L_{800})&\cong V_{800}\oplus V_{420}^{\oplus 2}\oplus V_{040}^{\oplus 2}\oplus V_{311}\oplus V_{121}\oplus V_{202}\oplus V_{400}\oplus V_{020}\oplus V_{000} \\
H^0(M_L,L_{1200})&\cong V_{1200}\oplus V_{820}^{\oplus 2}\oplus V_{630}\oplus V_{440}^{\oplus 3}\oplus V_{060}^{\oplus 3} \oplus V_{711} \oplus V_{521}^{\oplus 2} \oplus V_{331}^{\oplus 3}\oplus \\
&\oplus V_{141}\oplus V_{602}^{\oplus 2}\oplus V_{412}\oplus V_{222}^{\oplus 3}\oplus V_{303} \oplus V_{113}\oplus V_{004}\oplus V_{800}\oplus  \\
&\oplus V_{420}^{\oplus 2}\oplus V_{040}^{\oplus 2}\oplus V_{311}\oplus V_{121}\oplus V_{202}\oplus V_{400}\oplus V_{020}\oplus V_{000}
\end{align*}
obtained using formula (\ref{sectionsML}) (and a few applications of Lemma~\ref{singlecases}). After a bit of computations on these, it is possible to see that there are generators of the ring
\begin{equation}\label{ring400}
\bigoplus_{k\geq 0} H^0(M_L,L_{4k,0,0})
\end{equation}
in degree 2 and 3, that could narrow the base locus for $k>1$.
\item[ii)] As we have seen just above, we have $H^0(M_L,L_{400})\cong V_{400}\oplus V_{020}\oplus V_{000}$; this means that the divisors $D_H$ are not really general in their linear series, because they all lie in the $V_{400}$ component. The other two components $V_{020}$ and $V_{000}$ vanish along entire boundary components though, so we believe it could be possible to prove that the base locus of the divisors of type $D_H$ is the same as the entire linear system $H^0(M_L,L_{400})$.
\item[iii)] We are using  the pullback of $\OO_{\PP(V_{004})}(1)$ as a line bundle on $M_L$, while in fact this does not necessarily hold true (and we believe it does not). The problem is, when we extend the line bundle to the undeterminacy locus of $\tilde{f}$ we could (and will) reach some orbifold singularities. We might be able to extend $L_{400}$ only after taking a suitable tensor power of it (our guess is that we need to take the sixth power of it, and it will again be related to the generators in degree 2 and 3 of (\ref{ring400})). So, we cannot really talk about linear series, sections, or even intersection number. To overcome this issue, it would be necessary to appeal to some theory of $\Q$-Cartier divisor.
\end{itemize}
Despite all these problems, the volume calculation does give the answer we were expecting from the literature. So, we believe a statement about the volume function giving enumerative answer could be proved.  Before stating a conjecture, we will analyze now another example.

\subsubsection{The case (0,3,0): meeting 12 lines}

Another very natural question to ask is the number of twisted cubics that meet 12 lines, that we know to be 80160 (cf.  \cite{sketchnumber}, \cite{vakil}, \cite{pandha}). Considering the equation defining the Chow variety of a twisted cubics, we get a map
$$G_a/H_a\xrightarrow{f} \PP^{49}\cong \PP(H^0(\G,\OO(3)))\cong \PP(V_{030})$$
and we can get the divisors $D_L$ (of twisted cubics meeting the line $L$) as pullback of hyperplane sections from $f$. As before, we can extend $f$ to a general $M_L$; divisors $D_L$ will come then from sections of the (supposedly) line bundle $L_{030}$, and we can apply the volume function to $\lambda=3\omega_2$. The answer we get is not the expected one though, it is 1146960. We are then in a situation where the volume function does not give the right answer. Let us show how we can fix this.

The three issues that we had in the previous case still stand entirely, with two small differences (one good and one bad) in ii); the (supposedly) space of sections of the (supposedly) line bundle $L_{030}$ is the following:
$$V_{030}^{\oplus 2}\oplus V_{010}.$$
From Remark~\ref{remarkvanishing}, the sections in $V_{010}$ will vanish now on all boundary components (on $E_1$ and $E_3$ with multiplicity one, on $E_2$ with multiplicity two). The good difference is that we can now prove directly that the base locus of $V_{030}^{\oplus 2}$ is the same as the base locus of the entire space: in fact, the base locus cannot change inside the boundary, because sections of $V_{010}$ vanish there, and it cannot change on $G_a/H_a$, because we have a transitive $G$-action and the base locus has to be $G$-invariant. The bad difference is that the sections $D_L$ are not general inside $V_{030}^{\oplus}$. If we imagine $V_{030}^{\oplus}$ as the vector space $V_{030}\otimes \C^2$ where $G$ acts on the first component, all divisors $D_L$ are contained in a subspace $V_{030}\otimes v$ for a fixed vector $v$, and in this subspace the base locus does increase; this is why, in our opinion, the volume function does not give the right answer.

We can consider though the dual condition, that is the divisors $D_{L}'$ of twisted cubics such that the dual twisted cubics meets a line (or equivalently, such that one osculating plane contains a given line). This divisors will still be pullbacks from a map to $\PP(V_{030})$, but a different one now! They will end up corresponding to a space of sections $V_{030}\otimes v'$ for $v'$ linearly equivalent to $v$. To take 12 general sections in $V_{030}^{\oplus 2}$, we could then take 6 sections in $V_{030}\otimes v$ (hence, divisors of type $D_L$) and $6$ sections in $V_{030}\otimes v'$ (hence, divisors of type $D_L'$). As far as we know, the only known answer to this question is in Schubert's book \cite{Kalkul} (top of page 179), and it is exactly 1146960 as the volume function predicted. There, he also claims that taking 5 divisors of one type and 7 of the other is still a situation that is ``general enough'', and gives the same answer. We can state now a conjecture about it that could work.

\begin{conjecture}\label{conjecturre}
Given a dominant weight $\lambda$, the volume function gives the number of intersection of 12 general (Weil) divisors coming from the (pseudo) sections
$$V_{\lambda}\otimes (V_{\lambda}^*)^H$$
away from their base locus.
\end{conjecture}

We don't know of counterexamples for this conjecture. The main obstacles for the proof of such a conjecture are basically the three issues i)-iii) seen above. We believe this conjecture could also give, after getting a formula as in Corollary~\ref{finalformula} for $SL_3$-invariants in representations of $SL_6$, the number of Veronese surfaces in $\PP^5$ that are tangent to 27 planes, finding the volume of a divisor in the class $3\omega_1$ as a dominant weight for $SL_6$. As far as we know, this number has not been found yet.

To find the number of twisted cubics tangent to 12 quadric hypersurfaces, we believe we should consider the (pseudo) line bundle $L_{860}=L_{400}^{\otimes 2}\otimes L_{030}^{\otimes 2}$; the volume function evaluated there is the number 28744287411306496/2187. We will very likely have some even more complication than i), ii), iii) above. It is also possible that this could give a counterexample for Conjecture~\ref{conjecturre}.

\subsection{Further speculations}

Schubert gives a lot of formulas relating divisors of the type as we have just seen, such as $D_H,D_L,D_L'$ (that he calls respectively $\rho,\nu,\nu'$), and the boundary divisors; he claims that
\begin{align*}
\rho&=\delta+\xi+\omega+3\theta+2\delta+2\nu+2\delta'+3\theta'+\omega'+\xi'+\delta'\\
\nu&=\tfrac{3}{2}\delta+\tfrac{3}{2}\xi+\tfrac{1}{2}\omega+\tfrac{5}{2}\theta+2\delta+3\nu+3\delta'+\tfrac{9}{2}\theta'+\tfrac{3}{2}\omega'+\tfrac{3}{2}\xi'+\tfrac{3}{2}\delta'\\
\nu'&=\tfrac{3}{2}\delta+\tfrac{3}{2}\xi+\tfrac{3}{2}\omega+\tfrac{9}{2}\theta+3\delta+3\nu+2\delta'+\tfrac{5}{2}\theta'+\tfrac{1}{2}\omega'+\tfrac{3}{2}\xi'+\tfrac{3}{2}\delta'
\end{align*}
First, as a sanity check, we have again the dependence relations we have in $M_L$, if we remove all boundary divisors besides $\delta,\eta$ and $\delta'$, and substitute them respectively with $E_3,E_2,E_1$. Then, using the formulas above, we can see that the difference between $\nu$ and $\nu'$ is a combination of the divisors $\omega,\theta, \delta,\delta',\theta',\omega'$; so, they would be different divisor classes in a compactification of $G_a/H_a$ including in its image in $MHilb$ the general point of any of those boundary components. This opens two different further directions.

\subsubsection{More on the volume function}

The volume function on a G.I.T. quotient contains much more information than expected.
Let us consider again a wall-crossing situation from a chamber $\mathfrak{c}$  (with $L$ in its interior) to a chamber $\mathfrak{c}'$ (with $L'$ in its interior) , and the diagram

\begin{center}
\begin{tikzpicture}[description/.style={fill=white,inner sep=2pt}]
\matrix (m) [matrix of math nodes, row sep=1em,
column sep=1em, text height=1.5ex, text depth=0.25ex]
{ & \wt{M}_{\mathfrak{c},\mathfrak{c}'}  & \\
M_L &  & M_{L}' \\
& M_{L''} &  \\ };
\path[->,font=\scriptsize]
(m-1-2) edge node[auto] {} (m-2-1)
(m-1-2) edge node[auto] {} (m-2-3)
(m-2-1) edge node[auto] {} (m-3-2)
(m-2-3) edge node[auto] {} (m-3-2);
\end{tikzpicture}
\end{center} 

Let us denote as before $\wt{E}_{\mathfrak{c},\mathfrak{c}'}$ the exceptional divisor in $\wt{M}_{\mathfrak{c},\mathfrak{c}'}$. We can identify the $Pic(\wt{M}_{\mathfrak{c},\mathfrak{c}'})_{\Q}$ with the $Pic(M_L)_\Q\oplus\Q\cdot\wt{E}_{\mathfrak{c},\mathfrak{c}'}$.
For a line bundle $L_0$ in the chamber $\mathfrak{c}$, the volume is the self intersection number of $L_0$ on $M_L$, and or a line bundle $L_0'$ in the chamber $\mathfrak{c}'$, the volume is the self intersection number of $L_0'$ on $M_L'$. What happens now on $\wt{M}_{\mathfrak{c},\mathfrak{c}'}$? Every line bundle $L_0$ in $\mathfrak{c}$ lifts through the map in our square to a line bundle $L_0+\alpha(L_0)\wt{E}_{\mathfrak{c},\mathfrak{c}'}$, where $\alpha$ is a linear form. The volume of $L_0$, then, will be the top intersection number of $L_0+\alpha(L_0)\wt{E}_{\mathfrak{c},\mathfrak{c}'}$ in $\wt{M}_{\mathfrak{c},\mathfrak{c}'}$. Analogously, we have a linear form $\beta$ such that the volume of a line bundle $L_0'$ in $\mathfrak{c}'$ is the self intersection number of $L_0'+\beta(L_0')\wt{E}_{\mathfrak{c},\mathfrak{c}'}$ in $\wt{M}_{\mathfrak{c},\mathfrak{c}'}$. Now, the two forms $\alpha$ and $\beta$ are different, and they account for the different expressions of the volume function in the two chambers. Looking at how the volume function changes chamber by chamber, then, we could be able to find intersection products involving the divisor $\wt{E}_{\mathfrak{c},\mathfrak{c}'}$, and find intersection products of $D_L$ and $D_L'$ as different divisor classes. More in general, whenever we have $dim(V_{\lambda}^*)^H\geq 2$, and we want to find the intersection of general $G$-translates of a specific divisor $D$, that hence will be contained in a subspace
$$V_\lambda\otimes v\subset V_\lambda\otimes (V_\lambda^*)^H$$
we could be able to work on a more refined compactification of $G_a/H_a$ that would make us able to ``isolate'' the copy of $V_\lambda$ that we care about, and find the enumerative answer we look for. Collecting everything together, it could help us giving an expressio for the self intersection of divisors in the Chow quotient that we talked about in Remark~\ref{chowquotient}.

Another use of the volume could be to find mixed intersection products of divisors (just interpolating the right amount of value of the self intersection - but taking care of the ``mixed'' base loci that can occur). It would be also interesting to find some way to calculate intersection products with higher codimensional cycles as well. We have in fact (for $M_L$ general; cf. \cite{chow_vistoli}) $$rk(A^2(M_L))>rk(Sym^2(A^1(M_L))),$$
so there are codimensional 2 cycles that don't appear as twofold product of divisors; some of them are the most central ones in enumerative geometry, such as all twisted cubics through a point, or al twisted cubics bisecant to a line.

\subsubsection{Embeddings and valuations}

In \cite{lunavust}, a theory to classify all equivariant embeddings of an homogeneous space $G_a/H_a$ is set. The main algebraic object that drives the theory is the $\Lambda_G$-graded algebra
$$R=\C[G/H]^U\cong\C[U\backslash G/H]\cong\C[G/U]^H\cong\bigoplus_{\lambda\in\Lambda_{G_a}}V_\lambda^H,$$
that is the central object in Section~\ref{invariant_theory}, and whose multi-Hilbert polynomial was named $\Xi_H^G$ and was found in Theorem~\ref{mess}. The other ingredient in the theory are the \textbf{valuations} on $R$, defined as maps $\nu:R\setminus 0\to \Q$ such that
$$\nu(f)\in \N \quad \forall f\in \C[G_a/H_a]\subset \C[G/H]$$
$$\nu(f+g)\leq max(\nu(f),\nu(g))$$
$$\nu(fg)=\nu(f)\nu(g)$$
$$\nu(f)=0 \iff f\in V_0$$
Such an onbject can be seen as the order of pole of functions on $G_a/H_a$ at a divisor at the boundary, when extending these functions to rational functions on compactifications of $G_a/H_a$; valuations like these correspond in fact to all possible boundary divisors that can appear compactifying $G_a/H_a$ equivariantly.

The first question we can ask is about which evaluations we have that are constant on the entire pieces $V_\lambda^H$; these are called \textbf{central} valuations, and after a bit of calculation they can be obtained as positive linear combinations of the following three
$$v_1:R\setminus 0\to \N \quad v_1(V_{a_1,a_2,a_3})=\frac{a_1+2a_2+3a_3}{4}$$
$$v_2:R\setminus 0\to \N \quad v_2(V_{a_1,a_2,a_3})=\frac{a_1+2a_2+a_3}{2}$$
$$v_3:R\setminus 0\to \N \quad v_3(V_{a_1,a_2,a_3})=\frac{3a_1+2a_2+a_3}{4}$$
Not by chance, these three valuations correspond exactly to the three boundary divisors $E_1,E_2,E_3$ we have in (almost)  all our spaces $M_L$. 

The problem of finding \textit{all} such evaluations (not just the central ones) is much trickier. The answer could possibly not even be discrete. A good starting point would be to express a set of generators of $R$, to then specify where they are mapped by $\nu$. We were not able to do that either though; looking at specific values of the multiHilbert function $\Xi_H^G$, we were just able to find an incomplete list of weights $\lambda$ where we a generator must exist.

$$\begin{array}{c c c c c c}
(4,0,0) & (3,0,1) & (2,0,2) & (3,0,3) & (1,0,3) & (0,0,4) \\ 
(4,2,0) & (6,3,0) & (0,1,0) & (0,3,0) & (0,2,4) & (0,3,6) \\
(3,2,1) & (1,2,1) & (2,1,2) & (2,2,2) & (3,3,3) & (1,2,3) \\ 
\end{array}$$

We do know that there are only a finite number of generators, because of Hilbert's theorem on invariants applied to $\C[G/U]^H$. Given that any of these generators gives a map $G_a/H_a\to\PP(V_{a_1,a_2,a_3})$, an interesting problem would be to characterize such maps geometrically.

\bibliography{main}{}
\bibliographystyle{plain}

\singlespacing

\end{document}